\documentclass[11pt]{amsart}
\usepackage{amsmath,a4wide,scalerel}
\usepackage{mathabx}
\usepackage{xcolor,graphicx}
\usepackage{mathrsfs}
\usepackage{pgfplots}
\usepackage{subcaption}
\usepackage[normalem]{ulem}
\usepackage{float}
\usepackage{ctable} % for \specialrule command
\usepackage{bm}
\usepackage{hyperref}
\newtheorem{theorem}{Theorem}

\newtheorem{proposition}[theorem]{Proposition}
\newtheorem{remark}[theorem]{Remark}

\newcommand{\Z}{{\mathbb Z}}
\newcommand{\R}{{\mathbb R}}
\newcommand{\N}{{\mathbb N}}
\newcommand{\E}{{\mathbb E}}

\newcommand{\e}{\mathrm{e}}

\newcommand{\T}{\mathbb{T}}
\newcommand{\de}{{\rm det}}
\newcommand{\ad}{{\rm add}}
\newcommand{\mI}{{\rm mI}}
\newcommand{\mS}{{\rm mS}}
\newcommand{\tr}{{\rm tr}}

\usepackage{color}
\usepackage{dsfont}

\title[Splitting for stochastic linear Vlasov equations]{Splitting integrators for linear Vlasov equations with stochastic perturbations}

\date{\today}

\author{Charles-Edouard Br\'ehier}
			\address{Universite de Pau et des Pays de l'Adour, E2S UPPA, CNRS, LMAP, Pau, France}
			\email{charles-edouard.brehier@univ-pau.fr}
\author{David Cohen}
              \address{Department of Mathematical Sciences,
              Chalmers University of Technology and University of Gothenburg, 41296~Gothenburg, Sweden}
              \email{\tt david.cohen@chalmers.se}

\begin{document}

\begin{abstract}
We consider a class of linear Vlasov partial differential equations driven by Wiener noise. Different types of stochastic perturbations are treated: additive noise, multiplicative It\^o and Stratonovich noise, and transport noise. We propose to employ splitting integrators for the temporal discretization of these stochastic partial differential equations.
These integrators are designed in order to preserve qualitative properties of the exact solutions depending on the stochastic perturbation, such as preservation of norms or positivity of the solutions. We provide numerical experiments in order to illustrate the properties of the proposed integrators and investigate
mean-square rates of convergence.
\end{abstract}

\maketitle

{\small\noindent


\section{Introduction}

In this article, we are interested in stochastic perturbations of the linear Vlasov equation
\[
\left\lbrace
\begin{aligned}
&\partial_tf(t,x,v)+v\cdot\nabla_xf(t,x,v)+E(x)\cdot\nabla_v f(t,x,v)=0,\\%~,\quad t\ge 0, x\in\T^d, v\in \R^d.
&f(0,x,v)=f_0(x,v),%~,\quad x\in \T^d, v\in\R^d,
\end{aligned}
\right.
\]
see the next sections for the precise setting.
Our objective is to identify some qualitative properties of the solutions and to propose temporal discretization schemes which are able to preserve those properties.

The Vlasov equation has been introduced in the literature in the middle of the 20th century~\cite{v61}, and has become a fundamental tool for the mathematical description of (collisionless) plasma in astrophysics and in plasma physics. It is common to interpret $f(t,x,v)$ as the density of particles having position $x$ (assumed to take values in a torus) and velocity $v$ at time $t$, starting from an initial configuration $f_0(x,v)$, and where the particles are transported by the ordinary differential equation
\[
\left\lbrace
\begin{aligned}
&\dot{x}_t=v_t,\\
&\dot{v}_t=E(x_t).
\end{aligned}
\right.
\]
The linear model above is a simple version, where the vector field $E$ is imposed and is independent of time. In the Vlasov--Maxwell or the Vlasov--Poisson equations, for example, $E$ is not fixed and depends on the solution $f$, see for instance~\cite{bertrand05,CHEUNG1973201,MR2549372}. We refer to~\cite{MR1942465} for a review of kinetic models which take into account collisional effects, which may be linear (for instance Vlasov--Fokker--Planck, Bhatnagar--Gross--Krook equations) or nonlinear (Boltzmann equation) and to the monograph~\cite{MR1379589} for the analysis of kinetic partial differential equations (PDEs), including the Vlasov equation. Note that kinetic PDEs have also been used as popular models in mathematical biology in the recent years.

Let us mention the main qualitative properties satisfied by the linear Vlasov equations, we refer to Section~\ref{sec:deterministic} for details. If the initial value is nonnegative, then at all times the solution is also nonnegative. In addition, all $L^p$ norms are preserved.

Stochastic versions of Vlasov kinetic equations have been considered in some recent works, for instance to model random injection or removal of particles in the system or to model
the influence of stochastic (space-time) perturbations of the external vector field $E$. In~\cite{MR3661662} the authors prove regularization by noise results for a class of linear Vlasov equations with transport noise, and similar results are proved for nonlinear Vlasov--Poisson(--Fokker--Planck) equations in~\cite{bedrossian2022vlasovpoisson}. See also the article \cite{MR3251910}, where it is shown that introducing a stochastic (space-time) perturbation of the vector field $E$ can prevent collapse in Vlasov--Poisson systems.
Variational integrators are tested for some stochastic versions of the Vlasov equation in the article~\cite{MR4246876}. Finally, numerical methods are applied to some linear collisional kinetic equations with stochastic perturbation in the diffusion limit in the article~\cite{MR3975156}. We refer to~\cite{MR3202241} for a review of numerical methods applied to kinetic PDEs. The article~\cite{MR0731212} promotes the application of particle methods to approximate solutions of Vlasov equations. See also~\cite{MR1977366} for a description of various numerical approaches.

Our motivation in this work is driven by the perspective of geometric numerical integration and the desire to illustrate how the qualitative behavior of the solutions and of well-chosen integrators are modified under stochastic perturbations of various types. In addition, in this work we do not provide rigorous convergence analysis  for the proposed integrators. To the best of our knowledge, the numerical schemes constructed below have not been studied previously in the literature.

The numerical schemes considered in this article combine two techniques which have been extensively studied for the discretization of the Vlasov equation: the semi-Lagrangian approach~\cite{MR1672731,MR2051070,MR4099398,MR2586230,MR2353945,
MR3033035} and the splitting technique~\cite{MR3606462,MR3151388,MR3294671,MR2519600}. The lists of references above are not exhaustive. The semi-Lagrangian approach exploits characteristic curves to express the numerical solution at time $t_{n+1}$ at grid points in terms of the numerical solution at time $t_n$, with the application of an interpolation procedure. The splitting approach consists in decomposing on each time interval the dynamics into subsystems which can be solved exactly. We briefly review semi-Lagrangian and splitting techniques when applied to the deterministic linear Vlasov equation in Section~\ref{sec:deterministic}. In the sequel, we focus on the temporal discretization and we thus only deal with semi-discrete splitting methods.

Various types of stochastic perturbations of the linear Vlasov equation are considered. We illustrate below how the properties of the exact solutions differ for each type of perturbation, and how this needs to be taken into account in the construction of the splitting integrators. Basic concepts and results on splitting schemes, when applied to deterministic differential equations, are found for instance in the monographs~\cite{MR2840298,MR3642447} and in the review articles~\cite{MR2009376,blanes}, see also references therein. For applications of splitting schemes to the temporal discretization of stochastic partial differential equations, we refer the interested reader to the following (non-exhaustive) list:  \cite{m06,MR2646103,MR3617573,MR3912762,MR3839068,MR4019051,
MR4132896,MR4263224,MR4535413,MR4400428,bcg22,b23}. In all the considered stochastic perturbations of the linear Vlasov equation, a similar strategy is applied to define the splitting schemes: the contributions of the deterministic and stochastic parts in the evolution are treated separately. In addition, the deterministic part is always treated by the same Lie--Trotter splitting scheme. For each type of noise the exact solution for the stochastic contribution is known exactly. Considering splitting schemes is appealing in our context since this leads to effective explicit numerical schemes which are able to preserve some qualitative properties of the exact solutions. This is justified theoretically and demonstrated numerically for all the stochastic versions of the linear Vlasov equation considered in this work. Let us now describe these versions and the main results. We refer to Section~\ref{sec:notation} below for the notation.

\begin{itemize}
\item \emph{Additive noise perturbation}: in Section~\ref{sec:additive} we consider the following stochastic PDE
\[
\left\lbrace
\begin{aligned}
&\text df^\ad(t,x,v)+v\cdot\nabla_xf^\ad(t,x,v)\,\text dt+E(x)\cdot\nabla_v f^\ad(t,x,v)\,\text dt=\text dW(t,x,v)~,\\
&f^\ad(0,x,v)=f_0(x,v)~,
\end{aligned}
\right.
\]
see equation~\eqref{eq:SPDE-add}. The splitting scheme in this case is given by~\eqref{eq:LT-add}. Since the noise is additive, the solution is a Gaussian process, thus the positivity or the $L^p$ norm of the initial value $f_0$ cannot be preserved. Instead, we prove trace formulas: the second-order moment of the $L^2$ norm of the exact solution grows linearly with time (see Proposition~\ref{propo:trace}). This linear evolution is preserved for the numerical solution computed using the proposed splitting scheme (see Proposition~\ref{propo:tracenum}).
\item \emph{Multiplicative noise perturbation with It\^o interpretation}: in Section~\ref{sec:multiplicativeIto} we consider the following stochastic PDE
\[
\left\lbrace
\begin{aligned}
&\text df^\mI(t,x,v)+v\cdot\nabla_xf^\mI(t,x,v)\,\text dt+E(x)\cdot\nabla_v f^\mI(t,x,v)\,\text dt=f^\mI(t,x,v)\,\text dW(t,x,v)~,\\
&f^\mI(0,x,v)=f_0(x,v)~,
\end{aligned}
\right.
\]
see equation~\eqref{eq:SPDE-mult-Ito}. The splitting scheme in this case is given by~\eqref{eq:LT-mult-Ito}. Proposition~\ref{propo:multIto} provides the qualitative properties for the exact solution. First, one has almost sure preservation of the positivity of the initial value. Second, the mass of the initial value is preserved in expectation. Finally, under an appropriate condition on the diffusion coefficients, the second-order moment of the $L^2$ norm is shown to satisfy a simple evolution law. All those properties are shown to be preserved by the splitting scheme~\eqref{eq:LT-mult-Ito}, see Proposition~\ref{propo:LT-mult-Ito}.
\item \emph{Multiplicative noise perturbation with Stratonovich interpretation}: in Section~\ref{sec:multiplicativeStrato} we consider the following stochastic PDE
\[
\left\lbrace
\begin{aligned}
&\text df^\mS(t,x,v)+v\cdot\nabla_xf^\mS(t,x,v)\,\text dt+E(x)\cdot\nabla_v f^\mS(t,x,v)\,\text dt=f^\mS(t,x,v)\circ \text dW(t,x,v)~,\\
&f^\mS(0,x,v)=f_0(x,v)~,
\end{aligned}
\right.
\]
see equation~\eqref{eq:SPDE-mult-Strato}. The splitting scheme in this case is given by~\eqref{eq:LT-mult-Strato}. Proposition~\ref{propo:multStrato} provides the qualitative properties for the exact solution. First, one has almost sure preservation of the positivity of the initial value. Moreover, under an appropriate condition on the diffusion coefficients, the second-order moment of the $L^2$ norm is shown to satisfy a simple evolution law. All those properties are shown to be preserved by the splitting scheme~\eqref{eq:LT-mult-Strato}, see Proposition~\ref{propo:LT-mult-Strato}.
\item \emph{Transport noise perturbation}: in Section~\ref{sec:transport} we consider the following stochastic PDE
\[
\left\lbrace
\begin{aligned}
&\text df^\tr(t,x,v)+v\cdot\nabla_xf^\tr(t,x,v)\,\text dt+E(x)\cdot\nabla_v f^\tr(t,x,v)\,\text dt+\nabla_vf^\tr(t,x,v)\odot \text d\bm{W}(t,x)=0~,\\
&f^\tr(0,x,v)=f_0(x,v)~,
\end{aligned}
\right.
\]
see equation~\eqref{eq:SPDE-transport}. The splitting scheme in this case is given by~\eqref{eq:LT-transport}. Almost preservation of positivity and $L^p$ norms are stated in Proposition~\ref{propo:transport} for the exact solution. We show that those properties are also preserved by the splitting scheme~\eqref{eq:LT-transport}, see Proposition~\ref{propo:transportnum}.
\end{itemize}
All the theoretical results mentioned above are illustrated by numerical experiments for
each stochastic partial differential equation (SPDE). Snapshots of the numerical solutions also illustrate the influence of the stochastic perturbation on the behavior of the solutions. Even if we do not provide a rigorous convergence analysis, we conjecture that the proposed splitting schemes are consistent. Moreover, we report below numerical experiments to investigate mean-square rates of convergence and identify that in all cases convergence with order $1$ is expected to hold. We leave a convergence analysis and other questions such as treating nonlinear models or constructing higher order integrators for future works. Our implementation is based on the code from \cite{codesmatlab} and is available under
\href{https://doi.org/10.5281/zenodo.10495233}{https://doi.org/10.5281/zenodo.10495233}.

Sections~\ref{sec:additive},~\ref{sec:multiplicative} and~\ref{sec:transport} below are organized similarly, this allows us to exhibit the main common or different features of the considered SPDEs and their numerical discretizations.

\section{Notation}\label{sec:notation}

The dimension $d\in\mathbb{N}$ is an arbitrary integer. Let $\T^d=(\R/\Z)^d$ denote the $d$-dimensional torus. For any differentiable mapping $f:\T^d\times\R^d\to\R$, let $\nabla_xf=\bigl(\partial_{x_1}f,\ldots,\partial_{x_d}f\bigr)$ and $\nabla_vf=\bigl(\partial_{v_1}f,\ldots,\partial_{v_d}f\bigr)$ denote the gradients of $f$ with respect to the spatial variable $x\in\T^d$ and the velocity $v\in\R^d$ respectively.

For any real number $p\in[1,\infty)$, and any measurable mapping $f:\T^d\times\R^d\to\R$, define
\[
\|f\|_{L_{x,v}^p}=\Bigl(\iint {|f(x,v)|}^p\, \text dx\, \text dv\Bigr)^{\frac{1}{p}}:=\left(\int_{\R^d}\int_{\T^d} {|f(x,v)|}^p\, \text dx\, \text dv\right)^{\frac{1}{p}}
\]
and set $L_{x,v}^p=\{f;~\|f\|_{L_{x,v}^p}<\infty\}$. Below, to simplify the notation we will not write the domains of the integrals. 
In addition, for any measurable mapping $f:\T^d\times\R^d\to\R$, define
\[
\|f\|_{L_{x,v}^\infty}=\underset{(x,v)\in\T^d\times\R^d}{\rm ess~ sup}~|f(x,v)|
\]
and set $L_{x,v}^\infty=\{f;~\|f\|_{L_{x,v}^\infty}<\infty\}$.

Given a function $f:(t,x,v)\in\R^+\times\T^d\times\R^d\mapsto f(t,x,v)\in\R$, the notation $f(t)$ for $t\ge 0$ is frequently used in the sequel to denote the mapping $(x,v)\in\T^d\times\R^d\mapsto f(t,x,v)$. The expression $f(t)(x,v)=f(t,x,v)$ is also often employed below.

To describe the considered class of (stochastic) partial differential equations, let us denote by $E:x\in\T^d\mapsto E(x)\in\mathbb{R}^d$ a vector field which is assumed to be of class $\mathcal{C}^\infty$. In addition,
the initial value is denoted by $f_0:(x,v)\in\T^d\times\R^d\mapsto f_0(x,v)\in\R$ in the sequel.
Below, it is assumed that $f_0$ is non-random.
Precise regularity and integrability conditions on $f_0$ are imposed below.

It remains to describe the setting for the considered stochastic perturbation of linear Vlasov equations.
Let $(\Omega,\mathcal{F},\mathbb{P})$ be a probability space and $\bigl(\mathcal{F}_t\bigr)_{t\ge 0}$ be a filtration satisfying the usual conditions.
The expectation operator is denoted by $\E[\cdot]$. Let $K\in\N$ be an integer, and let $\big(\beta_k\bigr)_{1\le k\le K}$ be a family of independent standard real-valued Wiener processes, adapted to the filtration.

In Sections~\ref{sec:additive} and~\ref{sec:multiplicative}, the stochastic perturbation is of additive or multiplicative type, and is written as a real-valued Wiener process defined by
\begin{equation}\label{eq:Wtxv}
W(t,x,v)=\sum_{k=1}^{K}\beta_k(t)\sigma_k(x,v),\quad t\ge 0, x\in\T^d, v\in\R^d,
\end{equation}
where $\sigma_k:(x,v)\in\T^d\times\R^d\mapsto \sigma_k(x,v)\in\R$ are real-valued mappings, for $1\le k\le K$.

In Section~\ref{sec:transport}, the stochastic perturbation is of transport type and it does not depend on
of the variable $v$. It is written as a $\R^d$-valued Wiener process defined by
\begin{equation}\label{eq:Wtx}
\bm{W}(t,x)=\sum_{k=1}^{K}\beta_k(t)\bm{\sigma}_k(x),\quad t\ge 0, x\in\T^d,
\end{equation}
where $\sigma_{j,k}:x\in\T^d\mapsto \sigma_{j,k}(x)\in\R$ are real-valued mappings, and for all $x\in\T^d$ one has $\bm{\sigma}_k(x)=\bigl(\sigma_{1,k}(x),\ldots,\sigma_{d,k}(x)\bigr)\in\R^d$ .

The notation $W(t)$ is also used below for the random mapping $(x,v)\in\T^d\times\R^d\mapsto W(t,x,v)$, associated with~\eqref{eq:Wtxv} while the notation $\bm{W}(t)$ stands for the random mapping $x\in\T^d\mapsto \bm{W}(t,x)$, associated with~\eqref{eq:Wtx} respectively, for all $t\ge 0$. We refer for instance to~\cite[Chapter~4]{DPZ} and~\cite[Chapter~10]{LPS} for details on Wiener processes and stochastic integrals with values in Hilbert spaces.

It is assumed that the mappings $\sigma_k$ and $\sigma_{j,k}$, $1\le j\le d$, $1\le k\le K$ are of class $\mathcal{C}^\infty$. Further growth or integrability conditions on these mappings are imposed below depending on the considered class of problems.

Finally, the numerical methods considered in this paper use the following notation. For the temporal discretization, the time-step size is denoted by $\tau$. Without loss of generality, it is assumed that $\tau\in(0,1)$. For any nonnegative integer $n\ge 0$, set $t_n=n\tau$, and introduce Wiener increments
\begin{equation}\label{eq:WincBeta}
\delta \beta_{k,n}=\beta_k(t_{n+1})-\beta_k(t_n)~,\quad 1\le k\le K.
\end{equation}
The random variables $\bigl(\tau^{-\frac12}\delta\beta_{n,k}\bigr)_{n\ge 0,1\le k\le K}$ are independent standard real-valued Gaussian random variables (with mean equal to $0$ and variance equal to $1$).

Increments of the Wiener processes defined by~\eqref{eq:Wtxv} and~\eqref{eq:Wtx} are given for any nonnegative integer $n\ge 0$ by
\begin{align}
&\delta W_n(x,v)=W(t_{n+1},x,v)-W(t_n,x,v)=\sum_{k=1}^{K}\delta \beta_{k,n}\sigma_k(x,v)~,\quad (x,v)\in\T^d\times\R^d,\label{eq:WincW}\\
&\delta \bm{W}_n(x)=\bm{W}(t_{n+1},x)-\bm{W}(t_n,x)=\sum_{k=1}^{K}\delta \beta_{k,n}\bm{\sigma}_k(x)~,\quad x\in\T^d.\label{eq:WincbW}
\end{align}

In this work, we focus on the temporal discretization of stochastic linear Vlasov equations.
However, practical implementation requires some discretization procedure with respect to the spatial and velocity variables $x\in\T^d$ and $v\in\R^d$. In the numerical experiments below, a semi-Lagrangian approach is employed, see for instance \cite{MR1672731,MR2051070,MR4099398} or \cite[Chapter 6]{MR1660086}. The semi-Lagrangian approach exploits characteristic curves to express the numerical solution at time $t_{n+1}$ at grid points in terms of the numerical solution at time $t_n$. Since the characteristic curves in general do not hit grid points, an interpolation procedure is employed to the numerical solution at time $t_n$.
Note that, in order to preserve the positivity property of solutions to the considered SPDEs (see below), we use a linear interpolation in the implementation of the studied numerical schemes.
However the analysis of semi-Lagrangian discretization is out of the scope of this work. The mesh sizes with respect to the spatial and velocity domains are denoted by $\delta x$ and $\delta v$ respectively. Since the velocity space $\R^d$ is unbounded, in practice a truncation procedure at large velocities $v$ is imposed, whereas for the spatial variable $x$, periodic boundary conditions are used. Besides, the verification of trace formulas and evolution laws for some moments of the solutions and the study of mean-square convergence require a Monte Carlo averaging procedure in order to compute approximate values of expectations.

\section{Preliminaries on the deterministic linear Vlasov equation}\label{sec:deterministic}

The objective of this section is to provide basic and well-known background on the deterministic version of the linear Vlasov equation, before considering stochastically perturbed versions in the next sections. Some notation and fundamental properties introduced below are employed in the sequel. We study the linear Vlasov equation
\begin{equation}\label{eq:Vlasov}
\left\lbrace
\begin{aligned}
&\partial_tf^\de(t,x,v)+v\cdot\nabla_xf^\de(t,x,v)+E(x)\cdot\nabla_v f^\de(t,x,v)=0~,~t\ge 0, x\in\T^d, v\in \R^d,\\
&f^\de(0,x,v)=f_0(x,v)~,\quad x\in \T^d, v\in\R^d,
\end{aligned}
\right.
\end{equation}
where the unknown is a mapping $f^\de:(t,x,v)\in\R^+\times\T^d\times\R^d\mapsto f^\de(t,x,v)\in\R$. Under appropriate regularity and integrability conditions, given any initial value $f_0$, the partial differential equation~\eqref{eq:Vlasov} admits a unique solution.

\subsection{Analysis and properties of the problem}
The solution of~\eqref{eq:Vlasov} can be constructed using solutions of the associated ordinary differential equation
\begin{equation}\label{eq:ODE}
\left\lbrace
\begin{aligned}
&\dot{x}_t=v_t\\
&\dot{v}_t=E(x_t).
\end{aligned}
\right.
\end{equation}
Let $\bigl(\phi_t\bigr)_{t\in\R}$ denote the associated flow of this differential equation. Recall that this means that for any initial value $(x_0,v_0)\in\T^d\times\R^d$, the solution at time $t\ge 0$ is given by $(x_t,v_t)=\phi_t(x_0,v_0)$. In addition, the flow map property $\phi_{t+s}=\phi_t\circ\phi_s$ is satisfied for all $t,s\in\R$. Finally, for all $t\in\R$, the mapping $\phi_t$ is a smooth diffeomorphism and one has $(\phi_t)^{-1}=\phi_{-t}$. Since the vector field $(x,v):\T^d\times\R^d\mapsto (v,E(x))\in\R^d\times\R^d$ does not depend on the time variable, $\bigl(\phi_{t}^{-1}\bigr)_{t\in\R}=\bigl(\phi_{-t}\bigr)_{t\in \R}$ is obtained by reverting time in the dynamics~\eqref{eq:ODE}, i.\,e.
by considering the flow of the ordinary differential equation
\[
\left\lbrace
\begin{aligned}
&\dot{x}_t=-v_t\\
&\dot{v}_t=-E(x_t).
\end{aligned}
\right.
\]
The flow $\bigl(\phi_t\bigr)_{t\in\R}$ associated with~\eqref{eq:ODE} satisfies a remarkable property: it preserves the volume in $\T^d\times\R^d$. This can be written as follows: for any integrable function $\varphi:\T^d\times\R^d\to \R$, one has
\[
\iint \varphi(\phi_t(x,v))\, \text dx\, \text dv=\iint \varphi(x,v)\, \text dx\, \text dv.
\]

We are now in position to recall the links between the Vlasov PDE~\eqref{eq:Vlasov} and the ODE system~\eqref{eq:ODE}. On the one hand, assume that $\bigl(f^\de(t)\bigr)_{t\ge 0}$ is solution of the PDE~\eqref{eq:Vlasov}, then for any solution $\bigl(x_t,v_t\bigr)_{t\ge 0}$ of the ODE system~\eqref{eq:ODE}, applying the chain rule, one has
\[
\frac{\text df^\de(t,x_t,v_t)}{\text dt}=0.
\]
As a consequence, for all $t\ge 0$, $x_0\in\T^d$ and $v_0\in\R^d$, one has
\[
f^\de(t,\phi_t(x_0,v_0))=f^\de(t,x_t,v_t)=f^\de(0,x_0,v_0)=f_0(x_0,v_0).
\]
Therefore the ODE system~\eqref{eq:ODE} provides characteristic curves for the Vlasov equation~\eqref{eq:Vlasov}.  This gives a strategy to solve~\eqref{eq:Vlasov} by the method of lines: the solution $f(t)$ at any time $t\ge 0$ is given by
\begin{equation}\label{eq:solutionVlasov}
f^\de(t,x,v)=f_0\bigl(\phi_t^{-1}(x,v)\bigr)~,\quad x\in\T^d, v\in\R^d.
\end{equation}
On the other hand, under sufficient regularity conditions, one can check that the mapping $f^\de:(t,x,v)\in\R^+\times\T^d\times\R^d\mapsto f^\de(t,x,v)\in\R$ defined by the expression~\eqref{eq:solutionVlasov} solves the Vlasov equation~\eqref{eq:Vlasov}.

The expression~\eqref{eq:solutionVlasov} leads to define a group of linear operators $\bigl(S(t)\bigr)_{t\ge 0}$ as follows: for all $t\in\R$ and any measurable mapping $f:\T^d\times\R^d\to\R$, set
\begin{equation}\label{eq:S}
S(t)f=f\bigl(\phi_t^{-1}(\cdot)\bigr).
\end{equation}
The group property $S(t+s)=S(t)S(s)$, $t,s\in\R$, follows from the group property of the flow $\bigl(\phi_t\bigr)_{t\in\R}$.

Let us describe some remarkable properties of solutions of the Vlasov equation~\eqref{eq:Vlasov}, which are considered below for
stochastic perturbations of this PDE and for its numerical discretization. Let $f^\de(t)=S(t)f_0$ for all $t\ge 0$.
\begin{itemize}
\item \emph{Preservation of positivity}. Assume that $f_0(x,v)\ge 0$ for all $(x,v)\in\T^d\times\R^d$. Then one has $f^\de(t,x,v)\ge 0$ for all $t\ge 0$ and $(x,v)\in\T^d\times\R^d$.
\item \emph{Preservation of integrals}. Let $\Phi:\R\to\R^+$ be a real-valued measurable mapping. Then for all $t\ge 0$ one has
\[
\iint \Phi(f^\de(t,x,v))\,\text dx\, \text dv=\iint \Phi(f_0(x,v))\,\text dx\, \text dv.
\]
\item \emph{Isometry property}. For all $p\in[1,\infty]$ and $t\ge 0$, the linear operator $S(t):L_{x,v}^p\to L_{x,v}^p$ is an isometry: if $f_0\in L_{x,v}^p$, then $f^\de(t)=S(t)f_0\in L_{x,v}^p$ for all $t\ge 0$ and
\[
\|f^\de(t)\|_{L_{x,v}^p}=\|S(t)f_0\|_{L_{x,v}^p}=\|f_0\|_{L_{x,v}^p}.
\]
\end{itemize}
The preservation of positivity and the preservation of the $L_{x,v}^\infty$ norm are straightforward consequences of the expression~\eqref{eq:solutionVlasov}. The preservation of integrals property follows from the preservation of volume by the flow $\bigl(\phi_t\bigr)_{t\in\R}$ and implies the isometry property above when $p\in[1,\infty)$ by choosing $\Phi=|\cdot|^p$.

Finally, the preservation of volume by the flow $\bigl(\phi_t\bigr)_{t\ge 0}$ also implies the following result.
Assume that $f_0\in L_{x,v}^1$, then for any bounded and continuous function $\varphi:\T^d\times\R^d\to \R$ and for all $t\ge 0$, one has
\[
\iint f^\de(t,x,v)\varphi(x,v)\,\text dx\,\text dv=\iint f(0,x,v)\varphi(\phi_t(x,v)) \,\text dx\,\text dv.
\]
The properties above provide a possible probabilistic interpretation of solutions to the Vlasov PDE~\eqref{eq:Vlasov}. If the initial value $f_0$ is a probability density function,
in particular, this requires the conditions $f_0\ge 0$ and $\|f_0\|_{L_{x,v}^1}=1$, then for all $t\ge 0$, the mapping $f^\de(t)=S(t)f_0$ is a probability density function.
If $(\mathcal{X}_0,\mathcal{V}_0)$ is a $\T^d\times\R^d$-valued random variable having density $f_0$ with respect to the Lebesgue measure, then $f^\de(t)$ is the density with respect to the Lebesgue measure of the random variable $(\mathcal{X}_t,\mathcal{V}_t)=\phi_t(\mathcal{X}_0,\mathcal{V}_0)$. In other words, $f^\de(t)$ is the probability density function associated with the transport by the dynamics~\eqref{eq:ODE} of the distribution $f_0$ of the initial values $(x_0,v_0)$. In some contexts, in particular for Hamiltonian dynamics ($E=-\nabla V$ where $V:\T^d\to\R$ is a smooth mapping), the Vlasov equation~\eqref{eq:Vlasov} is referred to as the Liouville equation associated with the dynamics~\eqref{eq:ODE}.

\subsection{Numerical approximation}
Let us describe numerical integrators applied to the linear Vlasov equation~\eqref{eq:Vlasov}. As already mentioned, we focus on temporal discretization only in this work. To approximate solutions and preserve the properties mentioned above, it is natural to rely on splitting integrators. We only present Lie--Trotter versions, for simplicity and motivated by the fact that stochastic perturbations usually lead to numerical methods which are of strong order less than $2$ (the usual order of Strang splitting for deterministic problems, see for instance \cite{CHENG1976330,MR1672731,MR3151388,MR3606462}). Let us mention that high-order numerical schemes for the deterministic part of the problem (for instance using
a Strang splitting instead of a Lie--Trotter splitting) could however be of interest for stochastic problems with small noise, see for instance \cite[Chapter~4]{MR4369963}.

The principle of splitting integrators is to combine solutions of subsystems of evolution equations which can be solved exactly. For the deterministic linear Vlasov equation~\eqref{eq:Vlasov} it is natural to decompose the problem into the two subsystems
\begin{subequations}
\begin{align}
&\partial_tf^1(t,x,v)+v\cdot\nabla_xf^1(t,x,v)=0~,\quad t\ge 0, x\in\T^d, v\in \R^d\label{eq:VlasovSplit1}\\
&\partial_tf^2(t,x,v)+E(x)\cdot\nabla_vf^2(t,x,v)=0~,\quad t\ge 0, x\in\T^d, v\in \R^d\label{eq:VlasovSplit2}.
\end{align}
\end{subequations}
Let $\bigl(S^1(t)\bigr)_{t\in\R}$ and $\bigl(S^2(t)\bigr)_{t\in\R}$ be the associated groups of linear operators: the solutions at time $t\ge 0$ of~\eqref{eq:VlasovSplit1} and~\eqref{eq:VlasovSplit2} are respectively $f^1(t)=S^1(t)f_0$ and $f^2(t)=S^2(t)f_0$, where for all $t\in\R$, $x\in\T^d$, $v\in\R^d$ and any measurable mapping $f:\T^d\times\R^d\to\R$, one has
\begin{align*}
&S^1(t)f(x,v)=f(x-tv,v)\\
&S^2(t)f(x,v)=f(x,v-tE(x)).
\end{align*}
Note that the partial differential equations~\eqref{eq:VlasovSplit1} and~\eqref{eq:VlasovSplit2} are associated with the ordinary differential equations
\[
\left\lbrace
\begin{aligned}
&\dot{x}_t^1=v_t^1\\
&\dot{v}_t^1=0,
\end{aligned}
\right.\quad\text{and}
\quad%,\quad
\left\lbrace
\begin{aligned}
&\dot{x}_t^2=0\\
&\dot{v}_t^2=E(x_t^2),
\end{aligned}
\right.
\]
for which expressions of the exact solutions are known: for all $t\ge 0$, one has
\begin{align*}
\bigl(x_t^1,v_t^1\bigr)&=\phi_t^1(x_0^1,v_0^1)=\bigl(x_0^1+tv_0^1,v_0^1\bigr)\\
\bigl(x_t^2,v_t^2\bigr)&=\phi_t^2(x_0^2,v_0^2)=\bigl(x_0^2,v_0^2+tE(x_0^2)\bigr).\\
\end{align*}
The link between the Vlasov equation~\eqref{eq:Vlasov} and the dynamics~\eqref{eq:ODE} is retrieved for the subsystems~\eqref{eq:VlasovSplit1} and~\eqref{eq:VlasovSplit2} and the associated dynamics: for all $t\ge 0$ and any measurable mapping $f:\T^d\times\R^d\to \R$ one has
\begin{equation}\label{eq:S1S2}
S^1(t)f=f\bigl((\phi_t^1)^{-1}(\cdot)\bigr)~,\quad S^2(t)f=f\bigl((\phi_t^2)^{-1}(\cdot)\bigr)
\end{equation}
where $\bigl(\phi_t^{1}\bigr)_{t\in\R}$ and $\bigl(\phi_t^{2}\bigr)_{t\in\R}$ are the flows associated with the dynamics introduced above.

As a consequence, it is straightforward to check that the properties stated above for solutions of the Vlasov equation~\eqref{eq:Vlasov} (preservation of positivity, preservation of integrals and isometry property) and the interpretation of the solution as the probability density functions associated with the transport by dynamics of the distribution of initial value, persist for the solutions of the subsystems~\eqref{eq:VlasovSplit1} and~\eqref{eq:VlasovSplit2}.

Applying the Lie--Trotter splitting method yields to the definition of the following numerical method for the deterministic linear Vlasov equation~\eqref{eq:Vlasov}:
set $f_0^\de=f_0$ and for any nonnegative integer $n\ge 0$, set
\begin{equation}\label{eq:LTdeter}
f_{n+1}^\de=S^2(\tau)S^1(\tau)f_n^\de.
\end{equation}
Using the definitions\eqref{eq:S1S2} of $S^1(t)$ and $S^2(t)$ above, one has for all $n\ge 0$, $x\in\T^d$ and $v\in\R^d$
\begin{align*}
f_{n+1}^\de(x,v)&=\bigl(S^2(\tau)S^1(\tau)f_n^\de\bigr)(x,v)=\bigl(S^1(\tau)f_n^\de\bigr)(x,v-\tau E(x))\\
&=f_n^\de(x-\tau v-\tau^2E(x),v-\tau E(x)).
\end{align*}
With that expression, it appears that the scheme~\eqref{eq:LTdeter} can be interpreted as a discrete version of the expression~\eqref{eq:solutionVlasov} for the exact solution of~\eqref{eq:Vlasov}, where the flow $\phi_t$ is approximated using a splitting integrator based on the two subsystems of ordinary differential equations above.

It is straightforward to check that the properties stated above for the exact solution of~\eqref{eq:Vlasov} also hold for the Lie--Trotter splitting scheme~\eqref{eq:LTdeter}, for any value of the time step size $\tau$. In particular, one has the following results.
\begin{itemize}
\item \emph{Preservation of positivity}. Assume that $f_0(x,v)\ge 0$ for all $(x,v)\in\T^d\times\R^d$. Then for all $n\in\N$ one has $f_n^\de(x,v)\ge 0$ for all $(x,v)\in\T^d\times\R^d$.
\item \emph{Isometry property}. Let $p\in[1,\infty]$ and assume that $f_0\in L_{x,v}^p$. Then for all $n\ge 0$ one has $f_n^\de\in L_{x,v}^p$ and
\[
\|f_n^\de\|_{L_{x,v}^p}=\|f_0\|_{L_{x,v}^p}.
\]
\end{itemize}

We conclude these preliminaries on the deterministic linear Vlasov equation~\eqref{eq:Vlasov} with a numerical experiment with the goal to illustrate the behavior of the solution to the linear Vlasov PDE~\eqref{eq:Vlasov}, in dimension $d=1$.
In Figure~\ref{fig:snap-det} below, snapshots at times $\{0,0.5,1,1.5,2,2.5\}$ of the numerical solution computed using the Lie--Trotter splitting scheme~\ref{eq:LTdeter} are displayed. We consider a standard two-stream instability test case: the initial value $f_0$ is given by
\begin{equation}\label{eq:f0}
f_0(x,v)=\frac{\e^{-v^2/2}}{\sqrt{2\pi}}\bigl(1+ 0.05\cos{(2\pi x)}\bigr)v^2~,\quad x\in\T, v\in\R,
\end{equation}
The two-stream instability has been extensively used to illustrate wave-particle interactions since the seminal paper~\cite{Denavit}, see also for instance~\cite{Bittencourt,KNORR1973165,CHENG1976330}.
In addition, the vector field $E$ is given by
\begin{equation}\label{eq:E}
E(x)=\cos(2\pi x)~,\quad x\in\T.
\end{equation}
The discretization parameters for Figure~\ref{fig:snap-det} are $\delta x=\frac{1}{500}$, $\delta v=\frac{4\pi}{500}$, and $\tau=0.1$.

\begin{figure}[h]
\begin{subfigure}{.3\textwidth}
  \centering
  \includegraphics[width=1\linewidth]{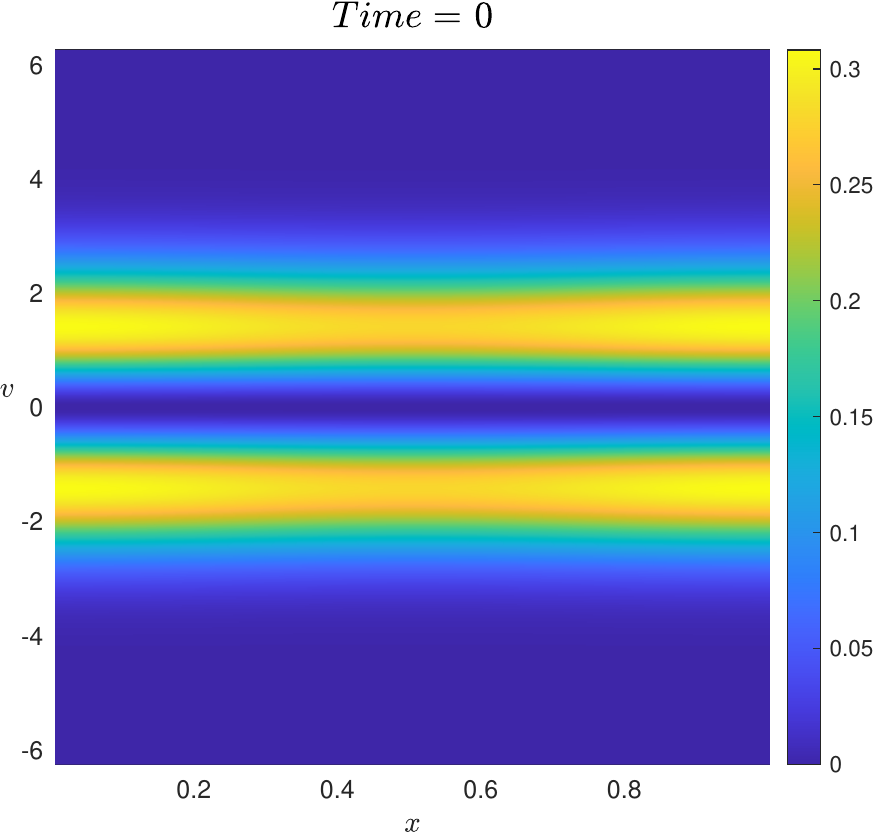}
\end{subfigure}%
\begin{subfigure}{.3\textwidth}
  \centering
  \includegraphics[width=1\linewidth]{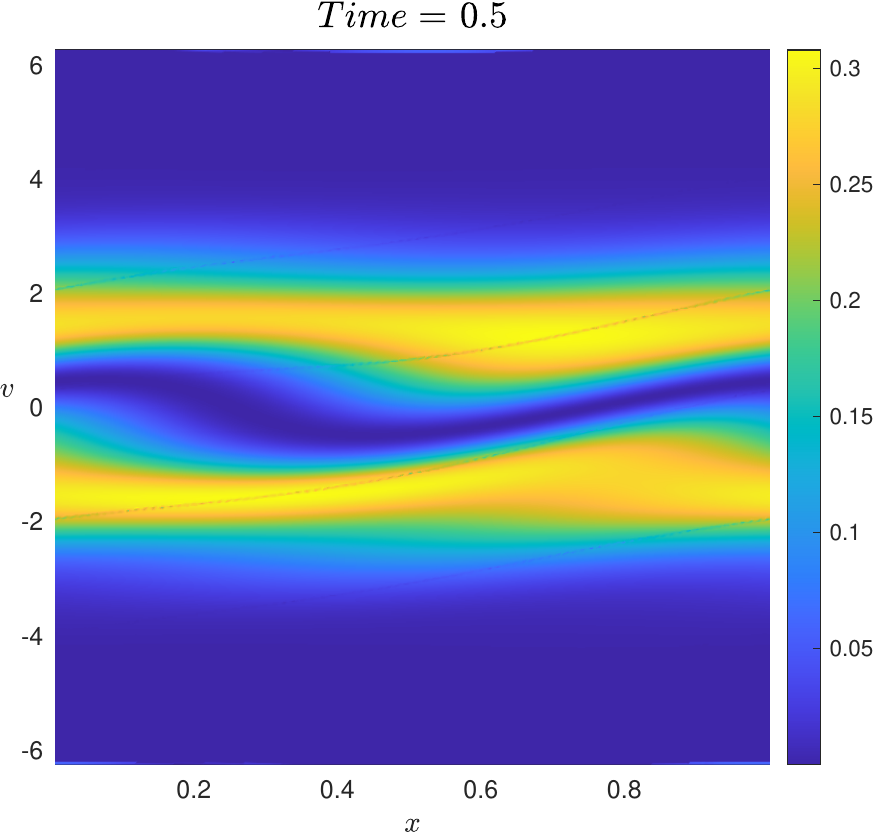}
\end{subfigure}%
\begin{subfigure}{.3\textwidth}
\includegraphics[width=1\linewidth]{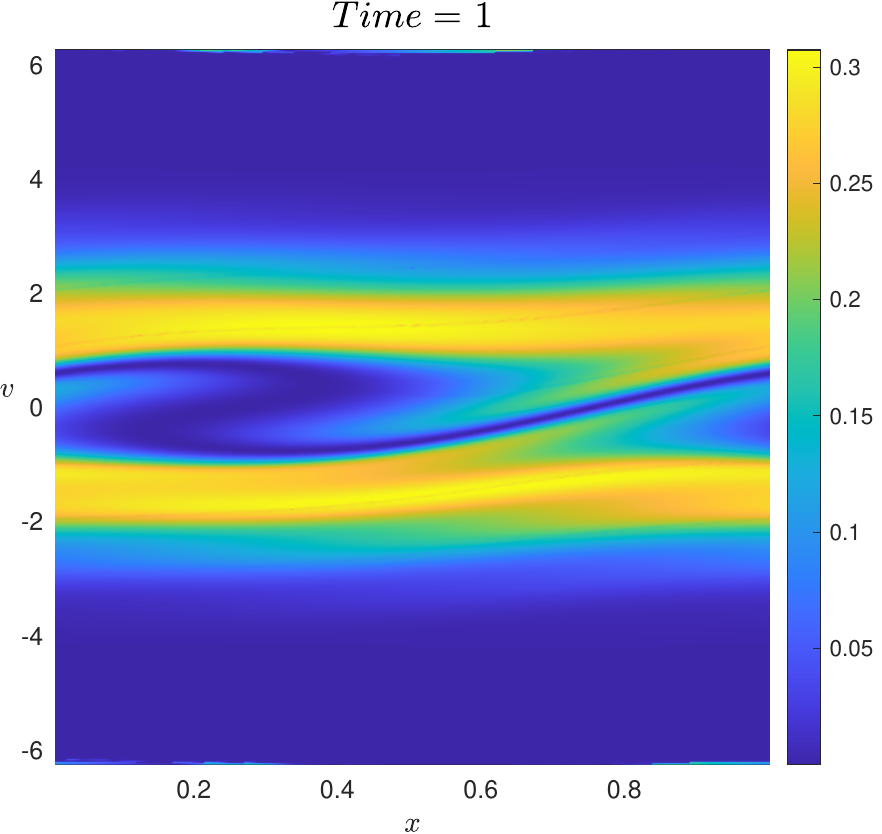}
\end{subfigure}\\[1ex]%
\begin{subfigure}{.3\textwidth}
  \centering
  \includegraphics[width=1\linewidth]{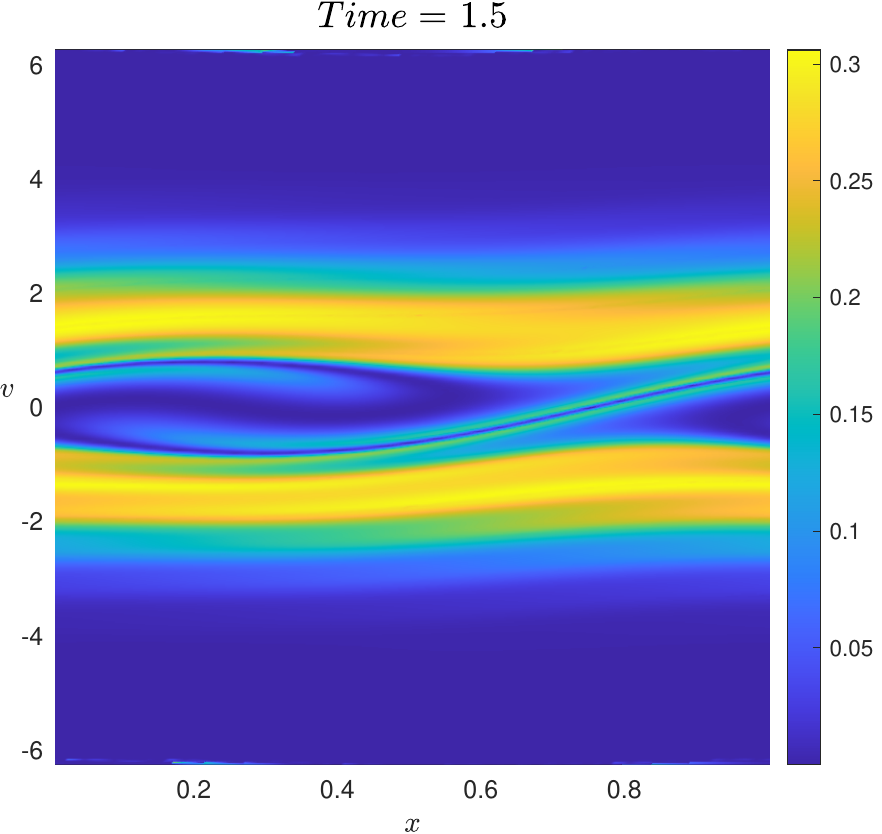}
\end{subfigure}%
\begin{subfigure}{.3\textwidth}
  \centering
  \includegraphics[width=1\linewidth]{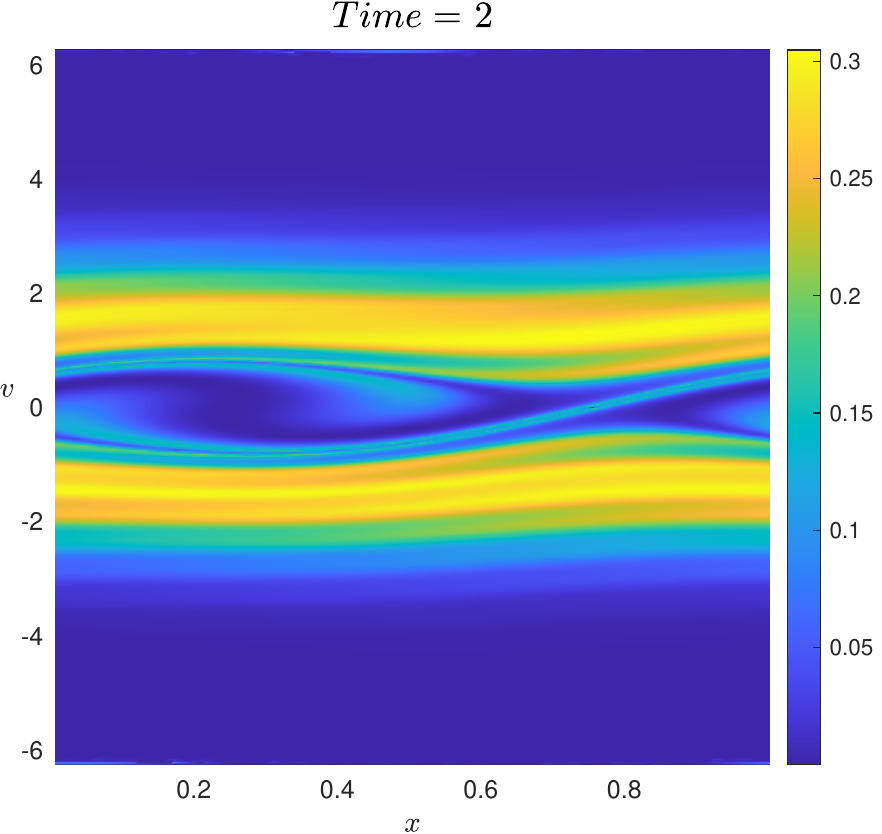}
\end{subfigure}%
\begin{subfigure}{.3\textwidth}
  \centering
  \includegraphics[width=1\linewidth]{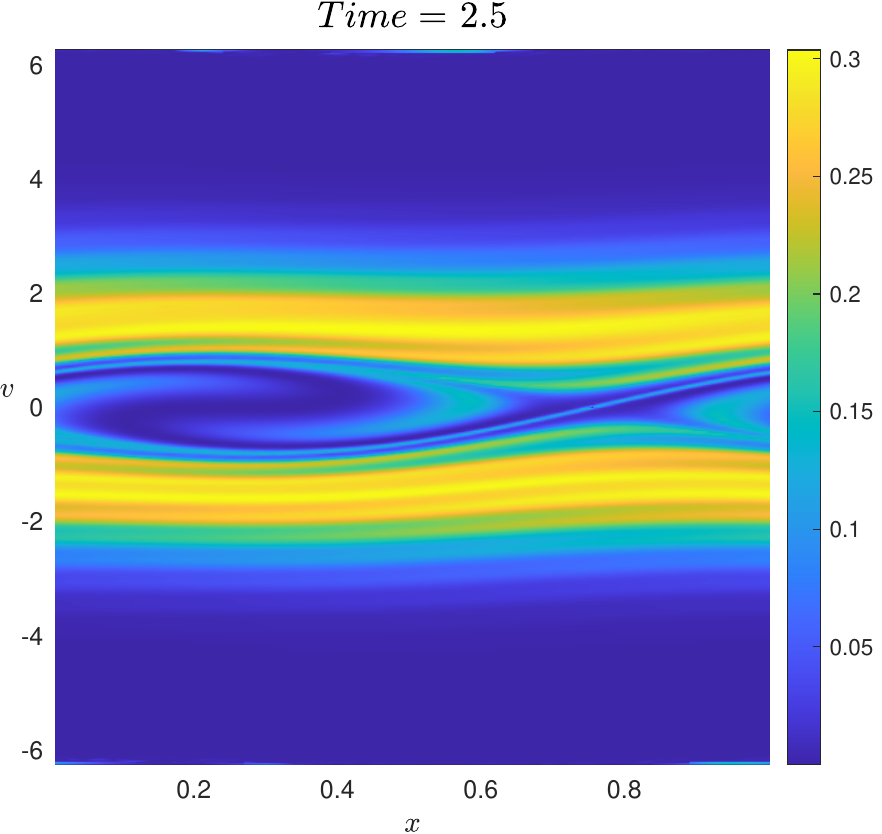}
\end{subfigure}
\caption{Snapshots: approximation of the solution of the deterministic PDE~\eqref{eq:Vlasov} with initial value $f_0$ given by~\eqref{eq:f0}, at times~$\{0,0.5,1,1.5,2,2.5\}$, using the Lie--Trotter splitting scheme~\eqref{eq:LTdeter} with time-step size $\tau=0.1$.}
\label{fig:snap-det}
\end{figure}

\section{The stochastic linear Vlasov equation perturbed by additive noise}\label{sec:additive}

In this section, we consider a version of the linear Vlasov equation~\eqref{eq:Vlasov} perturbed by additive noise of the type~\eqref{eq:Wtxv}: for $t\ge 0$, $x\in\T^d$, $v\in \R^d$
\begin{equation}\label{eq:SPDE-add}
\left\lbrace
\begin{aligned}
&\text df^\ad(t,x,v)+v\cdot\nabla_xf^\ad(t,x,v)\,\text dt+E(x)\cdot\nabla_v f^\ad(t,x,v)\,\text dt=\text dW(t,x,v),\\ %~,\quad t\ge 0, x\in\T^d, v\in \R^d,\\
&f^\ad(0,x,v)=f_0(x,v).%~,\quad x\in \T^d, v\in\R^d.
\end{aligned}
\right.
\end{equation}

\subsection{Analysis and properties of the problem}

We recall that the initial value $f_0$ is assumed to be non-random. The solution of the stochastic partial differential equation~\eqref{eq:SPDE-add}
can be interpreted in different ways, as explained below. Those interpretations are equivalent under appropriate regularity and integrability conditions.

On the one hand, using the group $\bigl(S(t)\bigr)_{t\in\R}$ of linear operators given by~\eqref{eq:S} in Section~\ref{sec:deterministic}
and the definition of the noise~\eqref{eq:Wtxv}, one can consider mild solutions
in the sense of~\cite[Chapter~5]{DPZ}, see also~\cite[Section~10.4]{LPS} for instance: for all $t\ge 0$, one has
\begin{equation}\label{eq:SPDE-add-solution1}
f^\ad(t)=S(t)f_0+\int_{0}^{t}S(t-s)\,\text dW(s)=S(t)f_0+\sum_{k=1}^{K}\int_{0}^{t}S(t-s)\sigma_k \,\text d\beta_k(s),
\end{equation}
where $\int_{0}^{t}S(t-s)\,\text dW(s)$ is considered as a stochastic integral with values in the Hilbert space $L_{x,v}^2$, see~\cite[Chapter~4]{DPZ}.
On the other hand, using the expression~\eqref{eq:S} for the linear operator $S(t)$, one has for all $t\ge 0$, $x\in\T^d$ and $v\in\R^d$
\begin{equation}\label{eq:SPDE-add-solution2}
f^\ad(t,x,v)=f_0(\phi_t^{-1}(x,v))+\sum_{k=1}^{K}\int_0^t \sigma_k(\phi_{t-s}^{-1}(x,v))\,\text d\beta_k(s),
\end{equation}
where $\bigl(\phi_t\bigr)_{t\in\R}$ denotes the flow of~\eqref{eq:ODE}.

Finally note that the expression above can be retrieved by applying the It\^o--Wentzell formula (see \autoref{sec:ItoWentzell} for the statement):
if $t\mapsto (x_t,v_t)=\phi_t(x_0,v_0)$ is the solution of the ordinary differential equation~\eqref{eq:ODE} with arbitrary initial value $(x_0,v_0)\in\T^d\times\R^d$, and if the solution of~\eqref{eq:SPDE-add} is sufficiently regular, then the stochastic process $t\in\R^+\mapsto f^\ad(t,x_t,v_t)$ satisfies
\begin{align*}
\text df^\ad(t,x_t,v_t)&=\sum_{k=1}^{K}\sigma_k(x_t,v_t)\,\text d\beta_k(t)\\
f^\ad(t,x_t,v_t)&=f_0(x_0,v_0)+\sum_{k=1}^{K}\int_{0}^{t}\sigma_k(x_s,v_s)\,\text d\beta_k(s).
\end{align*}

Let us now study properties of the solution to the SPDE~\eqref{eq:SPDE-add}.
First, note that the random field $(t,x,v)\in\R^+\times\mathbb{T}^d\times\R^d\mapsto f^\ad(t,x,v)$ is Gaussian.
Owing to~\eqref{eq:SPDE-add-solution2} one has $\E[f^\ad(t,x,v)]=f_0(\phi_t^{-1}(x,v))=f^\de(t,x,v)$ for all $t\ge 0$, $x\in\T^d$ and $v\in\R^d$, where $\bigl(f^\de(t)\bigr)_{t\ge 0}$ is the solution of deterministic linear Vlasov equation~\eqref{eq:Vlasov} with initial value $f_0$. 
The covariance structure of the Gaussian process $\E[f^\ad(t_1,x_1,v_1)f^\ad(t_2,x_2,v_2)]$ is computed using It\^o's isometry formula applied to~\eqref{eq:SPDE-add-solution2}: for all $t_1,t_2\ge 0$, $x_1,x_2\in\T^d$ and $v_1,v_2\in\R^d$ one has
\begin{align*}
\E[f^\ad(t_1,x_1,v_1)&f^\ad(t_2,x_2,v_2)]=f_0(\phi_{t_1}^{-1}(x_1,v_1))f_0(\phi_{t_2}^{-1}(x_2,v_2))\\
&+\sum_{k=1}^{K}\int_{0}^{\min(t_1,t_2)}\sigma_k(\phi_{t_1-s}^{-1}(x_1,v_1))\sigma_k(\phi_{t_2-s}^{-1}(x_2,v_2))\,\text ds.
\end{align*}

Since the solution to the SPDE~\eqref{eq:SPDE-add} is a Gaussian process, the preservation properties satisfied in the deterministic case cannot be satisfied. However, one has the following remarkable property, which is often called a \emph{trace formula} in the literature, see for instance
\cite{MR2083326,MR2359504,MR3172331,MR3348201,MR2379913,MR3033008,MR3771721,MR4077824,MR4077238,MR4375496,MR4535413}.
\begin{proposition}\label{propo:trace}
Assume that $\sigma_k\in L_{x,v}^2$ for all $1\le k\le K$, and that $f_0\in L_{x,v}^2$. Let $\bigl(f^\ad(t)\bigr)_{t\ge 0}$ be the solution of the SPDE~\eqref{eq:SPDE-add} driven by the additive noise~\eqref{eq:Wtxv}. Then for all $t\ge 0$ one has $f^\ad(t)\in L^2(\Omega, L_{x,v}^2)$ and
\begin{equation}\label{eq:trace}
\E[\|f^\ad(t)\|_{L_{x,v}^2}^2]=\|f_0\|_{L_{x,v}}^2+t\sum_{k=1}^{K}\|\sigma_k\|_{L_{x,v}^2}^2.
\end{equation}
\end{proposition}

The trace formula~\eqref{eq:trace} is proved below using the two formulations~\eqref{eq:SPDE-add-solution1} and~\eqref{eq:SPDE-add-solution2}.

\begin{proof}[First proof of Proposition~\ref{propo:trace}]
Choosing $t_1=t_2=t$, $x_1=x_2=x$ and $v_1=v_2=v$ in the expression of the covariance above, and integrating with respect to the variables $x$ and $v$, one obtains
\begin{align*}
\E[\|f^\ad(t)\|_{L_{x,v}^2}^2]&=\iint\E[f^\ad(t,x,v)^2]\,\text dx\,\text dv\\
&=\iint f_0(\phi_t^{-1}(x,v))^2\,\text dx\,\text dv+\sum_{k=1}^{K}\int_{0}^{t}\iint\E[\sigma_k(\phi_{t-s}^{-1}(x,v))^2]\,\text dx\,\text dv\,\text ds\\
&=\iint f_0(x,v)^2\,\text dx\,\text dv+\sum_{k=1}^{K}\int_{0}^{t}\iint\E[\sigma_k(x,v)^2]\,\text dx\,\text dv\,\text ds\\
&=\|f_0\|_{L_{x,v}}^2+t\sum_{k=1}^{K}\|\sigma_k\|_{L_{x,v}^2}^2
\end{align*}
using the fact that for all $t\ge 0$ the mapping $\phi_t$ preserves volume in $\T^d\times\R^d$, see Section~\ref{sec:deterministic}.
\end{proof}

\begin{proof}[Second proof of Proposition~\ref{propo:trace}]
Using the mild formulation~\eqref{eq:SPDE-add-solution1} of the solution $f(t)$ of~\eqref{eq:SPDE-add}, the It\^o isometry formula in the Hilbert space $L_{x,v}^2$ (see for instance~\cite[Theorem~10.16]{LPS}), and the isometry property for the linear operators $S(t):L_{x,v}^{2}\to L_{x,v}^2$ (see Section~\ref{sec:deterministic}), for all $t\ge 0$, one has
\begin{align*}
\E[\|f^\ad(t)\|_{L_{x,v}^2}^2]&=\E[\|S(t)f_0\|_{L_{x,v}^2}^2]+\sum_{k=1}^{K}\int_{0}^{t}\|S(t-s)\sigma_k\|_{L_{x,v}^2}^2\,\text ds\\
&=\E[\|f_0\|_{L_{x,v}^2}^2]+\sum_{k=1}^{K}\int_{0}^{t}\|\sigma_k\|_{L_{x,v}^2}^2\,\text ds\\
&=\E[\|f_0\|_{L_{x,v}^2}^2]+t\sum_{k=1}^{K}\|\sigma_k\|_{L_{x,v}^2}^2.
\end{align*}
\end{proof}

\subsection{Splitting scheme}

Let us now describe the proposed numerical scheme for the temporal discretization of the SPDE~\eqref{eq:SPDE-add}. The strategy is to use a splitting method in order to treat first the deterministic part, second the stochastic part. Using the Lie--Trotter integrator~\eqref{eq:LTdeter} for the deterministic part yields the following scheme: given the initial value $f_0$ and the time-step size $\tau\in(0,1)$, set $f_0^\ad=f_0$ and for any nonnegative integer $n\ge 0$ set
\begin{equation}\label{eq:LT-add}
f_{n+1}^\ad=S^2(\tau)S^1(\tau)f_n^\ad+\delta W_n=S^2(\tau)S^1(\tau)f_n^\ad+\sum_{k=1}^{K}\delta \beta_{n,k}\sigma_k ,
\end{equation}
where the Wiener increments $\delta W_n$ and $\delta\beta_{n,k}$ are given by~\eqref{eq:WincW} and~\eqref{eq:WincBeta} respectively, see  Section~\ref{sec:notation}. The scheme~\eqref{eq:LT-add} can also be written as follows: for all $n\ge 0$, $x\in\T^d$ and $v\in\R^d$ one has
\begin{equation}\label{eq:LT-add_xv}
f_{n+1}^\ad(x,v)=f_n^\ad(x-\tau v-\tau^2E(x),v-\tau E(x))+\sum_{k=1}^{K}\sigma_k(x,v) \delta \beta_{n,k}.
\end{equation}

The main result of this section states that the Lie--Trotter splitting scheme~\eqref{eq:LT-add} preserves the trace formula from Proposition~\ref{propo:trace}
for all times and for any value of the time-step size $\tau$.
\begin{proposition}\label{propo:tracenum}
Assume that $\sigma_k\in L_{x,v}^2$ for all $1\le k\le K$, and that $f_0\in L_{x,v}^2$. Let $\bigl(f_n^\ad\bigr)_{n\ge 0}$ be given by the Lie--Trotter splitting
scheme~\eqref{eq:LT-add} with time-step size $\tau\in(0,1)$. Then for any nonnegative integer $n\ge 0$, one has $f_n^\ad\in L^2(\Omega, L_{x,v}^2)$ and
\begin{equation}\label{eq:tracenum}
\E[\|f_n^\ad\|_{L_{x,v}^2}^2]=\|f_0\|_{L_{x,v}}^2+t_n\sum_{k=1}^{K}\|\sigma_k\|_{L_{x,v}^2}^2,
\end{equation}
where $t_n=n\tau$.
\end{proposition}

To obtain the trace formula~\eqref{eq:tracenum}, it suffices to prove that for any nonnegative integer $n\ge 0$ one has
\[
\E[\|f_{n+1}^\ad\|_{L_{x,v}^2}^2]=\E[\|f_n^\ad\|_{L_{x,v}^2}^2]+\tau \sum_{k=1}^{K}\|\sigma_k\|_{L_{x,v}^2}^2.
\]
Like for Proposition~\ref{propo:trace}, two proofs are given below.
\begin{proof}[First proof of Proposition~\ref{propo:tracenum}]
Observe that for all $x\in\T^d$ and $v\in\R^d$ the random variable $f_n^\ad(x-\tau v-\tau^2E(x),v-\tau E(x))$ and the Gaussian random variable $\bigl(\delta\beta_{n,k}\bigr)_{1\le k\le K}$ are independent. Using the expression~\eqref{eq:LT-add_xv}, one obtains, for all $(x,v)\in\T^d\times\R^d$, the identity
\[
\E[f_{n+1}^\ad(x,v)^2]=\E[f_n^\ad(x-tv-t^2E(x),v-tE(x))^2]+\tau\sum_{k=1}^{K}\sigma_k(x,v)^2.
\]
Note that applying twice a change of variables formulas one has
\begin{align*}
\iint\E[f_n^\ad(x-\tau v-\tau^2E(x),v-\tau E(x))^2]\,\text dx\,\text dv&=\iint\E[f_n^\ad(x-\tau v,v)^2]\,\text dx\,\text dv\\
&=\iint\E[f_n^\ad(x,v)^2]\,\text dx\,\text dv.
\end{align*}
Integrating with respect to the variables $x$ and $v$, one then obtains
\begin{align*}
\E[\|f_{n+1}^\ad\|_{L_{x,v}^2}^2]&=\iint\E[f_{n+1}^\ad(x,v)^2]\,\text dx\,\text dv\\
%&=\iint\E[f_n^\ad(x-\tau v-\tau^2E(x),v-\tau E(x))^2]\,\text dx\,\text dv+\tau\sum_{k=1}^{K}\iint\sigma_k(x,v)^2\,\text dx\,\text dv\\
&=\iint\E[f_n^\ad(x,v)^2]\,\text dx\,\text dv+\tau\sum_{k=1}^{K}\iint\sigma_k(x,v)^2\,\text dx\,\text dv\\
&=\E[\|f_n^\ad\|_{L_{x,v}^2}^2]+\tau \sum_{k=1}^{K}\|\sigma_k\|_{L_{x,v}^2}^2.
\end{align*}
\end{proof}

\begin{proof}[Second proof of Proposition~\ref{propo:tracenum}]
Observe that the random mapping $f_n^\ad$ and the Gaussian random variable $\bigl(\delta\beta_{n,k}\bigr)_{1\le k\le K}$ are independent. Using the isometry property for the linear operators $S^1(\tau),S^2(\tau):L_{x,v}^2\to L_{x,v}^2$, one then obtains
\begin{align*}
\E[\|f_{n+1}^\ad\|_{L_{x,v}^2}^2]&=\E[\|S^2(\tau)S^1(\tau)f_n^\ad\|_{L_{x,v}^2}^2]+\tau \sum_{k=1}^{K}\|\sigma_k\|_{L_{x,v}^2}^2\\
&=\E[\|f_n^\ad\|_{L_{x,v}^2}^2]+\tau \sum_{k=1}^{K}\|\sigma_k\|_{L_{x,v}^2}^2.
\end{align*}
\end{proof}

\subsection{Numerical experiments}

We begin the numerical experiments by illustrating the behavior of the linear Vlasov equation perturbed by additive noise~\eqref{eq:SPDE-add}, in dimension $d=1$. Like for Figure~\ref{fig:snap-det} in the deterministic case (see Section~\ref{sec:deterministic}), the initial value $f_0$ is given by~\eqref{eq:f0} and the vector field $E$ is given by~\eqref{eq:E}.

The noise perturbation is given either by
\begin{equation}\label{eq:sigma1cos}
\sigma_1(x,v)=\cos(v)\mathds{1}_{|v|\le 3}~,\quad x\in\T, v\in\R,
\end{equation}
or by
\begin{equation}\label{eq:sigma1sin}
\sigma_1(x,v)=\sin(v)\mathds{1}_{|v|\le 3}~,\quad x\in\T, v\in\R,
\end{equation}
with $K=1$ in both cases. In Figures~\ref{fig:snap-add-cos} and~\ref{fig:snap-add-sin} below, snapshots at the times $\{0,0.5,1,1.5,2,2.5\}$ of the numerical solution computed using the Lie--Trotter splitting scheme~\ref{eq:LT-add} are displayed, with $\sigma_1$ given by~\eqref{eq:sigma1cos} and~\eqref{eq:sigma1sin} respectively. The discretization parameters are given by $\delta x=\frac{1}{500}$, $\delta v=\frac{4\pi}{500}$, and $\tau=0.1$. One observes that the solutions behave differently from the deterministic case displayed in Figure~\ref{fig:snap-det}. One also observes major differences between Figures~\ref{fig:snap-add-cos} and~\ref{fig:snap-add-sin} which are due to imposing a noise perturbation which is either symmetric or skew-symmetric with respect to the velocity variable $v$. Recall that the average value $\E[f^\ad(t,x,v)]=f^\de(t,x,v)$ is solution of the deterministic PDE~\eqref{eq:Vlasov}, which justifies the persistence of the deterministic behavior and of the influence of the initial condition in the snapshots.

\begin{figure}[h]
\begin{subfigure}{.3\textwidth}
  \centering
\includegraphics[width=1\linewidth]{VlasovInit-eps-converted-to.pdf}
\end{subfigure}%
\begin{subfigure}{.3\textwidth}
  \centering
\includegraphics[width=1\linewidth]{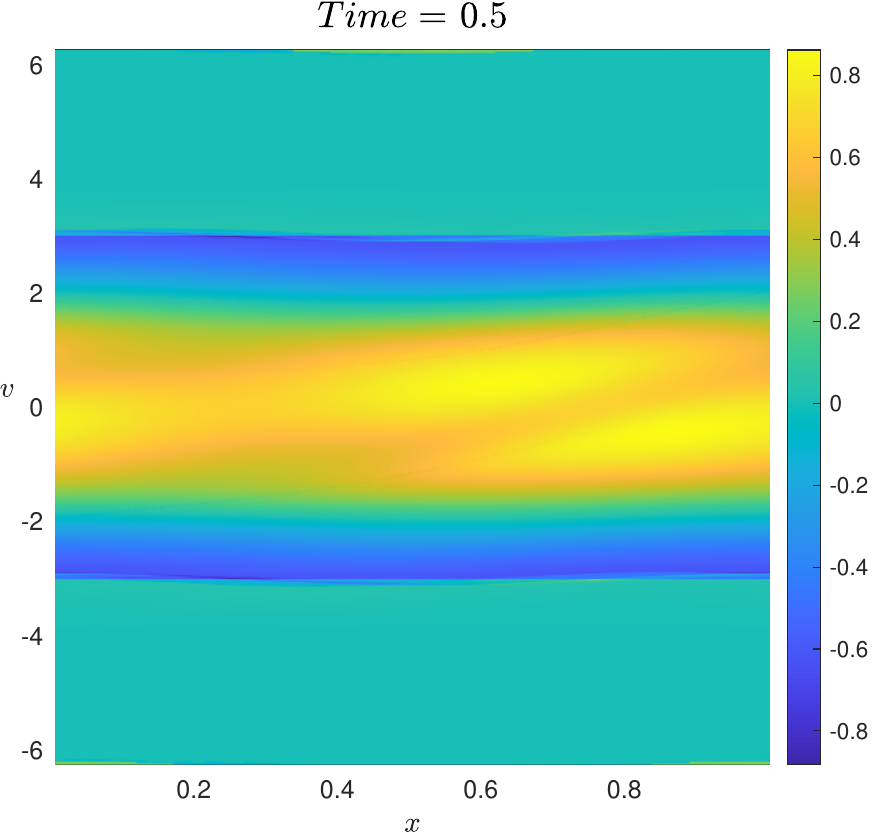}
\end{subfigure}%
\begin{subfigure}{.3\textwidth}
  \centering
  \includegraphics[width=1\linewidth]{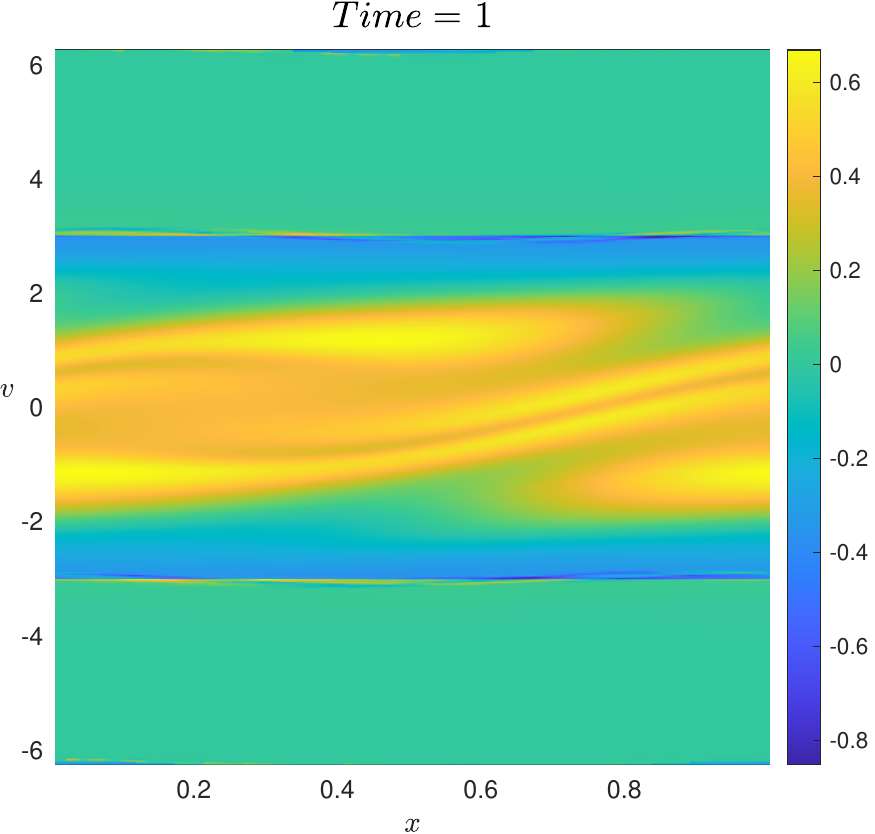}
\end{subfigure}\\[1ex]%
\begin{subfigure}{.3\textwidth}
  \centering
  \includegraphics[width=1\linewidth]{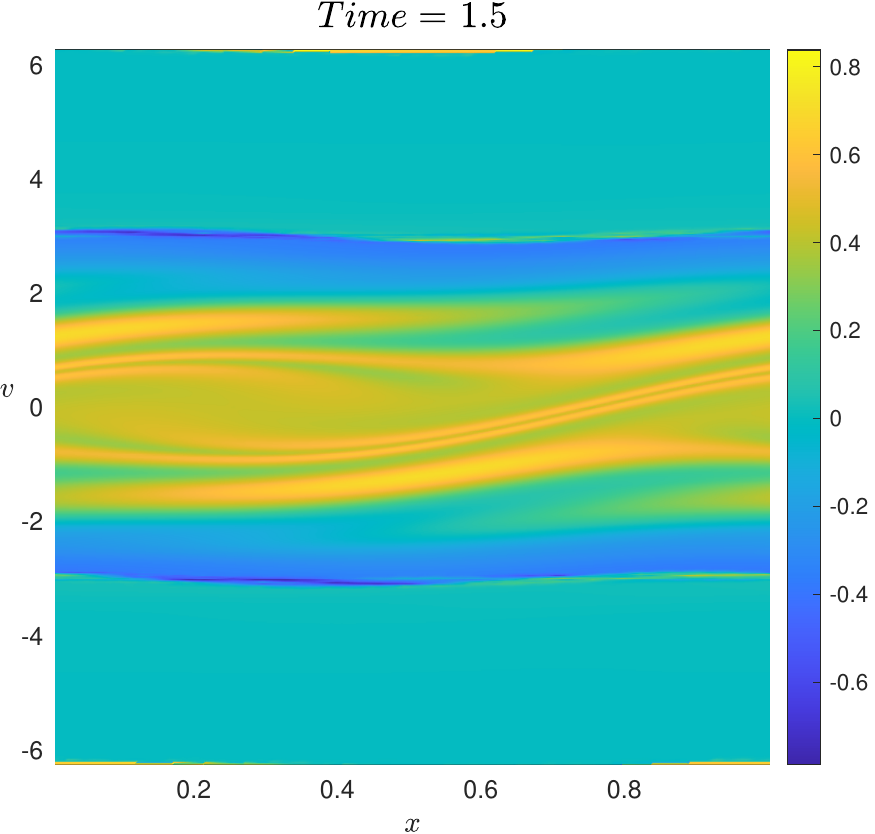}
\end{subfigure}
\begin{subfigure}{.3\textwidth}
  \centering
  \includegraphics[width=1\linewidth]{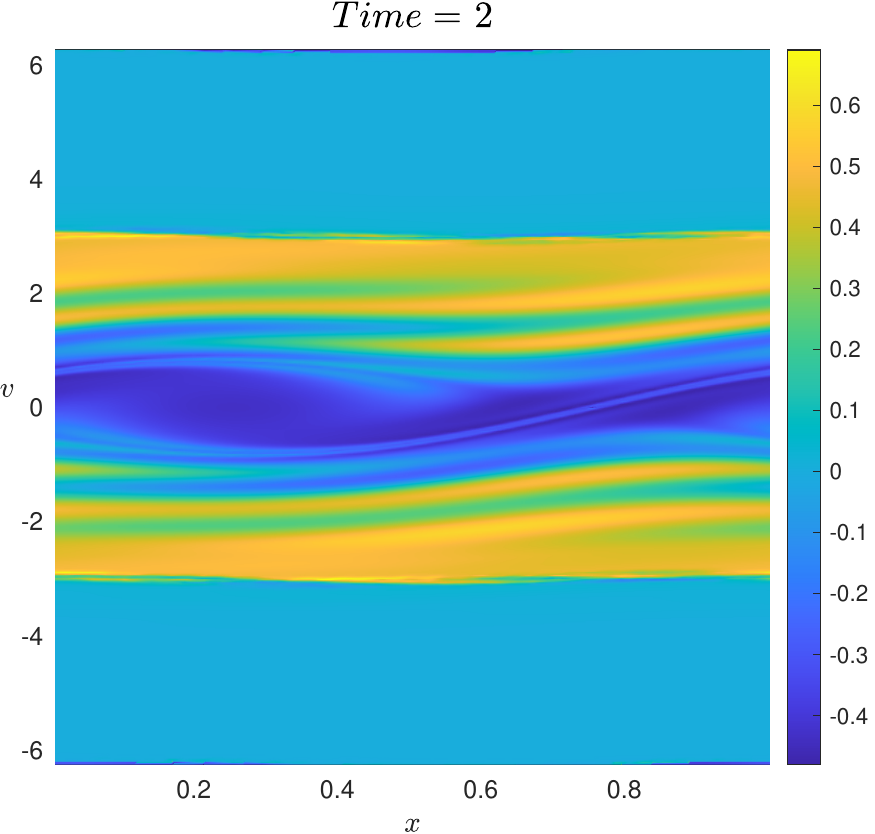}
\end{subfigure}
\begin{subfigure}{.3\textwidth}
  \centering
  \includegraphics[width=1\linewidth]{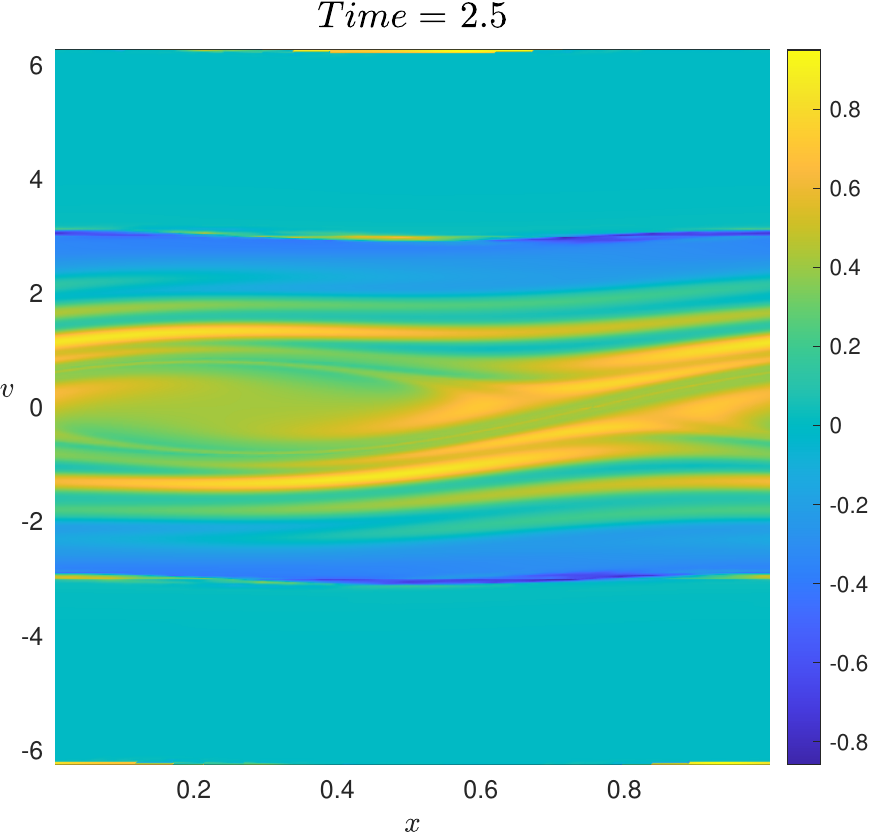}
\end{subfigure}
\caption{Snapshots: approximation of the solution of the stochastic PDE with additive noise~\eqref{eq:SPDE-add} with initial value $f_0$ given by~\eqref{eq:f0}, with $\sigma_1$ given by~\eqref{eq:sigma1cos} at times~$\{0,0.5,1,1.5,2,2.5\}$, using the Lie--Trotter splitting scheme~\eqref{eq:LT-add} with time-step size $\tau=0.1$.}
\label{fig:snap-add-cos}
\end{figure}

\begin{figure}[h]
\begin{subfigure}{.3\textwidth}
  \centering
\includegraphics[width=1\linewidth]{VlasovInit-eps-converted-to.pdf}
\end{subfigure}%
\begin{subfigure}{.3\textwidth}
  \centering
  \includegraphics[width=1\linewidth]{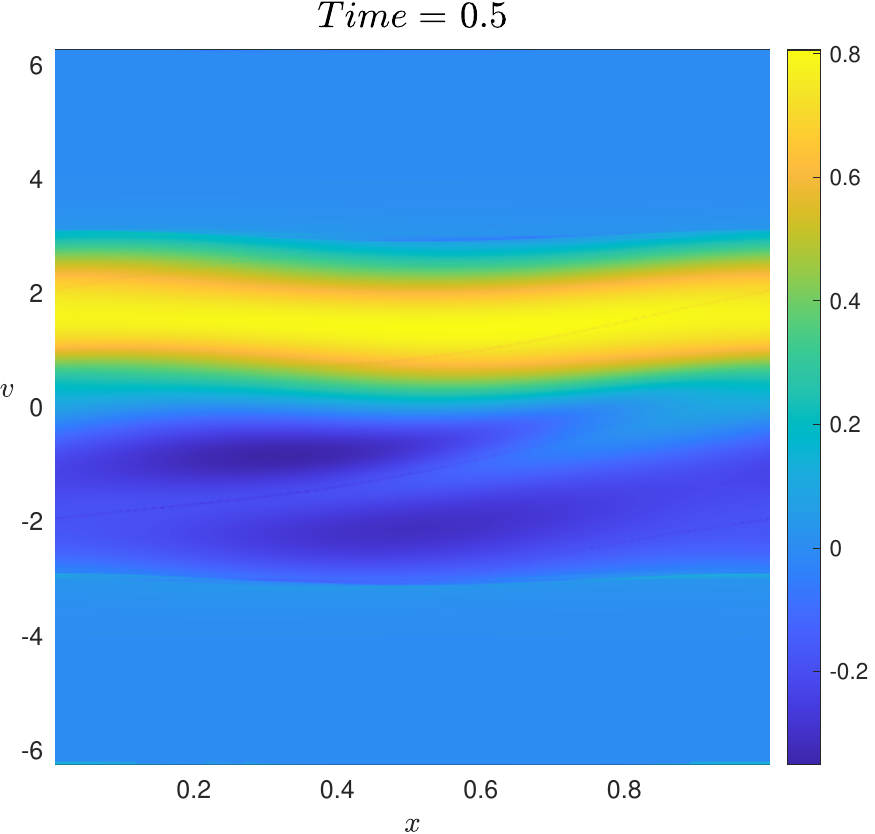}
\end{subfigure}%
\begin{subfigure}{.3\textwidth}
  \centering
  \includegraphics[width=1\linewidth]{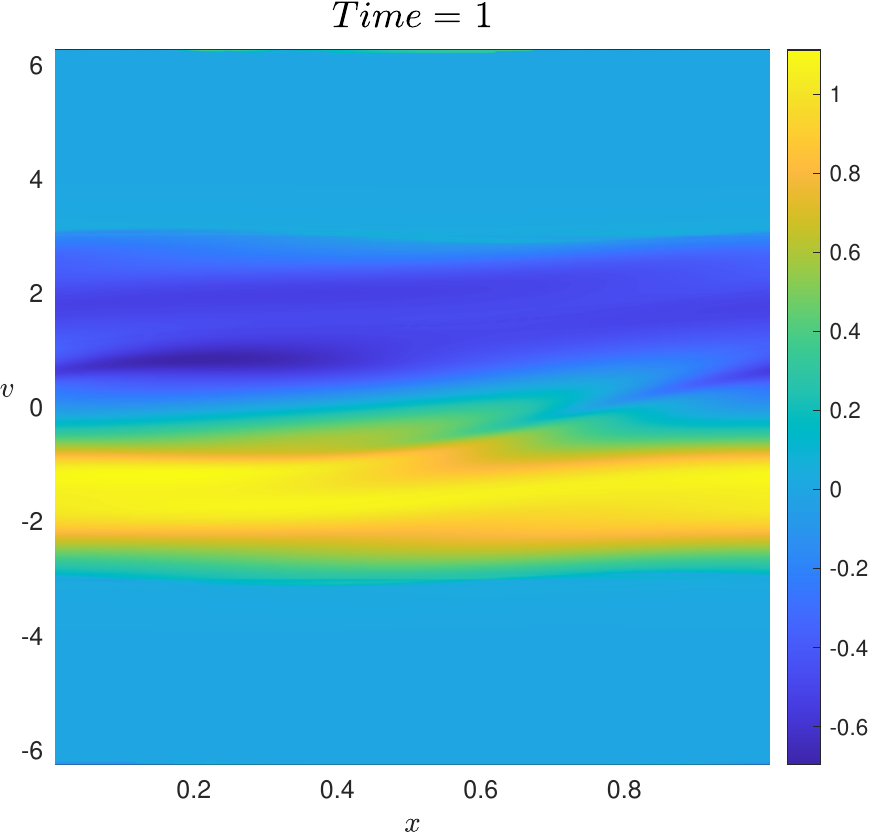}
\end{subfigure}\\[1ex]%
\begin{subfigure}{.3\textwidth}
  \centering
  \includegraphics[width=1\linewidth]{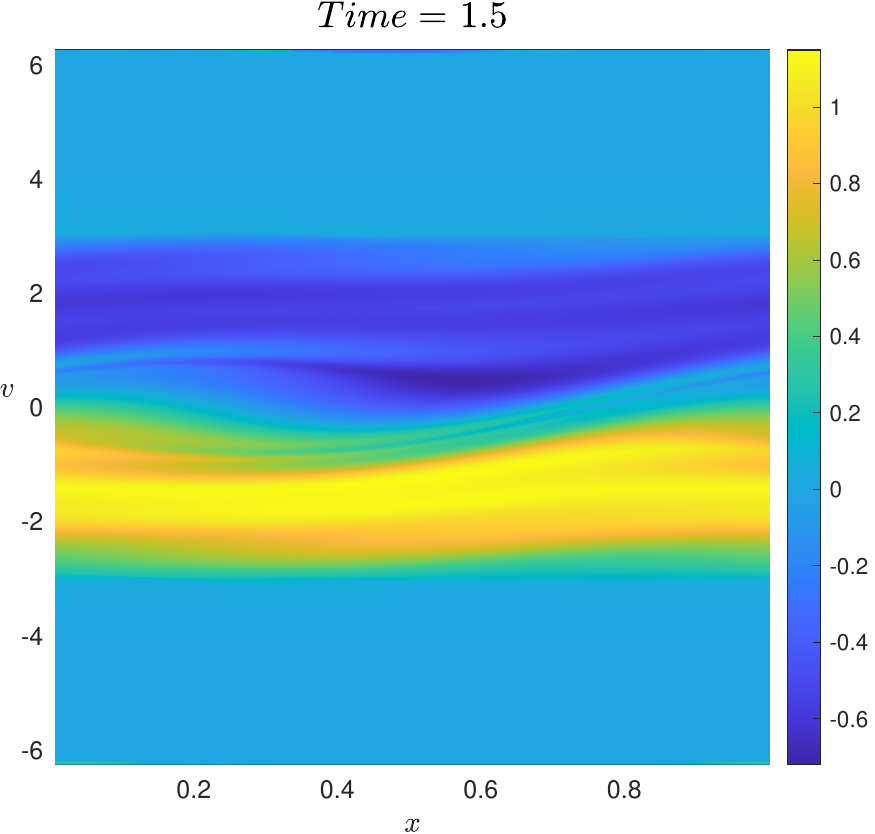}
\end{subfigure}
\begin{subfigure}{.3\textwidth}
  \centering
  \includegraphics[width=1\linewidth]{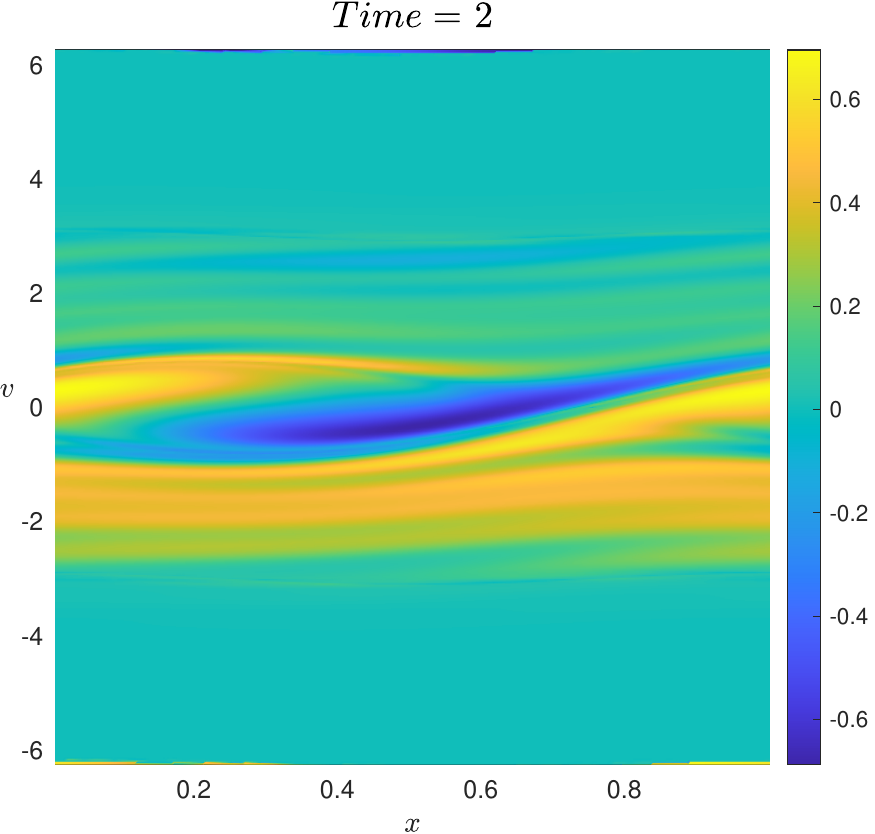}
\end{subfigure}
\begin{subfigure}{.3\textwidth}
  \centering
  \includegraphics[width=1\linewidth]{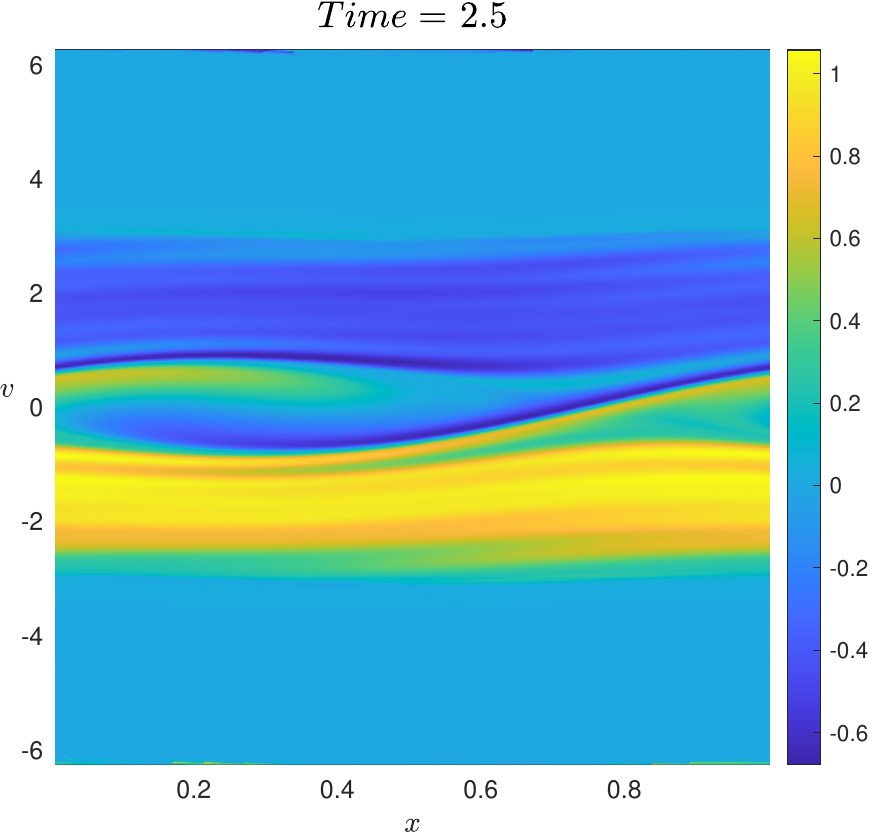}
\end{subfigure}
\caption{Snapshots: approximation of the solution of the stochastic PDE with additive noise~\eqref{eq:SPDE-add} with initial value $f_0$ given by~\eqref{eq:f0}, with $\sigma_1$ given by~\eqref{eq:sigma1sin} at times~$\{0,0.5,1,1.5,2,2.5\}$, using the Lie--Trotter splitting scheme~\eqref{eq:LT-add} with time-step size $\tau=0.1$.}
\label{fig:snap-add-sin}
\end{figure}

We continue these numerical experiments with Figure~\ref{fig:trace-add} in order to illustrate the trace formula~\eqref{eq:tracenum} from Proposition~\ref{propo:tracenum} for the $L^2_{x,v}$-norm of the Lie--Trotter splitting scheme~\eqref{eq:LT-add}. Let $d=1$ and $T=1$, the initial value $f_0$ is given by~\eqref{eq:f0} and the vector field $E$ is given by~\eqref{eq:E}. For the noise perturbation, one has either $K=1$ and
\begin{equation}\label{eq:sigma-addtrace1}
\sigma_1(x,v)=0.5\sin(v)\mathds{1}_{|v|\le 3},
\end{equation}
or $K=2$ and
\begin{equation}\label{eq:sigma-addtrace2}
\sigma_1(x,v)=0.5\e^{-v^2/2}\cos{(2\pi x)}~,\quad \sigma_2(x,v)=0.5\e^{-v^2/2}\sin{(2\pi x)},
\end{equation}
for all $x\in\T$ and $v\in\R$. The discretization parameters are given by $\delta x=\frac{1}{200}$, $\delta v=\frac{4\pi}{400}$, and $\tau=0.1$. The expectation in the trace formula~\eqref{eq:tracenum} is approximated by a standard Monte Carlo averaging procedure over $10^6$ independent samples. The exact line in Figure~\ref{fig:trace-add} corresponds to the trace formula~\eqref{eq:trace} from Proposition~\ref{propo:trace} satisfied by the exact solution. Even if Proposition~\ref{propo:tracenum} states that the Lie--Trotter splitting scheme~\eqref{eq:LT-add} preserves the trace formula at all times $t_n=n\tau$, some error is visible. This may be due to the discretization procedure with respect to the spatial and velocity variables $x$ and $v$, in particular since a truncation procedure for large $v$ is applied. We have verified that increasing the Monte Carlo sample size does not seem to reduce the error visible in Figure~\ref{fig:trace-add}.

\begin{figure}[h]
\begin{subfigure}{.5\textwidth}
  \centering
  \includegraphics[width=.9\linewidth]{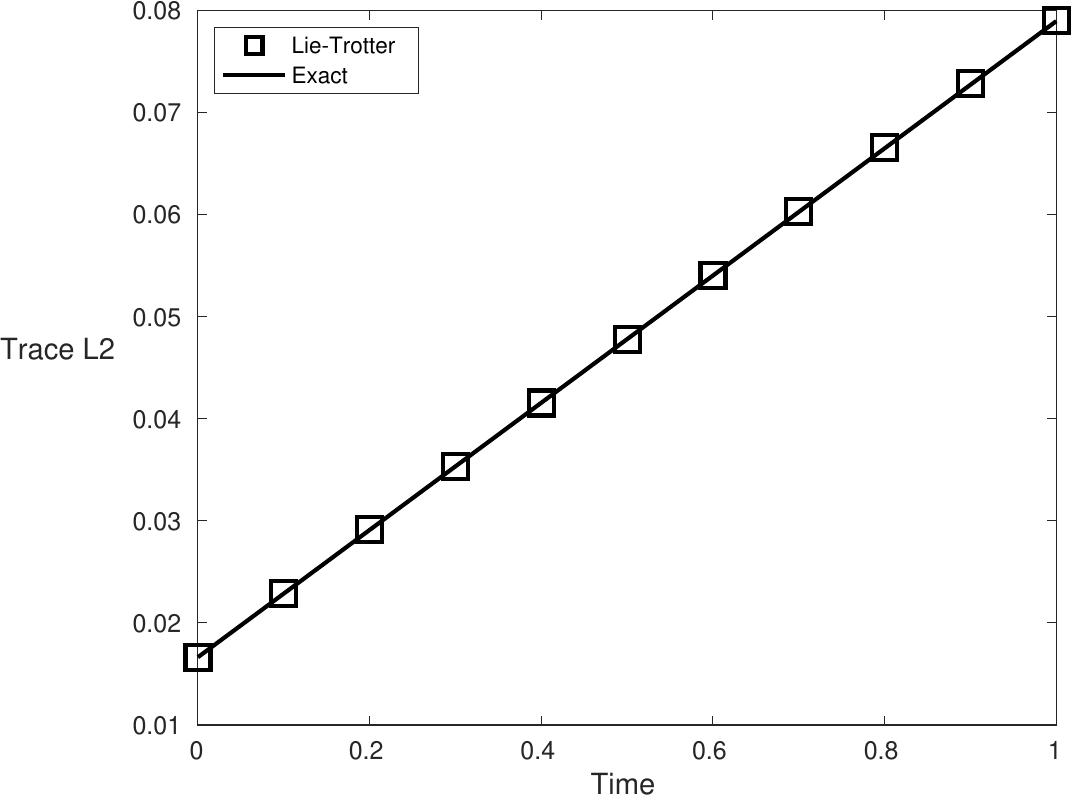}
  \caption{Noise given by~\eqref{eq:sigma-addtrace1}, $K=1$.}
\end{subfigure}%
\begin{subfigure}{.5\textwidth}
  \centering
  \includegraphics[width=.9\linewidth]{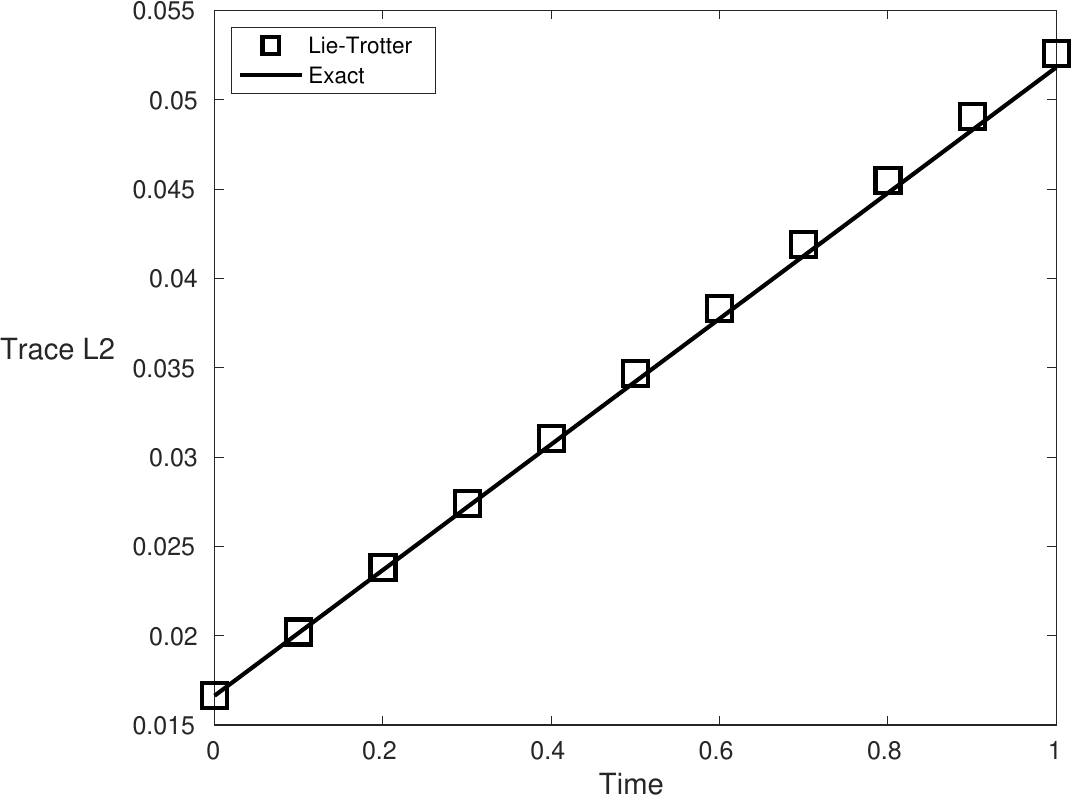}
  \caption{Noise given by~\eqref{eq:sigma-addtrace2}; $K=2$.}
\end{subfigure}
\caption{Trace formula: illustration of Proposition~\ref{propo:tracenum} when applying the Lie--Trotter splitting scheme~\eqref{eq:LT-add} to the SPDE with additive noise~\eqref{eq:SPDE-add} with time-step size $\tau=0.1$.}
\label{fig:trace-add}
\end{figure}

The final experiment in the additive noise case is devoted to investigate the mean-square order of convergence of the Lie--Trotter splitting scheme~\eqref{eq:LT-add}. In Figure~\ref{fig:ms-add}, a loglog plot displays how the error
\[
\sqrt{\sup_{x,v}\E[|f^{\ad}(x,v)-f^{\ad, \rm ref}(x,v)|^2]}
\]
converges to $0$ when $\tau$ decreases, where $x,v$ are grid points, the reference solution $f^{\ad, \rm ref}$ is computed using the splitting scheme with time-step size $\tau_{\rm ref}=2^{-14}$. The time-step size $\tau$ takes values in $\{2^{-7},\ldots,2^{-13}\}$, and the expectation is computed using a Monte Carlo averaging procedure over $500$ independent samples. Note that the sample size used to illustrate the behaviour of the mean-square error is much smaller than the sample size used for the illustration of the trace formula in Figure~\ref{fig:trace-add} above. This is due to the fact that the variance also decreases when $\tau$ decreases. We have verified that increasing the Monte Carlo sample size does not significantly modify the behaviour of the mean-square error observed below. The discretization parameters are $\delta x=\frac{1}{100}$, $\delta v=\frac{4\pi}{200}$. The final time is given by $T=0.5$, whereas $f_0$ and $E$ are again given by~\eqref{eq:f0} and~\eqref{eq:E} respectively. Like for Figure~\ref{fig:trace-add} above, the noise is given by~\eqref{eq:sigma-addtrace1} or~\eqref{eq:sigma-addtrace2}. Based on these numerical experiments, we conjecture that the order of mean-square convergence of the scheme~\eqref{eq:LT-add} is equal to $1$. We leave the rigorous verification of this conjecture for future works.

\begin{figure}[h]
\begin{subfigure}{.5\textwidth}
  \centering
  \includegraphics[width=.9\linewidth]{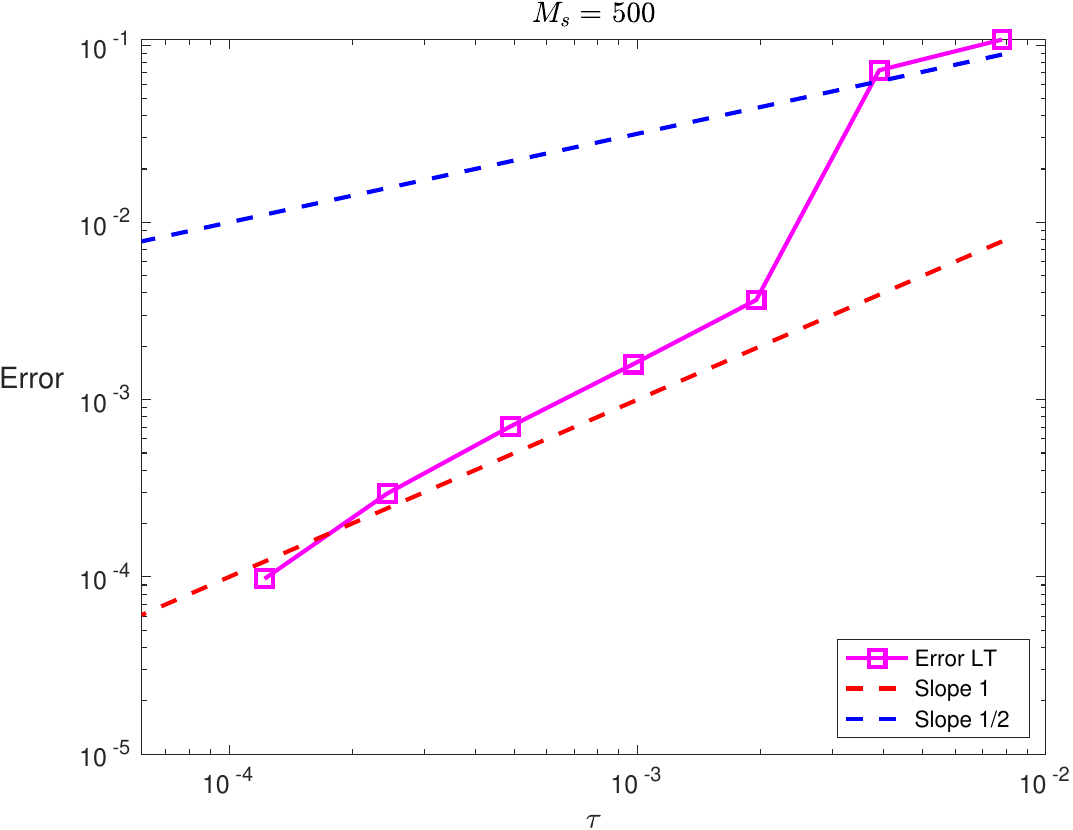}
  \caption{Noise given by~\eqref{eq:sigma-addtrace1}, $K=1$.}
\end{subfigure}%
\begin{subfigure}{.5\textwidth}
  \centering
  \includegraphics[width=.9\linewidth]{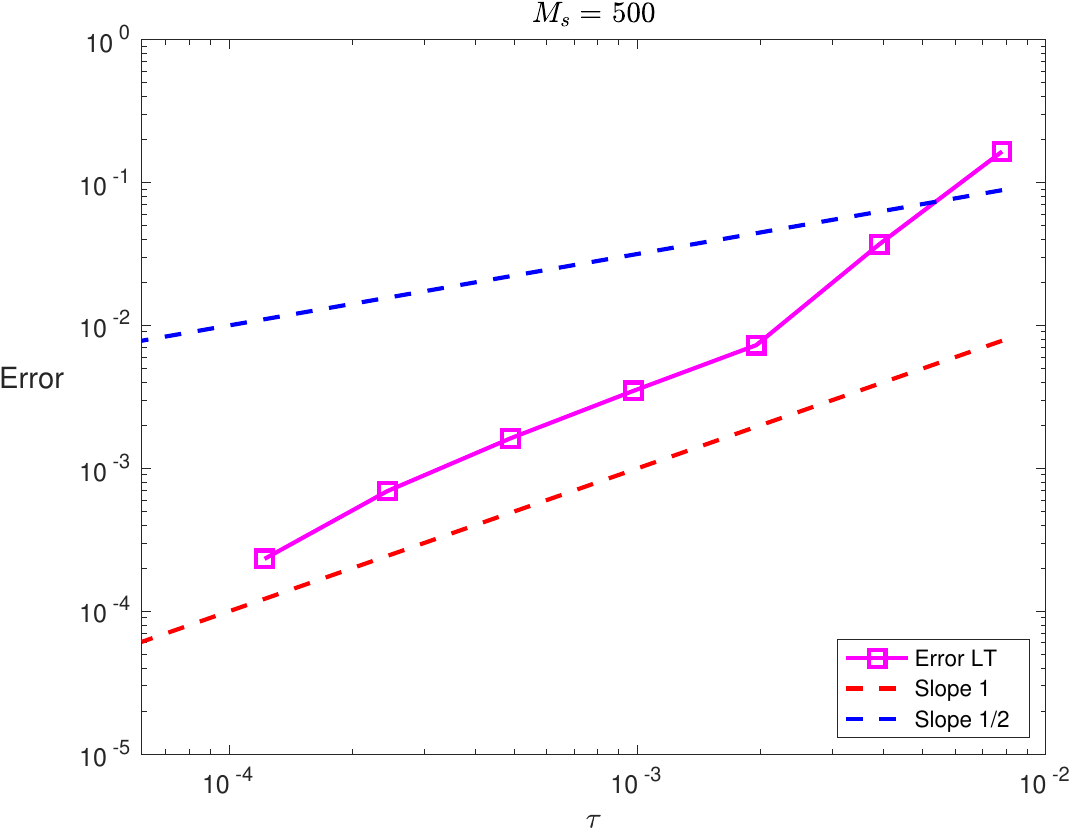}
  \caption{Noise given by~\eqref{eq:sigma-addtrace2}, $K=2$.}
\end{subfigure}
 \caption{Mean-square errors: Lie--Trotter scheme~\eqref{eq:LT-add} applied to the SPDE with additive noise~\eqref{eq:SPDE-add} driven by one-dimensional noise ($K=1$, left) and two-dimensional noises ($K=2$, right).}
\label{fig:ms-add}
\end{figure}

\section{The stochastic linear Vlasov equation perturbed by multiplicative noise}\label{sec:multiplicative}

In this section, we consider stochastic perturbations of the Vlasov equation~\eqref{eq:Vlasov} where the noise is multiplicative. In the analysis and applications of stochastic (partial) differential equations, it is well-known that several interpretations of multiplicative noise perturbations are possible. The It\^o interpretation of the noise is considered in Section~\ref{sec:multiplicativeIto}, then the Stratonovich interpretation is considered in Section~\ref{sec:multiplicativeStrato}. The objective of this section is to explain how to construct numerical schemes
which are consistent with the two possible interpretations of the multiplicative noise, and to investigate which properties of the exact solution can be preserved at the discrete level.

In this section, the following condition is imposed: for all $1\le k\le K$, the mapping $(x,v)\mapsto \sigma_k(x,v)$ is bounded. Some arguments below hold assuming in addition
that there exists a real number $\sigma$ such
that the mapping $\sum_{k=1}^{K}\sigma_{k}^2$ is constant equal to $\sigma^2$:
\begin{equation}\label{eq:conditionsigma2constant}
\sum_{k=1}^{K}\sigma_k(x,v)^2=:\sigma^2~,\quad \forall (x,v)\in\T^d\times\R^d.
\end{equation}
For instance the condition~\eqref{eq:conditionsigma2constant} is satisfied if $d=1$, $K=2$, $\sigma_1(x,v)=\cos(2\pi x)$ and $\sigma_2(x,v)=\sin(2\pi x)$, with $\sigma=1$. Note that this two-dimensional noise leads to a different dynamics for the SPDE than the case $K=1$ and $\sigma_1(x,v)=\sigma=1$.

\subsection{It\^o interpretation}\label{sec:multiplicativeIto}

We consider the following linear Vlasov equation driven by multiplicative noise interpreted in the It\^o sense: for $t\ge 0$, $x\in\T^d$, $v\in \R^d$
\begin{equation}\label{eq:SPDE-mult-Ito}
\left\lbrace
\begin{aligned}
&\text df^\mI(t,x,v)+v\cdot\nabla_xf^\mI(t,x,v)\,\text dt+E(x)\cdot\nabla_v f^\mI(t,x,v)\,\text dt=f^\mI(t,x,v)\,\text dW(t,x,v),\\
&f^\mI(0,x,v)=f_0(x,v)~,
\end{aligned}
\right.
\end{equation}
where the noise is defined by~\eqref{eq:Wtxv}.

\subsubsection{Analysis and properties of the problem}
Like in the additive noise case (Section~\ref{sec:additive}), let us consider several ways to define and deal with solutions of the SPDE~\eqref{eq:SPDE-mult-Ito}.

On the one hand, using the group $\bigl(S(t)\bigr)_{t\in\R}$ of linear operators given by~\eqref{eq:S} in Section~\ref{sec:deterministic}, one can consider mild solutions: for all $t\ge 0$, one has
\begin{equation}\label{eq:SPDE-mult-Ito-solution1}
\begin{aligned}
f^\mI(t)&=S(t)f_0+\int_{0}^{t}S(t-s)f^\mI(s) \,\text dW(s)\\
&=S(t)f_0+\sum_{k=1}^{K}\int_{0}^{t}S(t-s)\bigl(f^\mI(s)\sigma_k\bigr)\,\text d\beta_k(s).
\end{aligned}
\end{equation}
On the other hand, using the expression~\eqref{eq:S} for the linear operator $S(t)$, one has for all $t\ge 0$, $x\in\T^d$ and $v\in\R^d$
\begin{equation}\label{eq:SPDE-mult-Ito-solution2}
f^\mI(t,x,v)=f_0(\phi_t^{-1}(x,v))+\sum_{k=1}^{K}\int_{0}^{t}f^\mI(s,\phi_{t-s}^{-1}(x,v))\sigma_k(\phi_{t-s}^{-1}(x,v))\,\text d\beta_k(s).
\end{equation}
Finally, the connection with the ordinary differential equation~\eqref{eq:ODE} can also be seen by applying the It\^o--Wentzell formula (see \autoref{sec:ItoWentzell}): if $t\mapsto (x_t,v_t)=\phi_t(x_0,v_0)$ is the solution of the ordinary differential equation~\eqref{eq:ODE} with arbitrary initial value $(x_0,v_0)\in\T^d\times\R^d$, and if the solution of~\eqref{eq:SPDE-mult-Ito} is sufficiently regular, then the stochastic process $t\ge 0\mapsto f^\mI(t,x_t,v_t)$ satisfies
\begin{equation}\label{eq:SPDE-mult-Ito-solutionIW}
\text df^\mI(t,x_t,v_t)=\sum_{k=1}^{K}f^\mI(t,x_t,v_t)\sigma_k(x_t,v_t)\,\text d\beta_k(t).
\end{equation}
The formula~\eqref{eq:SPDE-mult-Ito-solutionIW} allows to retrieve the expression~\eqref{eq:SPDE-mult-Ito-solution2} of $f^\mI(t,x,v)$ above by writing $(x_t,v_t)=\phi_t(x,v)$.

Let us now describe the properties of the solutions of the SPDE~\eqref{eq:SPDE-mult-Ito}.
\begin{proposition}\label{propo:multIto}
Let $\bigl(f^\mI(t)\bigr)_{t\ge 0}$ be the solution of the SPDE~\eqref{eq:SPDE-mult-Ito} with (non-random) initial value $f_0$. One has the following properties.
\begin{itemize}
\item \emph{Preservation of positivity}. Assume that $f_0(x,v)\ge 0$ for all $(x,v)\in\T^d\times\R^d$.
Then one has $f^\mI(t,x,v)\ge 0$ almost surely for all $t\ge 0$ and $(x,v)\in\T^d\times\R^d$.
\item \emph{Preservation of the expected mass}. Assume that $f_0\in L_{x,v}^1$. Then almost surely one has $f^\mI(t)\in L_{x,v}^1$ for all $t\ge 0$, and
\[
\iint \E[f^\mI(t,x,v)]\,\text dx\,\text dv=\iint f_0(x,v)\,\text dx\,\text dv.
\]
\item \emph{Evolution law for the $L^2$ norm}. Assume that $f_0\in L_{x,v}^2$ and that the condition~\eqref{eq:conditionsigma2constant} is satisfied.
Then one has $f^\mI(t)\in L^2(\Omega,L_{x,v}^2)$ for all $t\ge 0$, and
\[
\E[\|f^\mI(t)\|_{L_{x,v}^2}^2]=\e^{\sigma^2 t}\|f_0\|_{L_{x,v}^2}^2.
\]

\end{itemize}
\end{proposition}

\begin{proof}[Proof of Proposition~\ref{propo:multIto}]
\leavevmode
\begin{itemize}
\item Owing to the expression~\eqref{eq:SPDE-mult-Ito-solutionIW} of the solution and applying a comparison principle for solutions of stochastic differential equations, one has $f^\mI(t,x_t,v_t)\ge 0$ almost surely for all $t\ge 0$, and for any arbitrary initial value $(x_0,v_0)\in\T^d\times\R^d$. Choosing $(x_0,v_0)=\phi_t^{-1}(x,v)$ then yields $f^\mI(t,x,v)\ge 0$ almost surely for all $(t,x,v)\in \R^+\times\T^d\times\R^d$.

\item Owing to the expression~\eqref{eq:SPDE-mult-Ito-solution2} of the solution using the flow $\bigl(\phi_t\bigr)_{t\in\R}$, one has for all $t\ge 0$
\begin{align*}
\iint \E[f^{\mI}(t,x,v)]&\,\text dx\,\text dv=\iint f_0(\phi_t^{-1}(x,v))\,\text dx\,\text dv\\
&\quad+\sum_{k=1}^{K}\iint \E\Bigl[\int_{0}^{t}\sigma_k(\phi_{t-s}^{-1}(x,v))f^\mI(s,\phi_{t-s}^{-1}(x,v))\,\text d\beta_k(s)\Bigr]\,\text dx\,\text dv\\
&=\iint f_0(x,v)\,\text dx\,\text dv,
\end{align*}
since the expectation of the stochastic It\^o integral vanishes and since $\phi_{t}$ preserves the volume in $\T^d\times\R^d$ for all $t\ge 0$.
\item Owing to the expression~\eqref{eq:SPDE-mult-Ito-solution2} of the solution using the flow $\bigl(\phi_t\bigr)_{t\in\R}$, applying It\^o's isometry formula yields
\begin{align*}
\iint \E&[f^\mI(t,x,v)^2]\,\text dx\,\text dv=\iint f_0(\phi_t^{-1}(x,v))^2\,\text dx\,\text dv\\
&\quad+\sum_{k=1}^{K}\int_{0}^{t}\iint \E\bigl[\sigma_k(\phi_{t-s}^{-1}(x,v))^2f^{\mI}(s,\phi_{t-s}^{-1}(x,v))^2\bigr]\,\text dx\,\text dv\,\text d s\\
&=\iint f_0(x,v)^2\,\text dx\,\text dv+\sum_{k=1}^{K}\int_{0}^{t}\iint \E\bigl[\sigma_k(x,v)^2f^{\mI}(s,x,v)^2\bigr]\,\text dx\,\text dv\,\text d s,
\end{align*}
since the mapping $\phi_{t-s}$ preserves the volume in $\T^d\times\R^d$ for all $t\ge s\ge 0$. Using the condition~\eqref{eq:conditionsigma2constant} then yields the identity
\begin{align*}
\iint \E[f^\mI(t,x,v)^2]\,\text dx\,\text dv&=\iint f_0(\phi_t^{-1}(x,v))^2\,\text dx\,\text dv\\
&\quad+\sigma^2\int_{0}^{t}\iint \E\bigl[f^\mI(s,x,v)^2\bigr]
\,\text dx\,\text dv\,\text d s.
\end{align*}
This implies that for all $t\ge 0$ one has
\[
\E[\|f^\mI(t)\|_{L_{x,v}^2}^2]=\e^{\sigma^2 t}\|f_0\|_{L_{x,v}^2}^2.
\]
\end{itemize}
\end{proof}

\begin{remark}
The evolution law for the $L^2$ norm in Proposition~\ref{propo:multIto} can be proved by an alternative approach: using the expression~\eqref{eq:SPDE-mult-Ito-solution1} of the mild solution, and using the It\^o isometry formula in the Hilbert space $L_{x,v}^2$, one has
\begin{align*}
\E[\|f^\mI(t)\|_{L_{x,v}^2}^2]&=\|S(t)f_0\|_{L_{x,v}^2}^2+\sum_{k=1}^{K}\int_{0}^{t}\E[\|S(t-s)\bigl(f^\mI(s)\sigma_k\bigr)\|_{L_{x,v}^2}^2]\,\text ds\\
&=\|f_0\|_{L_{x,v}^2}^2+\sum_{k=1}^{K}\int_{0}^{t}\E[\|f^\mI(s)\sigma_k\|_{L_{x,v}^2}^2]\,\text ds,
\end{align*}
since the linear operator $S(t-s):L_{x,v}^2\to L_{x,v}^2$ is an isometry, for all $t\ge s\ge 0$. When the condition~\eqref{eq:conditionsigma2constant} is satisfied, one has
\[
\sum_{k=1}^{K}\|f^\mI(s)\sigma_k\|_{L_{x,v}^2}^2=\sigma^2\|f^\mI(s)\|_{L_{x,v}^2}^2
\]
and the conclusion is obtained as in the proof above.
\end{remark}

\subsubsection{Splitting scheme}
Let us now describe the proposed numerical scheme for the temporal discretization of the SPDE~\eqref{eq:SPDE-mult-Ito} driven by multplicative It\^o noise.
Like in the additive noise case presented in Section~\ref{sec:additive}, a Lie--Trotter splitting strategy is applied. The treatment of the deterministic part is not modified.  The auxiliary stochastic subsystem with multiplicative It\^o noise
\[
\text df(t,x,v)=f(t,x,v)\,\text dW(t,x,v)~,\quad (t,x,v)\in\R^+\times\T^d\times\R^d
\]
is solved exactly: for all $t\ge s\ge 0$, one has
\[
f(t,x,v)=\e^{\sum_{k=1}^{K}\sigma_{k}(\beta_k(t)-\beta_k(s))-\frac{(t-s)\sum_{k=1}^{K}\sigma_k(x,v)^2}{2}}f(s,x,v)~,\quad \forall~(x,v)\in\T^d\times\R^d.
\]
Using the Lie--Trotter integrator~\eqref{eq:LTdeter} for the deterministic part and combining the discretizations of the deterministic and stochastic parts yields the following scheme: given the initial value $f_0$ and the time-step size $\tau\in(0,1)$, set $f_0^\mI=f_0$ and for any nonnegative integer $n\ge 0$ set
\begin{equation}\label{eq:LT-mult-Ito}
\left\lbrace
\begin{aligned}
&\hat{f}_{n+1}^\mI=S^2(\tau)S^1(\tau)f_n^\mI\\
&f_{n+1}^\mI(x,v)=\e^{\sum_{k=1}^{K}\sigma_{k}(x,v)\delta\beta_{n,k}-\frac{\tau\sum_{k=1}^{K}\sigma_k(x,v)^2}{2}}\hat{f}_{n+1}^\mI(x,v)~,\quad x\in\T^d, v\in \R^d,
\end{aligned}
\right.
\end{equation}
where the Wiener increments $\delta\beta_{n,k}$ are given by~\eqref{eq:WincBeta} in Section~\ref{sec:notation}.

The Lie--Trotter splitting scheme~\eqref{eq:LT-mult-Ito} satisfies the same properties as the exact solution stated in Proposition~\ref{propo:multIto}.
\begin{proposition}\label{propo:LT-mult-Ito}
Let $\bigl(f_n^\mI\bigr)_{n\ge 0}$ be the solution of the Lie--Trotter splitting scheme~\eqref{eq:LT-mult-Ito} with initial value $f_0$.
One then has the following properties.
\begin{itemize}
\item \emph{Preservation of positivity}. Assume that $f_0(x,v)\ge 0$ for all $(x,v)\in\T^d\times\R^d$. Then for any time-step size $\tau\in(0,1)$,
one has $f_n^\mI(x,v)\ge 0$ almost surely for any nonnegative integer $n\ge 0$ and all $(x,v)\in\T^d\times\R^d$.
\item \emph{Preservation of the expected mass}. Assume that $f_0\in L_{x,v}^1$. Then almost surely one has $f_n^{\mI}\in L_{x,v}^1$ for all $n\ge 0$, and
\[
\iint \E[f_n^\mI(x,v)]\,\text dx\,\text dv=\iint f_0(x,v)\,\text dx\,\text dv.
\]
\item \emph{Evolution law for the $L^2$ norm}. Assume that $f_0\in L_{x,v}^2$ and that the condition~\eqref{eq:conditionsigma2constant} is satisfied. Then one has $f_n^\mI\in L^2(\Omega,L_{x,v}^2)$ for all $n\ge 0$, and
\[
\E[\|f_n^{\mI}\|_{L_{x,v}^2}^2]=\e^{\sigma^2 t_n}\|f_0\|_{L_{x,v}^2}^2,
\]
where $t_n=n\tau$.
\end{itemize}
\end{proposition}

Concerning the positivity-preserving property, a similar splitting scheme has been proposed in the recent work~\cite{b23} for another class of SPDEs.

\begin{proof}[Proof of Proposition~\ref{propo:LT-mult-Ito}]
\leavevmode
\begin{itemize}
\item The proof of the positivity preserving property is performed by recursion. It is satisfied if $n=0$ since $f_0^\mI=f_0$. Assume that $f_n^{\mI}(x,v)\ge 0$ almost surely for all $(x,v)\in\T^d\times\R^d$. Owing to the positivity preserving property for the deterministic problem (see Section~\ref{sec:deterministic}), one has for all $(x,v)\in\T^d\times\R^d$
\begin{align*}
\hat{f}_{n+1}^\mI(x,v)&=\bigl(S^2(\tau)S^1(\tau)f_n^\mI\bigr)(x,v)=f_n^\mI(x-tv-t^2E(x),v-tE(x))\ge 0\\
f_{n+1}^\mI(x,v)&=\e^{\sum_{k=1}^{K}\sigma_{k}\delta\beta_{n,k}-\frac{(t-s)\sum_{k=1}^{K}\sigma_k(x,v)^2}{2}}\hat{f}_{n+1}^\mI(x,v)\ge 0.
\end{align*}
Therefore one has $f_{n+1}^\mI(x,v)\ge 0$ almost surely for all $(x,v)\in\T^d\times\R^d$.
\item Observe that the random variables $\hat{f}_{n+1}^\mI(x,v)$ and $\bigl(\delta\beta_{n,k}\bigr)_{1\le k\le K}$ are independent.
As a result, using the well-known expression for the exponential moments of Gaussian random variables, one has
\begin{align*}
\iint \E[f_{n+1}^{\mI}(x,v)]\,\text dx\,\text dv&=\iint \E[\e^{\sum_{k=1}^{K}\delta\beta_{n,k}\sigma_{k}(x,v)-\frac{\tau}{2}\sum_{k=1}^{K}\sigma_{k}(x,v)^2}]\E[\hat{f}_{n+1}^\mI(x,v)]\,\text dx\,\text dv\\
&=\iint \E[\hat{f}_{n+1}^\mI(x,v)]\,\text dx\,\text dv.
\end{align*}
Finally, one has $\hat{f}_{n+1}^\mI=S^2(\tau)S^1(\tau)f_n^\mI$ owing to~\eqref{eq:LT-mult-Ito}, where the linear operators $S^1(\tau)$ and $S^2(\tau)$ preserve the integral, thus one obtains the identity
\[
\iint \E[f_{n+1}^\mI(x,v)]\,\text dx\,\text dv=\iint \E[\hat{f}_{n+1}^\mI(x,v)]\,\text dx\,\text dv=\iint \E[f_{n}^\mI(x,v)]\,\text dx\,\text dv.
\]
\item Applying the same arguments as above, one obtains
\begin{align*}
\E[\|f_{n+1}^\mI\|_{L_{x,v}^2}^2]&=\iint \E[\e^{2\sum_{k=1}^{K}\sigma_k(x,v)\delta\beta_{n,k}-\tau\sum_{k=1}^{K}\sigma_k(x,v)^2}]\E[\hat{f}_{n+1}^\mI(x,v)^2] \,\text dx \,\text dv\\
&=\iint \e^{\tau\sum_{k=1}^{K}\sigma_k(x,v)^2}\E[\hat{f}_{n+1}^\mI(x,v)^2] \,\text dx\, \text dv.
\end{align*}
Then using the condition~\eqref{eq:conditionsigma2constant} and the expression $\hat{f}_{n+1}^\mI=S^2(\tau)S^1(\tau)f_n^\mI$ gives
\[
\E[\|f_{n+1}^\mI\|_{L_{x,v}^2}^2]=\e^{\sigma^2\tau}\E[\|\bigl(S^2(\tau)S^1(\tau)\bigr)f_n^\mI\|_{L_{x,v}^2}^2]=\e^{\sigma^2\tau}\E[\|f_n^\mI\|_{L_{x,v}^2}^2],
\]
using the fact that $S^1(\tau)$ and $S^2(\tau)$ are isometries from $L_{x,v}^2$ to $L_{x,v}^2$. The evolution law then follows from a straightforward recursion argument.
\end{itemize}
\end{proof}

\begin{remark}\label{rem:temporalnoiseIto}
If the noise in the SPDE~\eqref{eq:SPDE-mult-Ito} is a purely temporal Wiener process, i.e. $W(t,x,v)=\beta(t)$ for all $(t,x,v)\in\R^+\times\T^d\times\R^d$ where $\bigl(\beta(t)\bigr)_{t\ge 0}$ is a standard real-valued Wiener process, then the exact solution of~\eqref{eq:SPDE-mult-Ito} and the numerical solution given by~\eqref{eq:LT-mult-Ito} can be written
\[
f^\mI(t)=\e^{\beta(t)-\frac{t}{2}}f^\de(t)~,\quad t\ge 0~;\quad f_n^\mI=\e^{\beta(t_n)-\frac{t_n}{2}}f_n^\de~,\quad n\ge 0,
\]
where $\bigl(f^\de(t)\bigr)_{t\ge 0}$ is the exact solution of the deterministic equation~\eqref{eq:Vlasov}, given by~\eqref{eq:solutionVlasov}, and $\bigl(f_n^\de\bigr)_{n\ge 0}$ is given by the deterministic Lie--Trotter splitting scheme~\eqref{eq:LTdeter}.

In that situation, Propositions~\ref{propo:multIto} and~\ref{propo:LT-mult-Ito} are straightforward consequences of the results described for the deterministic problem in Section~\ref{sec:deterministic}.
\end{remark}

\begin{remark}
Applying the standard Euler--Maruyama scheme to treat the stochastic part of~\eqref{eq:SPDE-mult-Ito} provides the scheme
\[
f_{n+1}^{\rm{mIEM}}=S^2(\tau)S^1(\tau)f_n^{\rm{mIEM}}+\sum_{k=1}^{K}\delta\beta_{n,k}\sigma_{k}f_n^{\rm{mIEM}}.
\]
That scheme does not satisfy the positivity preserving property and the evolution law for the $L^2$ norm stated in Proposition~\ref{propo:LT-mult-Ito} for the proposed Lie--Trotter scheme~\eqref{eq:LT-mult-Ito}.
\end{remark}

\subsubsection{Numerical experiments}

We begin the numerical experiments by illustrating the behavior of the linear Vlasov equation perturbed by multiplicative It\^o noise~\eqref{eq:SPDE-mult-Ito}, in dimension $d=1$. In all the experiments below, the initial value $f_0$ is given by~\eqref{eq:f0} and the vector field $E$ is given by~\eqref{eq:E}. The discretization parameters are given by $\delta x=\frac{1}{400}$, $\delta v=\frac{4\pi}{800}$, and $\tau=0.1$. The snapshots of the numerical solution at times $\{0,0.5,1,1.5,2,2.5\}$ computed using the splitting scheme~\eqref{eq:LT-mult-Ito} are provided in Figures~\ref{fig:snap-it1} and~\ref{fig:snap-it2}, with $K=1$ and diffusion coefficient $\sigma_1$ given by~\eqref{eq:sigma1cos} and~\eqref{eq:sigma1sin} respectively. In both experiments, one observes that the solution remains nonnegative, which illustrates the positivity preserving property stated in Proposition~\ref{propo:LT-mult-Ito} on the considered realization.

\begin{figure}[h]
\begin{subfigure}{.3\textwidth}
  \centering
\includegraphics[width=1\linewidth]{VlasovInit-eps-converted-to.pdf}
\end{subfigure}%
\begin{subfigure}{.3\textwidth}
  \centering
\includegraphics[width=1\linewidth]{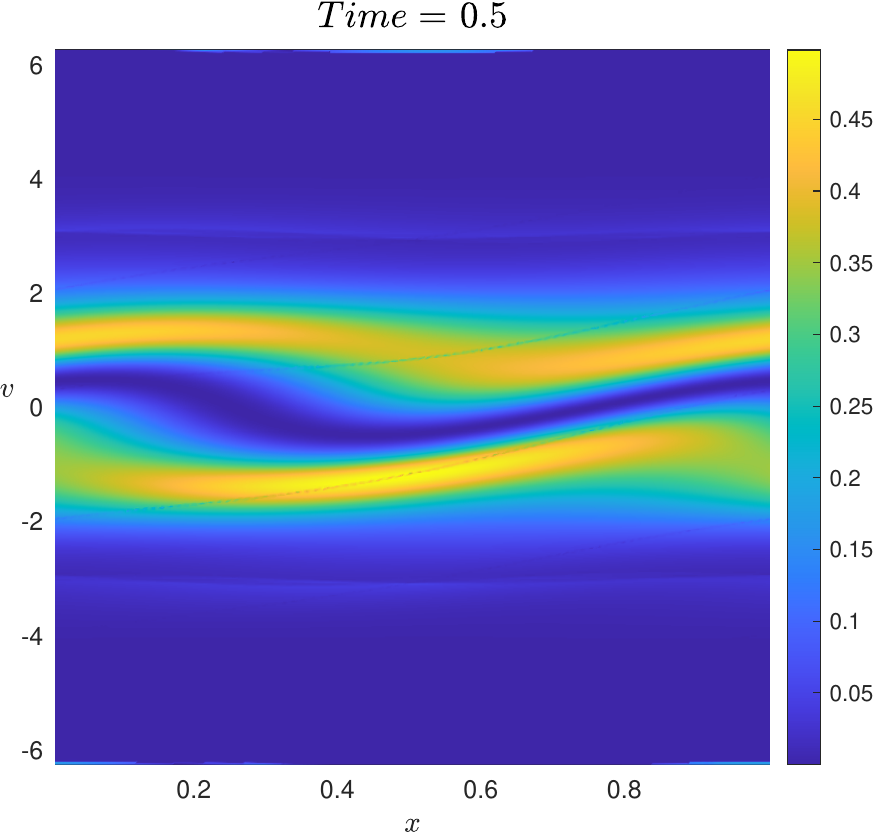}
\end{subfigure}%
\begin{subfigure}{.3\textwidth}
  \centering
  \includegraphics[width=1\linewidth]{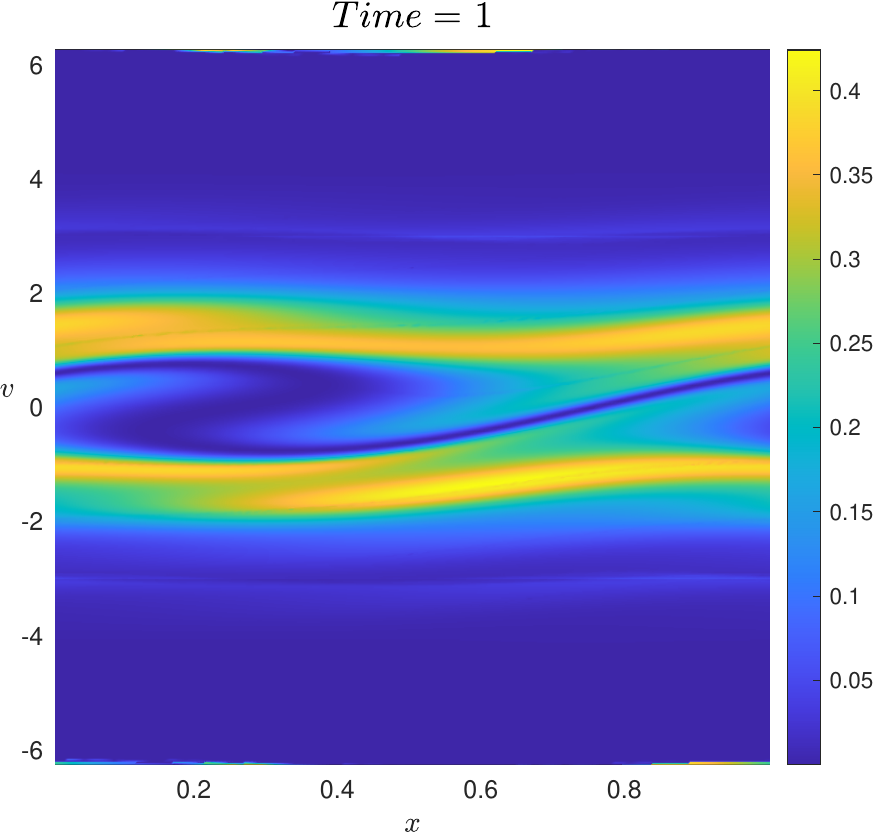}
\end{subfigure}\\[1ex]%
\begin{subfigure}{.3\textwidth}
  \centering
  \includegraphics[width=1\linewidth]{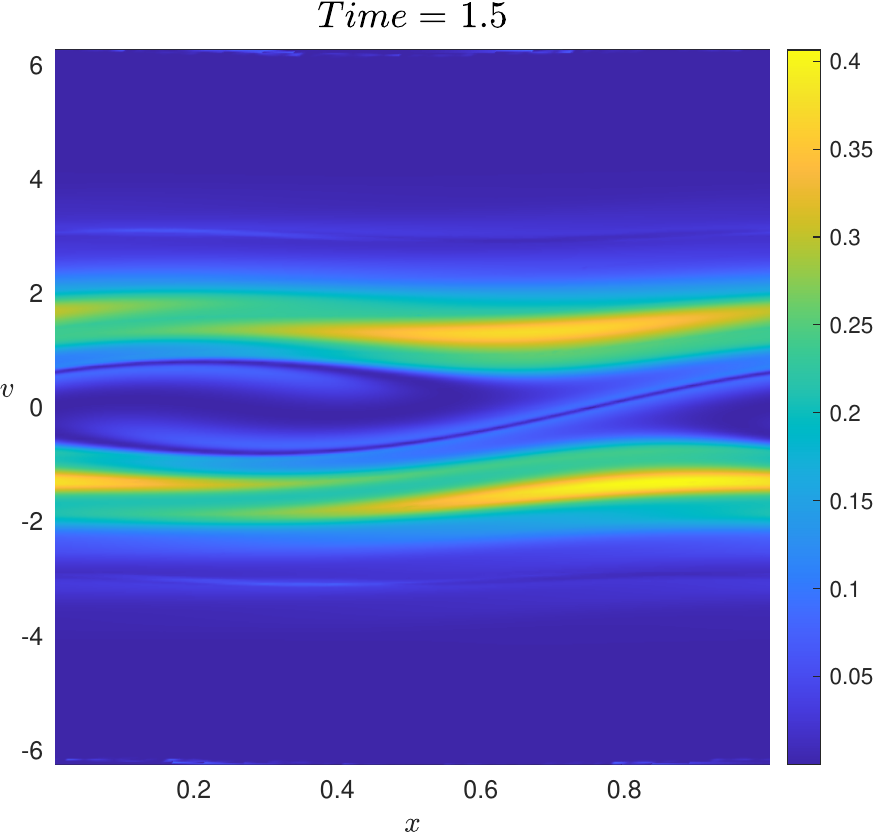}
\end{subfigure}
\begin{subfigure}{.3\textwidth}
  \centering
  \includegraphics[width=1\linewidth]{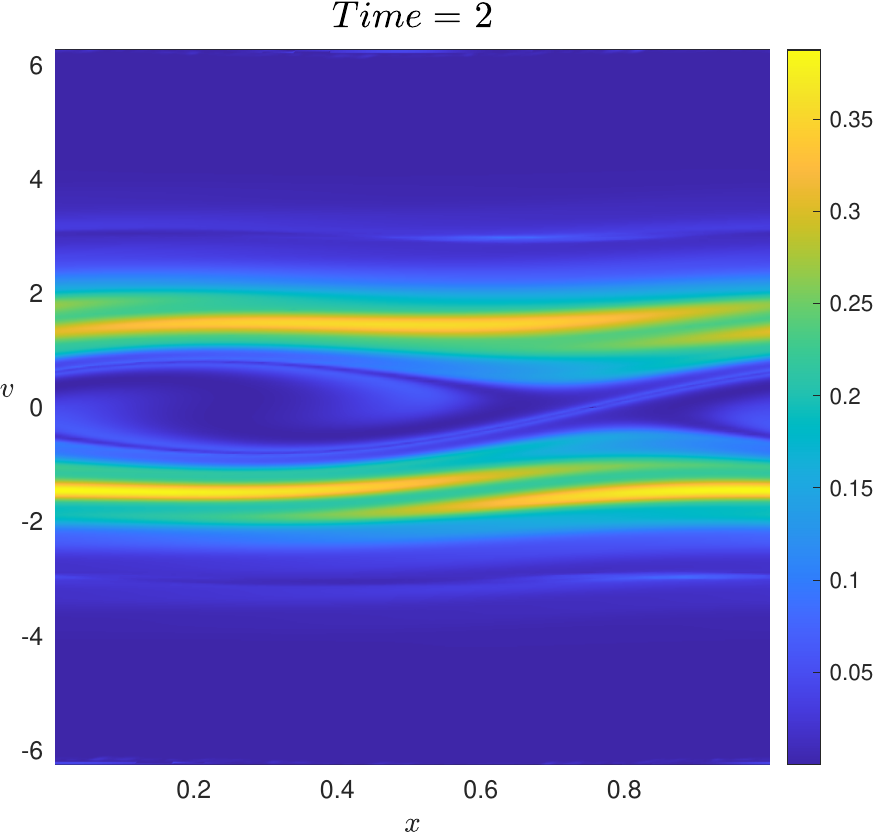}
\end{subfigure}
\begin{subfigure}{.3\textwidth}
  \centering
  \includegraphics[width=1\linewidth]{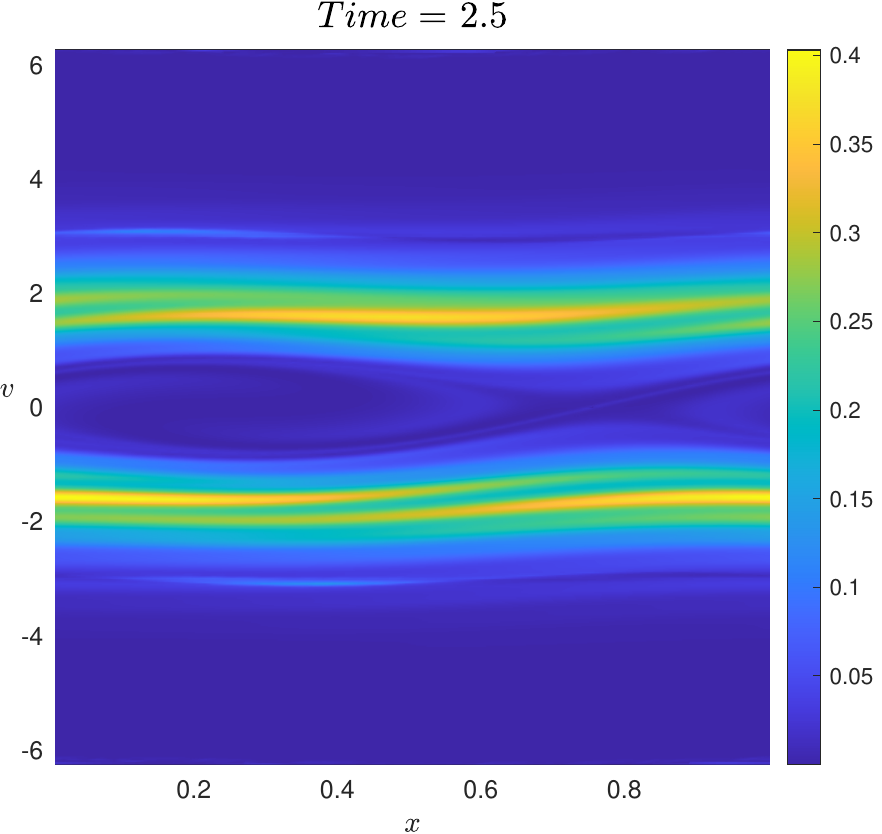}
\end{subfigure}
\caption{Snapshots: approximation of the solution of the stochastic PDE with multiplicative It\^o noise~\eqref{eq:SPDE-mult-Ito} with initial value $f_0$ given by~\eqref{eq:f0}, with $\sigma_1$ given by~\eqref{eq:sigma1cos} at times~$\{0,0.5,1,1.5,2,2.5\}$, using the Lie--Trotter splitting scheme~\eqref{eq:LT-mult-Ito} with time-step size $\tau=0.1$.}
\label{fig:snap-it1}
\end{figure}

\begin{figure}[h]
\begin{subfigure}{.3\textwidth}
  \centering
\includegraphics[width=1\linewidth]{VlasovInit-eps-converted-to.pdf}
\end{subfigure}%
\begin{subfigure}{.3\textwidth}
  \centering
  \includegraphics[width=1\linewidth]{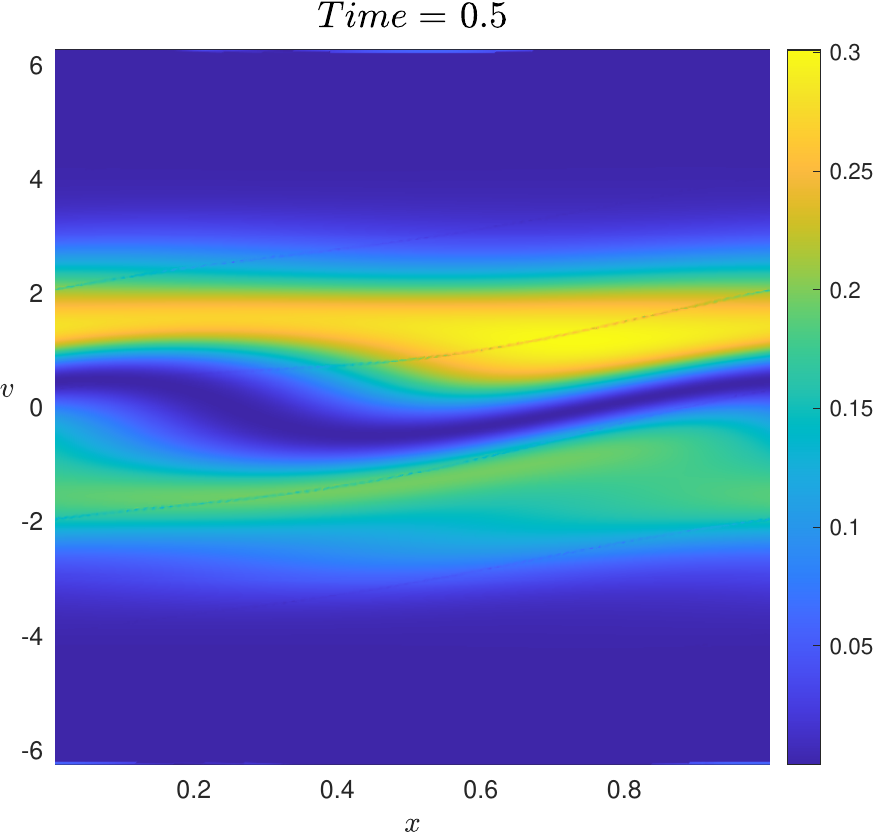}
\end{subfigure}%
\begin{subfigure}{.3\textwidth}
  \centering
  \includegraphics[width=1\linewidth]{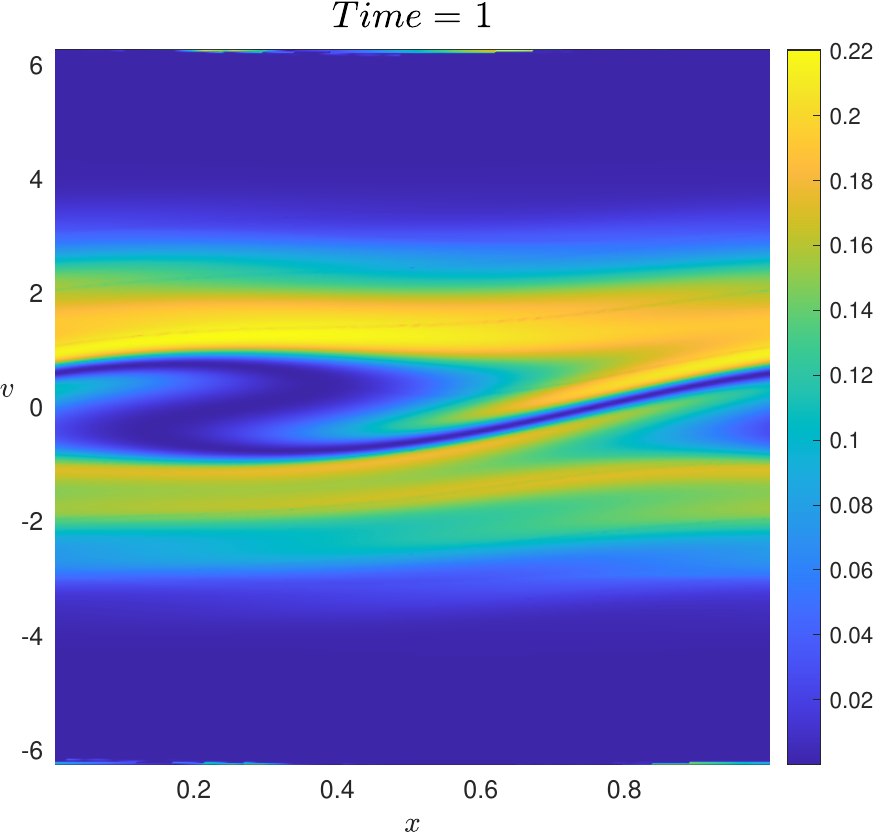}
\end{subfigure}\\[1ex]%
\begin{subfigure}{.3\textwidth}
  \centering
  \includegraphics[width=1\linewidth]{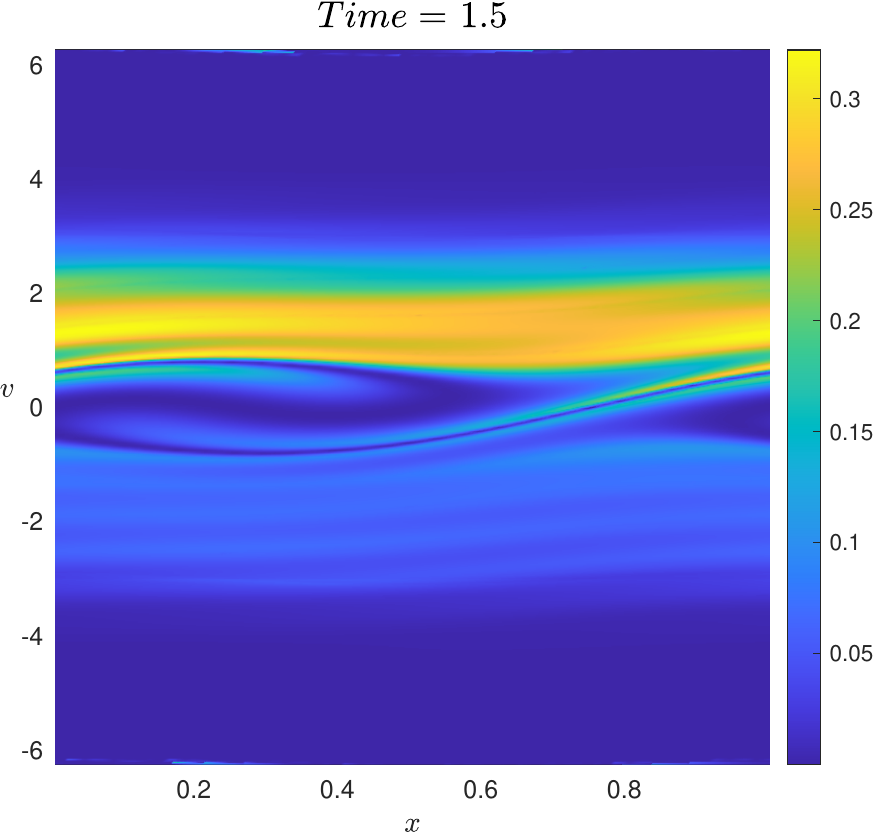}
\end{subfigure}
\begin{subfigure}{.3\textwidth}
  \centering
  \includegraphics[width=1\linewidth]{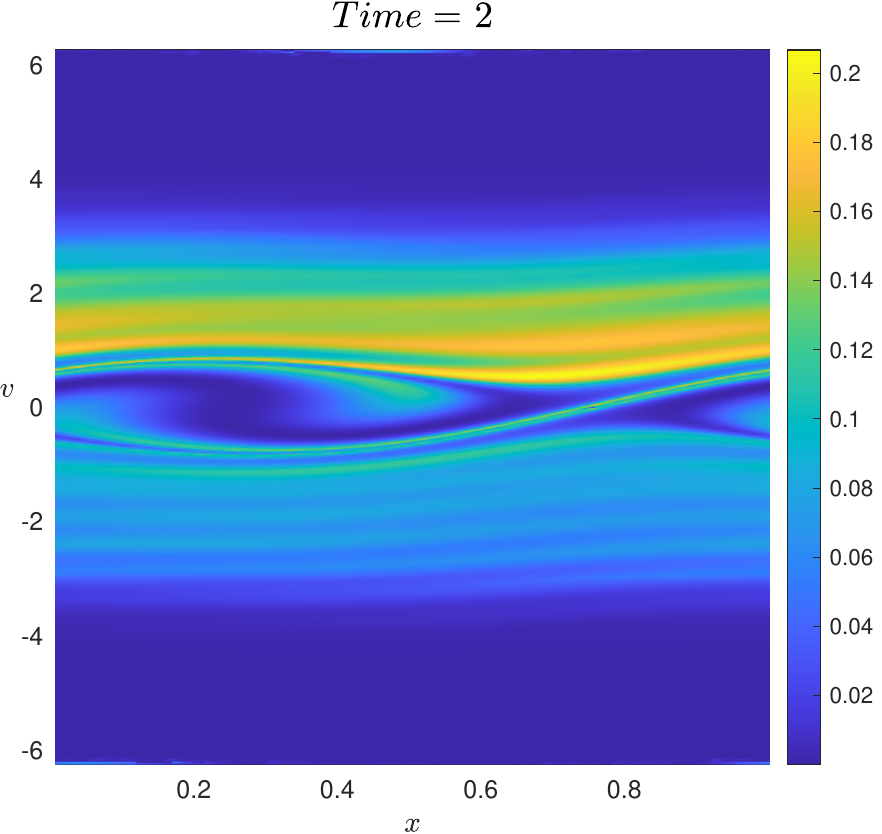}
\end{subfigure}
\begin{subfigure}{.3\textwidth}
  \centering
  \includegraphics[width=1\linewidth]{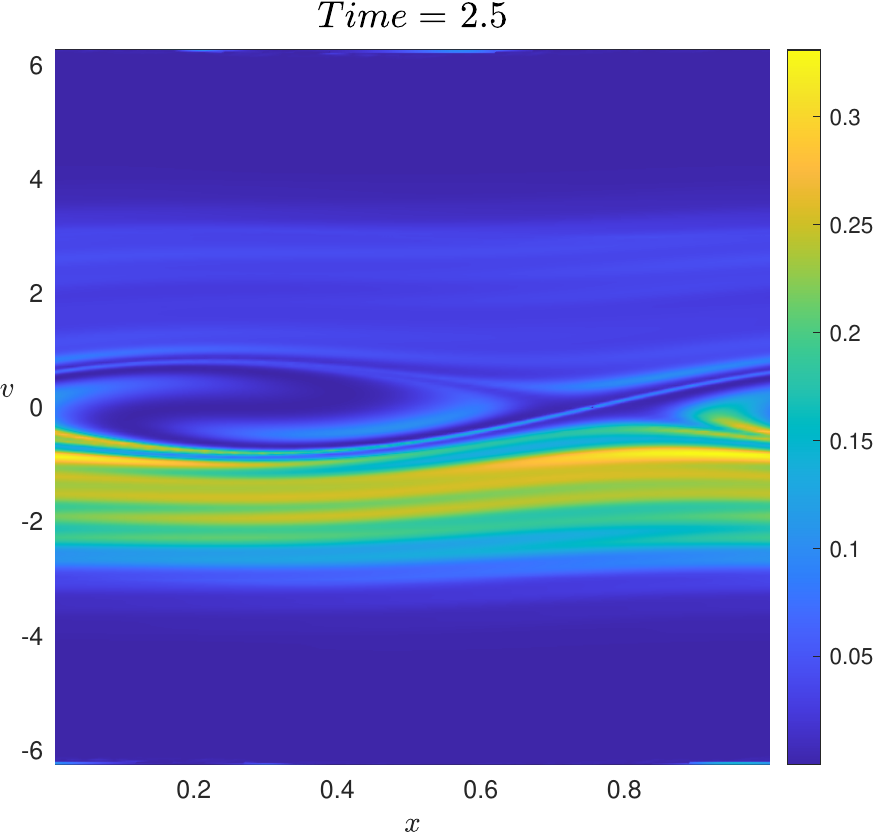}
\end{subfigure}
\caption{Snapshots: approximation of the solution of the stochastic PDE with multiplicative It\^o noise~\eqref{eq:SPDE-mult-Ito} with initial value $f_0$ given by~\eqref{eq:f0}, with $\sigma_1$ given by~\eqref{eq:sigma1sin} at times~$\{0,0.5,1,1.5,2,2.5\}$, using the Lie--Trotter splitting scheme~\eqref{eq:LT-mult-Ito} with time-step size $\tau=0.1$.}
\label{fig:snap-it2}
\end{figure}

Let us now check the almost sure positivity preserving property for the Lie--Trotter splitting scheme~\eqref{eq:LT-mult-Ito} in a more rigorous way: we have run $2.10^4$ independent samples on the time interval $[0,1]$, with initial value~\eqref{eq:f0}, with the same discretization parameters as above, and with different choices of the diffusion coefficients: either $K=1$ and $\sigma_1$ given by~\eqref{eq:sigma1sin}, or $K=2$ and $\sigma_1$ and $\sigma_2$ give by
\begin{equation}\label{eq:sigma12}
\sigma_1(x,v)=\cos(2\pi x)~,\quad \sigma_2(x,v)=\sin(2\pi x).
\end{equation}
All the samples only take nonnegative values, which confirms the positivity preserving property stated in Proposition~\ref{propo:LT-mult-Ito}.

Next, we illustrate the preservation of the expected mass and the evolution law of the $L^2$ norm stated in Proposition~\ref{propo:LT-mult-Ito}. In these experiments, one has $d=1$, $T=1$, $\delta x=\frac{1}{200}$,
$\delta v=\frac{4\pi}{400}$ and $\tau=0.1$. The expectations  are computed using an averaging procedure over $5.10^5$ samples. Since the solution is nonnegative, the mass is in fact equal to the $L^1$ norm of the solution. The results are presented in Figure~\ref{fig:normMultIto}, with different choices of the diffusion coefficients: $K=1$ with $\sigma_1(x,v)=1$, $K=1$ with $\sigma_1$ given by~\eqref{eq:sigma1sin}, and $K=2$ with $\sigma_1$ and $\sigma_2$ given by~\eqref{eq:sigma12}, respectively. Note that the condition~\eqref{eq:conditionsigma2constant} is satisfied in the first and in the third case. We observe a good agreement with the theoretical results given in Proposition~\ref{propo:LT-mult-Ito}.

\begin{figure}[h]
\begin{subfigure}{.45\textwidth}
  \centering
  \includegraphics[width=.9\linewidth]{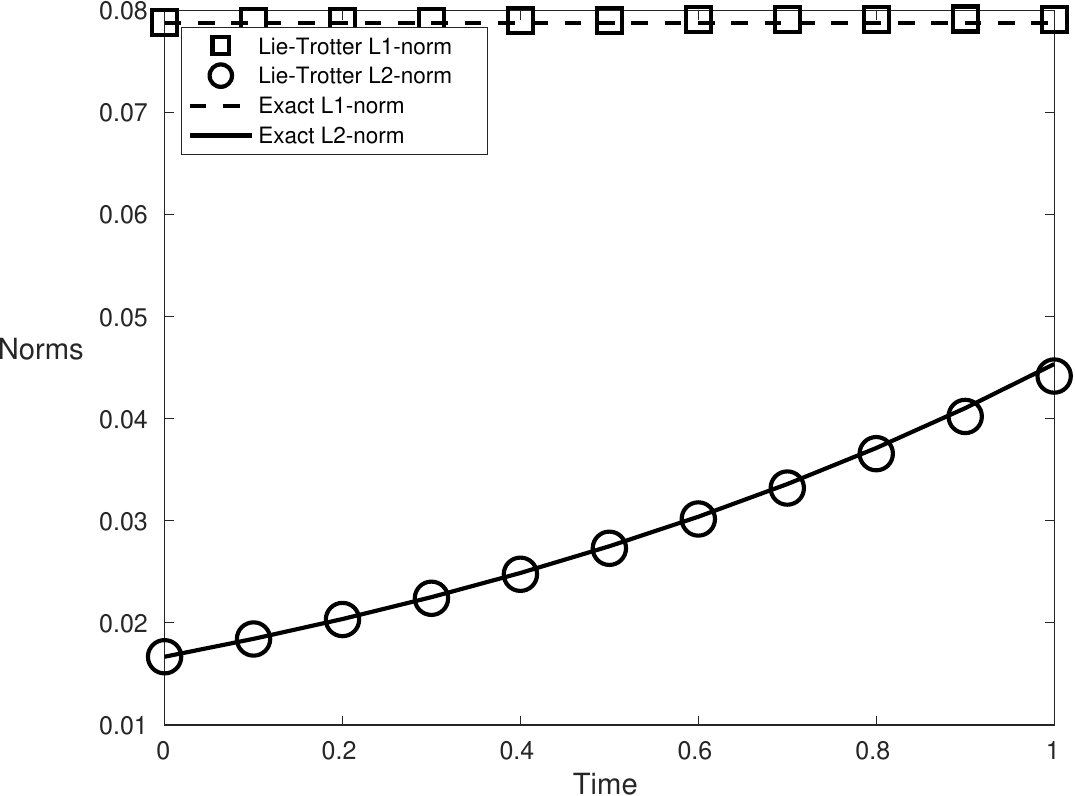}
  \caption{Noise given by $\sigma_1(x,v)=1$, $K=1$.}
\end{subfigure}%
\begin{subfigure}{.45\textwidth}
  \centering
  \includegraphics[width=.9\linewidth]{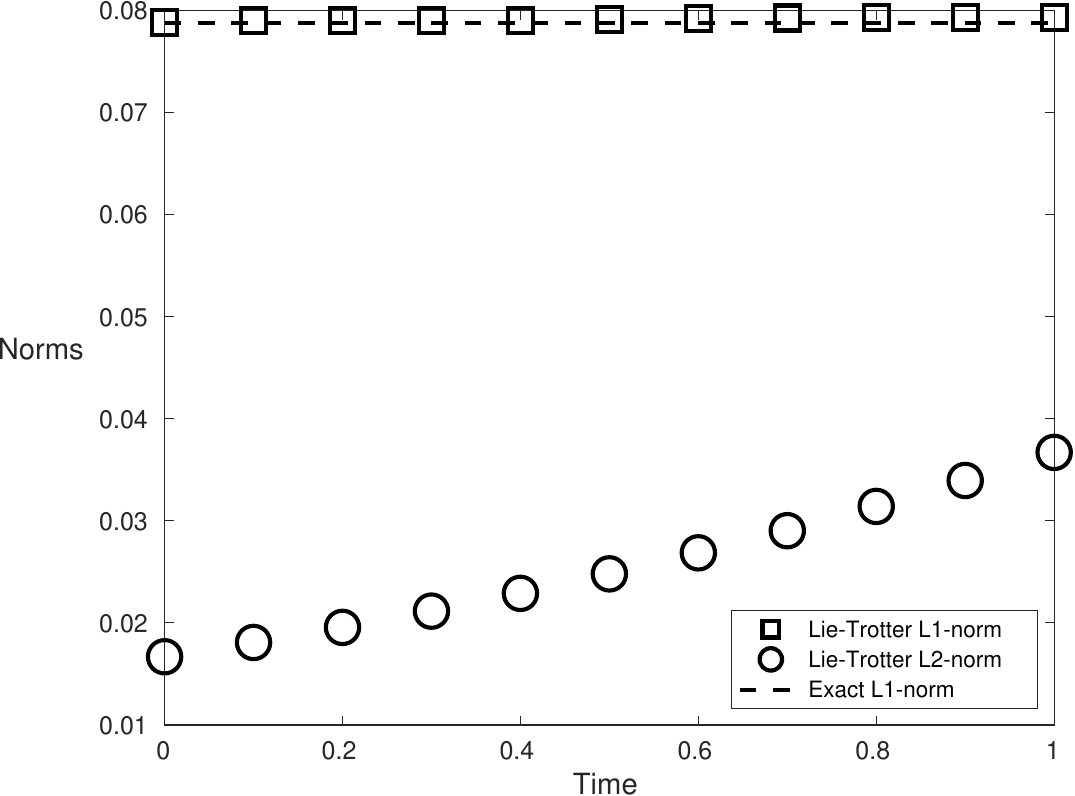}
  \caption{Noise given by~\eqref{eq:sigma1sin}, $K=1$.}
\end{subfigure}\\[1ex]%
\begin{subfigure}{.45\textwidth}
  \centering
  \includegraphics[width=.9\linewidth]{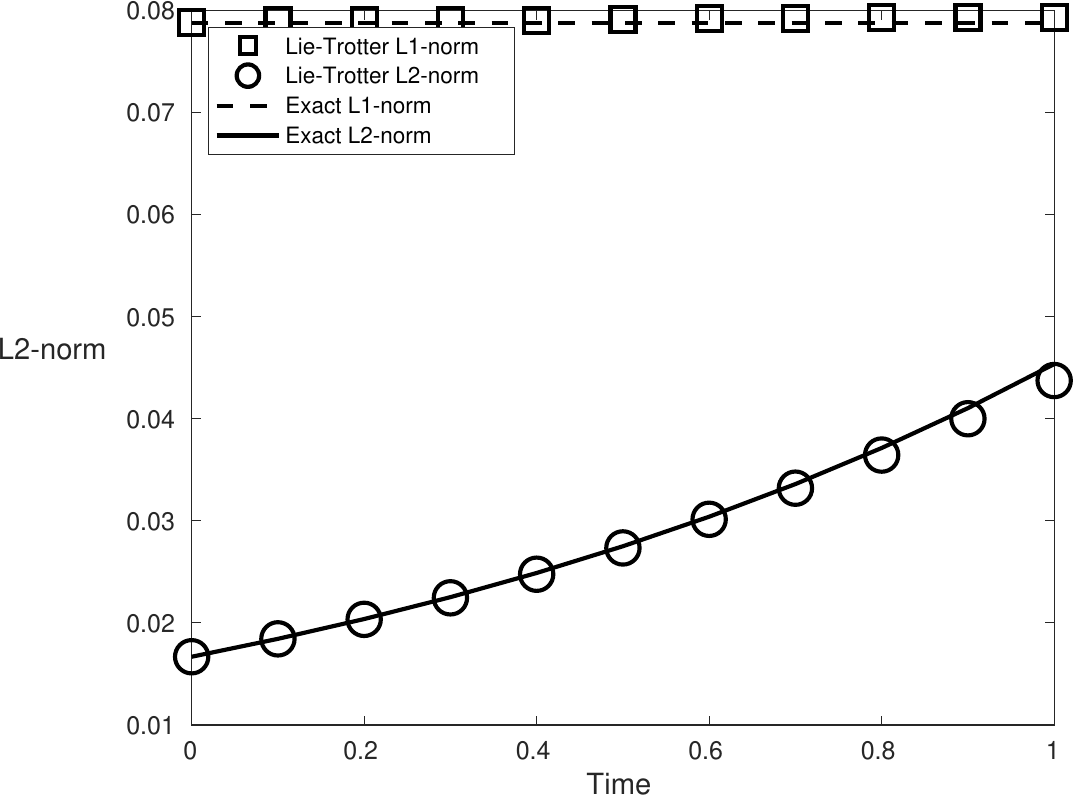}
  \caption{Noise given by~\eqref{eq:sigma12}, $K=2$.}
\end{subfigure}
\caption{Preservation of the expected mass and evolution law for the $L_{x,v}^2$ norm: illustration of Proposition~\ref{propo:LT-mult-Ito} when applying the Lie--Trotter scheme~\eqref{eq:LT-mult-Ito} with $\tau=0.1$ to the SPDE with multiplicative It\^o noise~\eqref{eq:SPDE-mult-Ito} with time-step size $\tau=0.1$.}
\label{fig:normMultIto}
\end{figure}

We conclude these numerical experiments in the multiplicative It\^o noise case by investigating the mean-square order of convergence of the Lie--Trotter splitting scheme~\eqref{eq:LT-mult-Ito}. The same procedure as in the additive noise case (Section~\ref{sec:additive}) is applied. A reference solution is computed using the splitting scheme with time-step size $\tau_{\rm ref}=2^{-14}$, and the errors are computed when the time-step size $\tau$ takes values in $\{2^{-7},\ldots,2^{-13}\}$. The expectation is computed using a Monte Carlo averaging procedure over $500$ independent samples. The discretization parameters are
$\delta x=\frac{1}{100}$, $\delta v=\frac{4\pi}{200}$. The final time is $T=0.5$. The noise is given by~\eqref{eq:sigma1sin} or by~\eqref{eq:sigma12}. The results are presented in a loglog plot in Figure~\ref{fig:ms-MultIto}. We observe a mean-square convergence order equal to $1$.

\begin{figure}[h]
\begin{subfigure}{.5\textwidth}
  \centering
  \includegraphics[width=.9\linewidth]{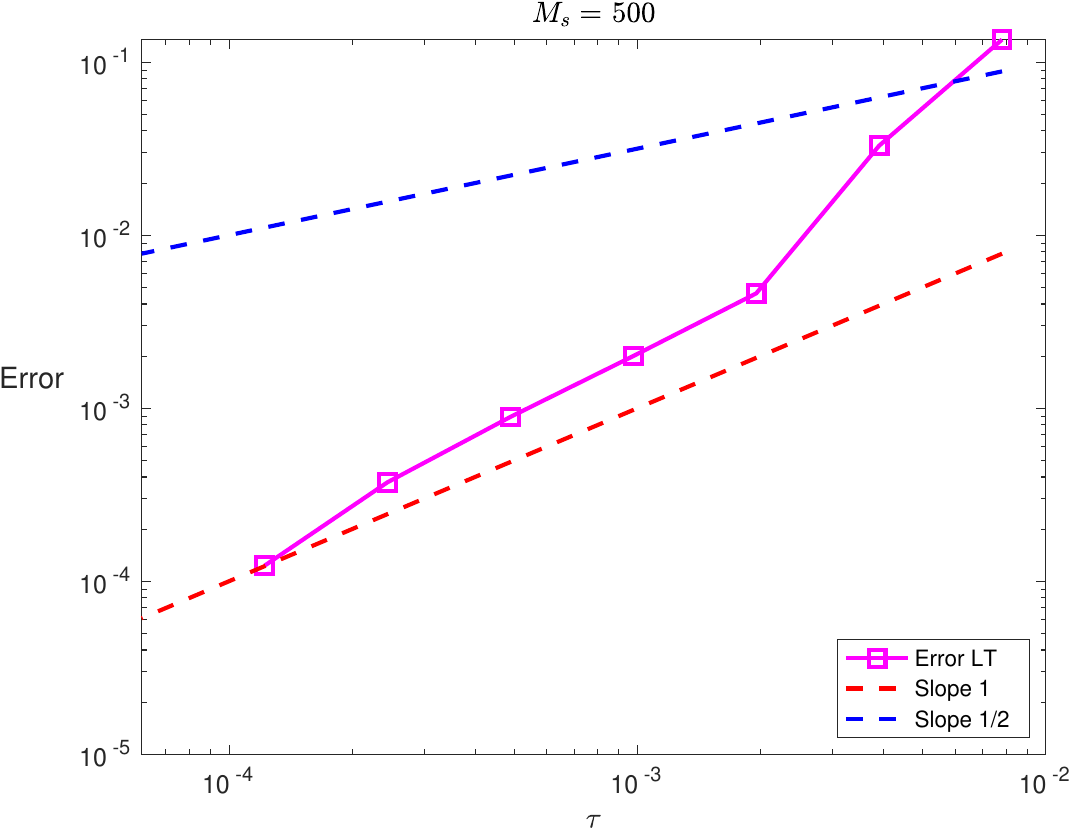}
  \caption{Noise given by~\eqref{eq:sigma1sin}.}
\end{subfigure}%
\begin{subfigure}{.5\textwidth}
  \centering
  \includegraphics[width=.9\linewidth]{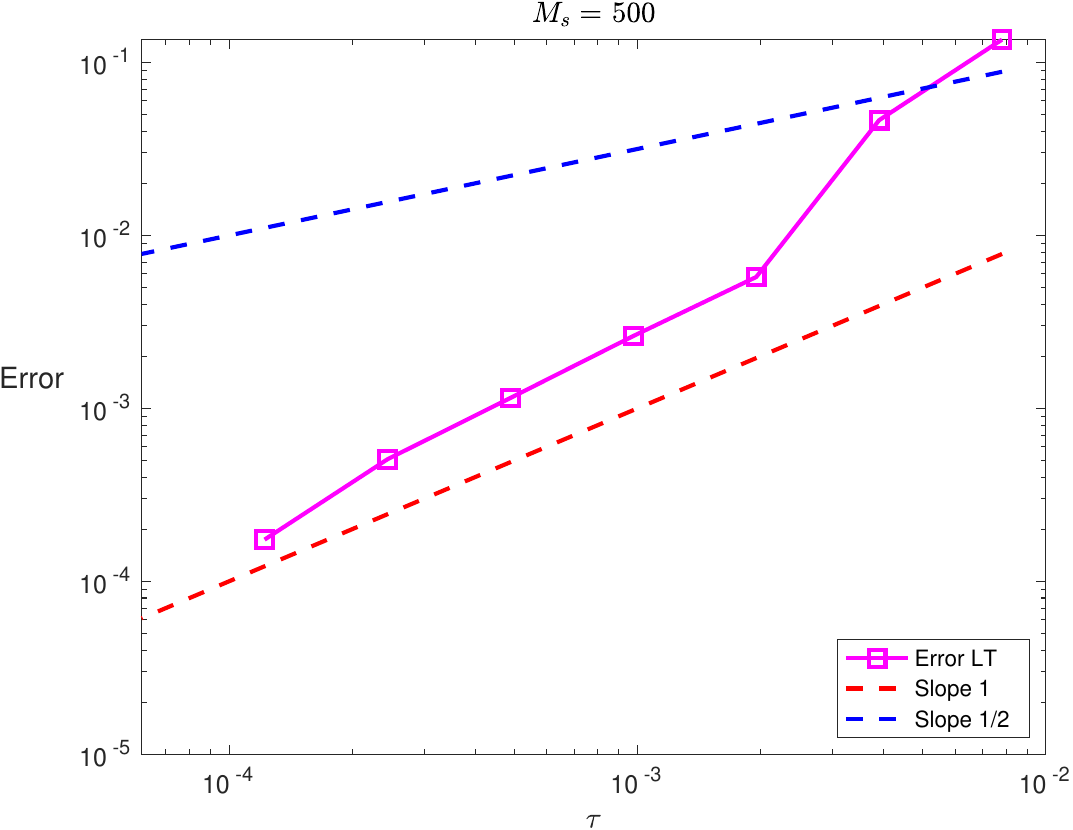}
  \caption{Noise given by~\eqref{eq:sigma12}.}
\end{subfigure}
\caption{Mean-square errors: Mean-square errors: Lie--Trotter scheme~\eqref{eq:LT-mult-Ito} applied to the SPDE with multiplicative It\^o noise~\eqref{eq:SPDE-mult-Ito} driven by one-dimensional noise ($K=1$, left) and by two-dimensional noise ($K=2$, right).}
\label{fig:ms-MultIto}
\end{figure}

\subsection{Stratonovich interpretation}\label{sec:multiplicativeStrato}

Let us now consider the linear Vlasov equation driven by a multiplicative noise interpreted in the Stratonovich sense, for $t\ge 0$, $x\in\T^d$ and $v\in \R^d$:
\begin{equation}\label{eq:SPDE-mult-Strato}
\left\lbrace
\begin{aligned}
&\text df^\mS(t,x,v)+v\cdot\nabla_xf^\mS(t,x,v)\,\text dt+E(x)\cdot\nabla_v f^\mS(t,x,v)\,\text dt=f^\mS(t,x,v)\circ \text dW(t,x,v)~,\\
&f^\mS(0,x,v)=f_0(x,v)~,
\end{aligned}
\right.
\end{equation}
where we recall that the noise is defined by~\eqref{eq:Wtxv} (Section~\ref{sec:notation}) and that the symbol $\circ$ denotes the Stratonovich product. The stochastic partial differential equation~\eqref{eq:SPDE-mult-Strato} has the equivalent It\^o formulation, for $t\ge 0, x\in\T^d, v\in \R^d$:
\begin{equation}\label{eq:SPDE-mult-Strato-Ito}
\left\lbrace
\begin{aligned}
&\text df^\mS(t,x,v)+v\cdot\nabla_xf^\mS(t,x,v)\,\text dt+E(x)\cdot\nabla_v f^\mS(t,x,v)\,\text dt=f^\mS(t,x,v)\,\text dW(t,x,v)\\
&\quad+\frac12\sum_{k=1}^{K}\sigma_k(x,v)^2 f^\mS(t,x,v)\,\text dt~,\\
&f^\mS(0,x,v)=f_0(x,v)~.
\end{aligned}
\right.
\end{equation}

\subsubsection{Analysis and properties of the problem}

Using the tools described in Section~\ref{sec:deterministic} and like in the multiplicative It\^o noise case studied in Section~\ref{sec:multiplicativeIto}, solutions of~\eqref{eq:SPDE-mult-Strato} can be written in different ways.

On the one hand, using the group $\bigl(S(t)\bigr)_{t\in\R}$ of linear operators given by~\eqref{eq:S} in Section~\ref{sec:deterministic}, one can consider mild solutions of~\eqref{eq:SPDE-mult-Strato-Ito}: for all $t\ge 0$, one has
\begin{equation}\label{eq:SPDE-mult-Strato-solution1}
f^\mS(t)=S(t)f_0+\sum_{k=1}^{K}\int_{0}^{t}S(t-s)\bigl(f^\mS(s)\sigma_k\bigr)\,\text d\beta_k(s)+\frac12\sum_{k=1}^{K}\int_{0}^{t}S(t-s)\bigl(f^\mS(s)\sigma_k^2\bigr) \,\text ds.
\end{equation}
On the other hand, using the expression~\eqref{eq:S} for the linear operator $S(t)$, one has for all $t\ge 0$, $x\in\T^d$ and $v\in\R^d$
\begin{equation}\label{eq:SPDE-mult-Strato-solution2}
\begin{aligned}
f^\mS(t,x,v)&=f_0(\phi_t^{-1}(x,v))+\sum_{k=1}^{K}\int_{0}^{t}f^\mS(s,\phi_{t-s}^{-1}(x,v))\sigma_k(\phi_{t-s}^{-1}(x,v))\,\text d\beta_k(s)\\
&\quad+\frac12\sum_{k=1}^{K}\int_{0}^{t}f^\mS(s,\phi_{t-s}^{-1}(x,v))\sigma_k^2(\phi_{t-s}^{-1}(x,v))\,\text ds.
\end{aligned}
\end{equation}
Finally, the connection with the ordinary differential equation~\eqref{eq:ODE} can also be seen by applying the It\^o--Wentzell formula (see \autoref{sec:ItoWentzell}): if $t\mapsto (x_t,v_t)=\phi_t(x_0,v_0)$ is the solution of the ordinary differential equation~\eqref{eq:ODE} with arbitrary initial value $(x_0,v_0)\in\T^d\times\R^d$, and if the solution of~\eqref{eq:SPDE-mult-Strato} is sufficiently regular, then the stochastic process $t\ge 0\mapsto f^\mS(t,x_t,v_t)$ satisfies
\begin{equation}\label{eq:SPDE-mult-Strato-solutionIW}
\begin{aligned}
\text df^\mS(t,x_t,v_t)&=\sum_{k=1}^{K}f^\mS(t,x_t,v_t)\sigma_k(x_t,v_t)\,\text d\beta_k(t)+\frac12\sum_{k=1}^{K}f^\mS(t,x_t,v_t)\sigma_k(x_t,v_t)^2\,\text dt\\
&=\sum_{k=1}^{K}f^\mS(t,x_t,v_t)\sigma_k(x_t,v_t)\circ \text d\beta_k(t).
\end{aligned}
\end{equation}
The formula~\eqref{eq:SPDE-mult-Strato-solutionIW} allows to retrieve the expression~\eqref{eq:SPDE-mult-Strato-solution2} of $f^\mI(t,x,v)$ above
by writing $(x_t,v_t)=\phi_t(x,v)$.

Let us now describe the properties of the solutions of the SPDE~\eqref{eq:SPDE-mult-Strato}.
\begin{proposition}\label{propo:multStrato}
Let $\left(f^\mS(t)\right)_{t\geq0}$ be the solution of the SPDE~\eqref{eq:SPDE-mult-Strato} with initial value
$f_0$. One has the following properties.
\begin{itemize}
\item \emph{Preservation of positivity.}
Assume that $f_0(x,v)\ge 0$ for all $(x,v)\in\T^d\times \R^d$. Then, one has
$f^\mS(t,x,v)\ge 0$ almost surely for all $t\geq0$ and $(x,v)\in\T^d\times\R^d$.
\item \emph{Evolution law for the $L^2$ norm.} Assume that $f_0\in L_{x,v}^2$ and that
the condition~\eqref{eq:conditionsigma2constant} is satisfied. Then one has
$f(t)\in L^2(\Omega,L_{x,v}^2)$ for all $t\geq0$, and
\[
\E[\|f^\mS(t)\|_{L_{x,v}^2}^2]=\e^{2\sigma^2 t}\|f_0\|_{L_{x,v}^2}^2.
\]
\end{itemize}
\end{proposition}

\begin{proof}[Proof of Proposition~\ref{propo:multStrato}]
\leavevmode
\begin{itemize}
\item Owing to the expression~\eqref{eq:SPDE-mult-Strato-solutionIW} of the solution and applying a comparison principle for solutions of stochastic differential equations, one has $f^\mS(t,x_t,v_t)\ge 0$ almost surely for all $t\ge 0$, and for any arbitrary initial value $(x_0,v_0)\in\T^d\times\R^d$. Choosing $(x_0,v_0)=\phi_t^{-1}(x,v)$ then yields $f^\mS(t,x,v)\ge 0$ almost surely for all $(t,x,v)\in \R^+\times\T^d\times\R^d$.

\item Owing to the expression~\eqref{eq:SPDE-mult-Strato-solution1} of the mild solution, using It\^o's formula in the Hilbert space $L_{x,v}^2$ and the isometry property of the group $\bigl(S(t)\bigr)_{t\geq0}$ (see Section~\ref{sec:deterministic}), one obtains for all $t\ge 0$
\begin{equation*}
\frac12\frac{\text d\E[\|f^\mS(t)\|_{L_{x,v}^2}^2]}{\text dt}=\sum_{k=1}^{K}\E[\|\sigma_k f^\mS(t)\|_{L_{x,v}^2}^2].
\end{equation*}
Using the condition~\eqref{eq:conditionsigma2constant} and integrating then yields the identity
\[
\E[\|f^\mS(t)\|_{L_{x,v}^2}^2]=\e^{2\sigma^2 t}\|f_0\|_{L_{x,v}^2}^2.
\]
\end{itemize}
\end{proof}

\subsubsection{Splitting scheme}
Let us now describe the proposed numerical scheme for the temporal discretization of the SPDE~\eqref{eq:SPDE-mult-Strato} driven by multiplicative It\^o noise. Like in the multiplicative It\^o noise case presented in Section~\ref{sec:multiplicativeIto}, a Lie--Trotter splitting strategy is applied. The treatment of the deterministic part is not modified. Compared with Section~\ref{sec:multiplicativeIto}, the auxiliary stochastic subsystem needs to be considered with Stratonovich interpretation of the noise:
\[
\text df(t,x,v)=f(t,x,v)\circ \text dW(t,x,v)~,\quad (t,x,v)\in\R^+\times\T^d\times\R^d
\]
The auxiliary stochastic subsystem above is solved exactly: for all $t\ge s\ge 0$ one has
\[
f(t,x,v)=\e^{\sum_{k=1}^{K}\sigma_{k}(\beta_k(t)-\beta_k(s))}f(s,x,v)~,\quad \forall~(x,v)\in\T^d\times\R^d.
\]
Using the Lie--Trotter integrator~\eqref{eq:LTdeter} for the deterministic part and combining the discretizations of the deterministic and stochastic parts yields the following scheme: given the initial value $f_0$ and the time-step size $\tau\in(0,1)$, set $f_0^\mS=f_0$ and for any nonnegative integer $n\ge 0$ set
\begin{equation}\label{eq:LT-mult-Strato}
\left\lbrace
\begin{aligned}
&\hat f_{n+1}^\mS(x,v)=S^2(\tau)S^1(\tau)f_n^\mS(x,v)\\
&f_{n+1}^\mS(x,v)=\e^{\sum_{k=1}^{K}\sigma_k(x,v)\delta\beta_{n,k}}\hat f_n^\mS(x,v)~,\quad x\in\T^d, v\in\R^d,
\end{aligned}
\right.
\end{equation}
where we recall that the Wiener increments $\delta\beta_{n,k}$ are given by~\eqref{eq:WincBeta}, see Section~\ref{sec:notation}.

The Lie--Trotter splitting scheme~\eqref{eq:LT-mult-Strato} satisfies the same properties as the exact solution stated in Proposition~\ref{propo:multStrato}.

\begin{proposition}\label{propo:LT-mult-Strato}
Let $(f_n^\mS)_{n\ge 0}$ be the solution of the Lie--Trotter splitting scheme~\eqref{eq:LT-mult-Strato} with initial value $f_0$.
One then has the following properties.
\begin{itemize}
\item \emph{Preservation of positivity.} Assume that $f_0(x,v)\ge 0$ for all $(x,v)\in\T^d\times\R^d$. Then, for any time-step size $\tau\in(0,1)$, one has $f_n^\mS(x,v)\ge 0$ almost surely for any nonnegative integer $n\in\N$ and all $(x,v)\in\T^d\times\R^d$.
\item \emph{Evolution law for the $L^2$ norm}. Assume that $f_0\in L^2_{x,v}$ and that the condition~\eqref{eq:conditionsigma2constant} is satisfied. Then, one has $f_n^\mS\in L^2(\Omega,L^2_{x,v})$, for all $n\ge 0$, and
$$
\E[\|f_n^\mS\|_{L_{x,v}^2}^2]=\e^{2\sigma^2t_n}\|f_0\|_{L_{x,v}^2}^2,
$$
where $t_n=n\tau$.
\end{itemize}
\end{proposition}

The proof is similar to the proof of Proposition~\ref{propo:multIto} from Section~\ref{sec:multiplicativeIto}.

\begin{proof}[Proof of Proposition~\ref{propo:LT-mult-Strato}]
\leavevmode
\leavevmode
\begin{itemize}
\item The proof of the positivity preserving property is performed by recursion. It is satisfied if $n=0$ since $f_0^\mS=f_0$. Assume that $f_n^{\mS}(x,v)\ge 0$ almost surely for all $(x,v)\in\T^d\times\R^d$. Owing to the positivity preserving property for the deterministic problem (see Section~\ref{sec:deterministic}), one has for all $(x,v)\in\T^d\times\R^d$
\begin{align*}
\hat{f}_{n+1}^\mS(x,v)&=\bigl(S^2(\tau)S^1(\tau)f_n^\mS\bigr)(x,v)=f_n^\mS(x-tv-t^2E(x),v-tE(x))\ge 0\\
f_{n+1}^\mS(x,v)&=\e^{\sum_{k=1}^{K}\sigma_{k}(x,v)\delta\beta_{n,k}}\hat{f}_{n+1}^\mS(x,v)\ge 0.
\end{align*}
Therefore one has $f_{n+1}^\mS(x,v)\ge 0$ almost surely for all $(x,v)\in\T^d\times\R^d$.

\item Observe that the random variables $\hat{f}_{n+1}^\mS(x,v)$ and $\bigl(\delta\beta_{n,k}\bigr)_{1\le k\le K}$ are independent.
As a result, using the well-known expression for the exponential moments of Gaussian random variables, one has
\begin{align*}
\E[\|f_{n+1}^\mS\|_{L_{x,v}^2}^2]&=\iint \E[\e^{2\sum_{k=1}^{K}\sigma_k(x,v)\delta\beta_{n,k}}]\E[\hat{f}_{n+1}^\mS(x,v)^2] \,\text dx \,\text dv\\
&=\iint \e^{2\tau\sum_{k=1}^{K}\sigma_k(x,v)^2}\E[\hat{f}_{n+1}^\mS(x,v)^2] \,\text dx\, \text dv.
\end{align*}
Then using the condition~\eqref{eq:conditionsigma2constant} and the expression $\hat{f}_{n+1}^\mS=S^2(\tau)S^1(\tau)f_n^\mS$ gives
\[
\E[\|f_{n+1}^\mS\|_{L_{x,v}^2}^2]=\e^{2\sigma^2\tau}\E[\|\bigl(S^2(\tau)S^1(\tau)\bigr)f_n^\mS\|_{L_{x,v}^2}^2]=\e^{\sigma^2\tau}\E[\|f_n^\mS\|_{L_{x,v}^2}^2],
\]
using the fact that $S^1(\tau)$ and $S^2(\tau)$ are isometries from $L_{x,v}^2$ to $L_{x,v}^2$. The evolution law then follows by a straightforward recursion argument.
\end{itemize}
\end{proof}

Remark~\ref{rem:temporalnoiseStrato} below is a discussion in the Stratonovich noise case of the situation described in Remark~\ref{rem:temporalnoiseIto} in the It\^o noise case above.
\begin{remark}\label{rem:temporalnoiseStrato}
If the noise in the SPDE~\eqref{eq:SPDE-mult-Strato} is a purely temporal Wiener process, i.e. $W(t,x,v)=\beta(t)$ for all $(t,x,v)\in\R^+\times\T^d\times\R^d$ where $\bigl(\beta(t)\bigr)_{t\ge 0}$ is a standard real-valued Wiener process, then the exact solution of~\eqref{eq:SPDE-mult-Strato} and the numerical solution given by~\eqref{eq:LT-mult-Strato} can be written
\[
f^\mS(t)=\e^{\beta(t)}f^\de(t)~,\quad t\ge 0~;\quad f_n^\mS=\e^{\beta(t_n)}f_n^\de~,\quad n\ge 0,
\]
where $\bigl(f^\de(t)\bigr)_{t\ge 0}$ is the exact solution of the deterministic equation~\eqref{eq:Vlasov}, given by~\eqref{eq:solutionVlasov}, and $\bigl(f_n^\de\bigr)_{n\ge 0}$ is given by the deterministic Lie--Trotter splitting scheme~\eqref{eq:LTdeter}.

In that situation, Propositions~\ref{propo:multStrato} and~\ref{propo:LT-mult-Strato} are straightforward consequences of the results described for the deterministic problem in Section~\ref{sec:deterministic}.
\end{remark}

\subsubsection{Numerical experiments}

We begin the numerical experiments by illustrating the behavior of the linear Vlasov equation perturbed by multiplicative Stratonovich noise~\eqref{eq:SPDE-mult-Strato}, in dimension $d=1$. In all the experiments below, the initial value $f_0$ is given by~\eqref{eq:f0} and the vector field $E$ is given by~\eqref{eq:E}. The discretization parameters are given by $\delta x=\frac{1}{400}$, $\delta v=\frac{4\pi}{800}$, and $\tau=0.1$. The snapshots of the numerical solution at times $\{0,0.5,1,1.5,2,2.5\}$ computed using the splitting scheme~\eqref{eq:LT-mult-Strato} are provided in Figures~\ref{fig:snap-strato1} and~\ref{fig:snap-strato2}, with $K=1$ and diffusion coefficient $\sigma_1$ given by~\eqref{eq:sigma1cos} and~\eqref{eq:sigma1sin} respectively. In both experiments, one observes that the solution remains nonnegative, which illustrates the positivity preserving property stated in Proposition~\ref{propo:LT-mult-Strato} on the considered realization.

\begin{figure}[h]
\begin{subfigure}{.3\textwidth}
  \centering
\includegraphics[width=1\linewidth]{VlasovInit-eps-converted-to.pdf}
\end{subfigure}%
\begin{subfigure}{.3\textwidth}
  \centering
\includegraphics[width=1\linewidth]{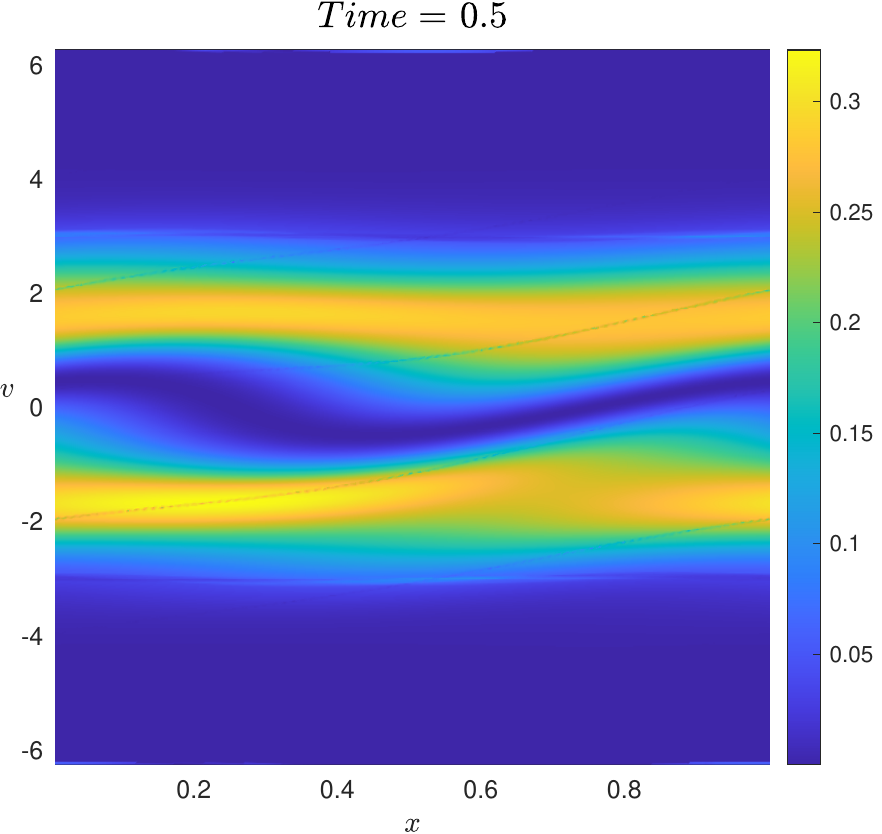}
\end{subfigure}%
\begin{subfigure}{.3\textwidth}
  \centering
  \includegraphics[width=1\linewidth]{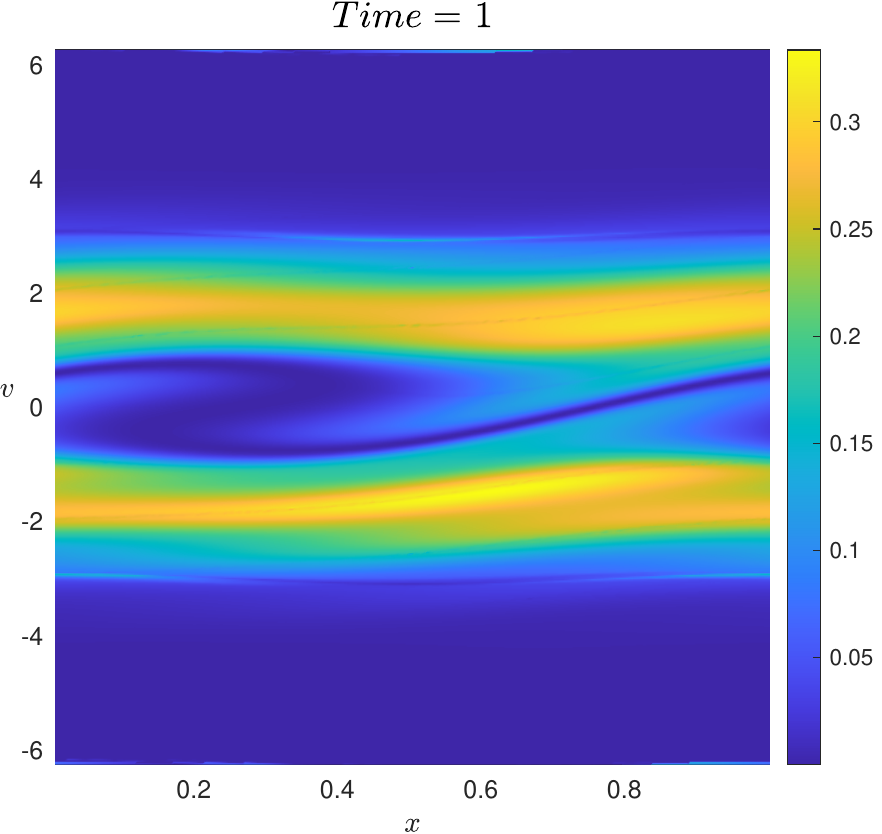}
\end{subfigure}\\[1ex]%
\begin{subfigure}{.3\textwidth}
  \centering
  \includegraphics[width=1\linewidth]{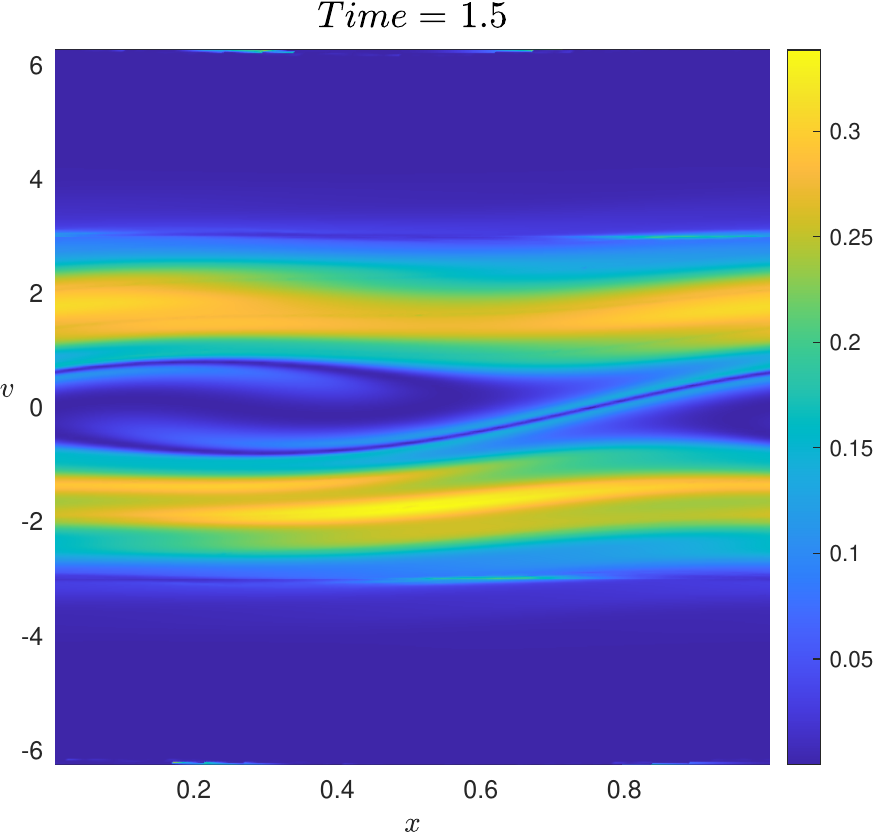}
\end{subfigure}
\begin{subfigure}{.3\textwidth}
  \centering
  \includegraphics[width=1\linewidth]{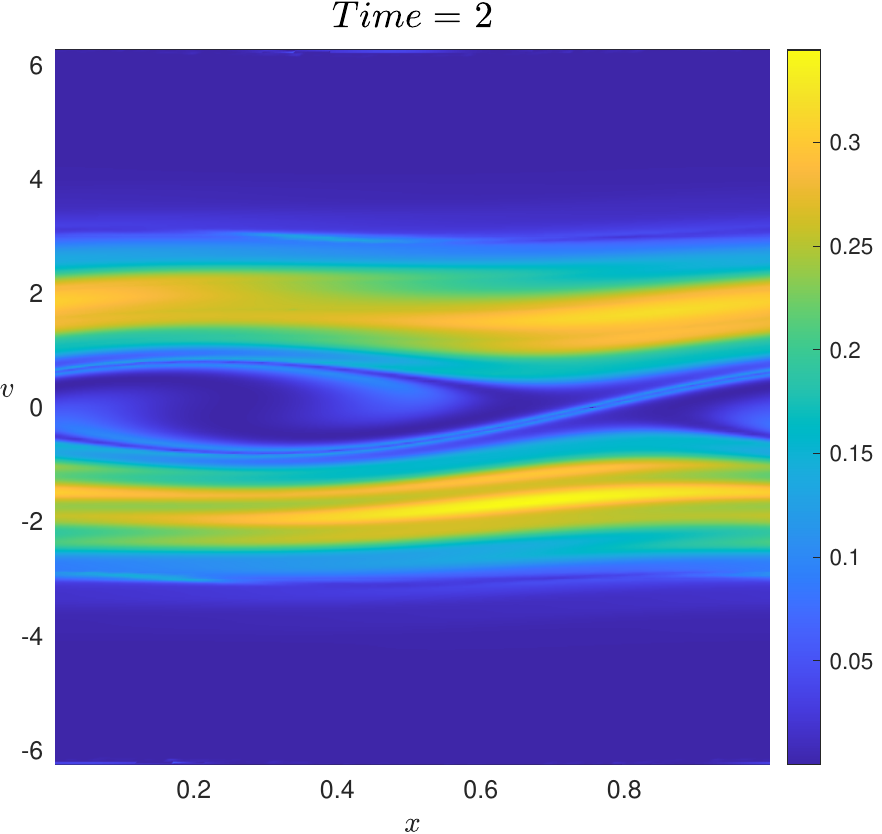}
\end{subfigure}
\begin{subfigure}{.3\textwidth}
  \centering
  \includegraphics[width=1\linewidth]{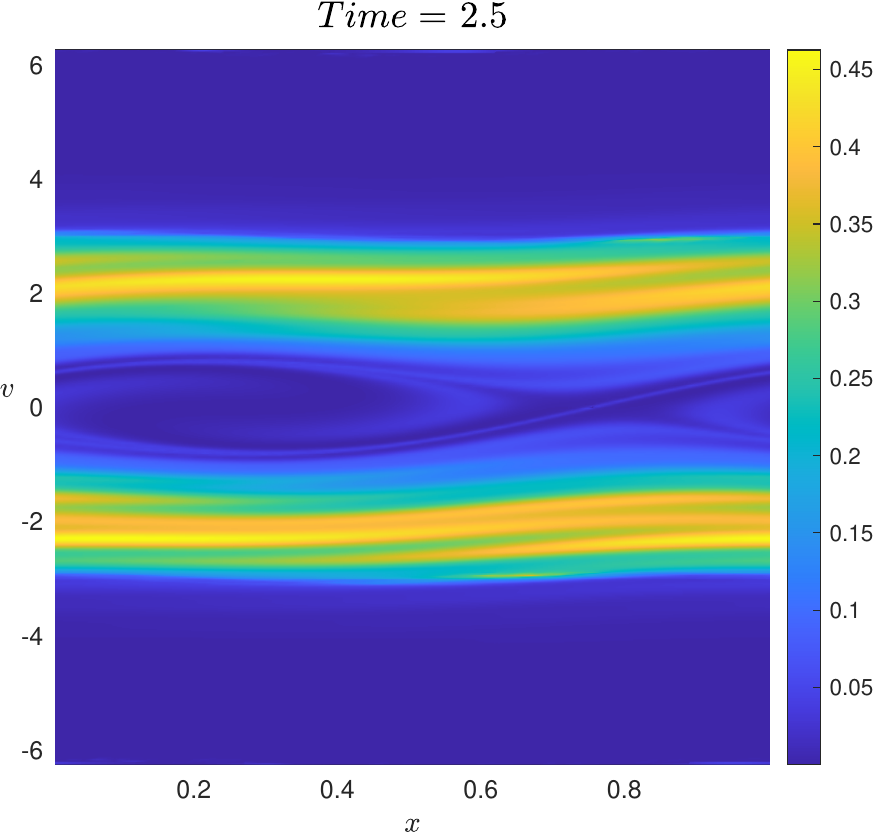}
\end{subfigure}
\caption{Snapshots: approximation of the solution of the stochastic PDE with multiplicative Stratonovich noise~\eqref{eq:SPDE-mult-Strato} with initial value $f_0$ given by~\eqref{eq:f0}, with $\sigma_1$ given by~\eqref{eq:sigma1cos} at times~$\{0,0.5,1,1.5,2,2.5\}$, using the Lie--Trotter splitting scheme~\eqref{eq:LT-mult-Strato} with time-step size $\tau=0.1$.}
\label{fig:snap-strato1}
\end{figure}

\begin{figure}[h]
\begin{subfigure}{.3\textwidth}
  \centering
\includegraphics[width=1\linewidth]{VlasovInit-eps-converted-to.pdf}
\end{subfigure}%
\begin{subfigure}{.3\textwidth}
  \centering
  \includegraphics[width=1\linewidth]{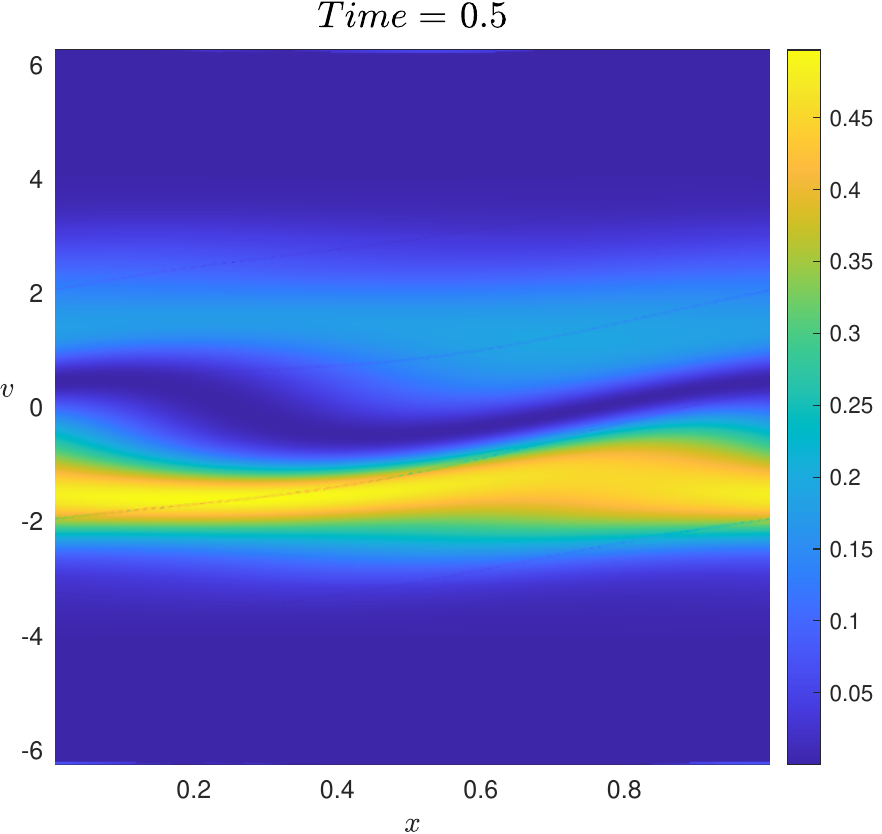}
\end{subfigure}%
\begin{subfigure}{.3\textwidth}
  \centering
  \includegraphics[width=1\linewidth]{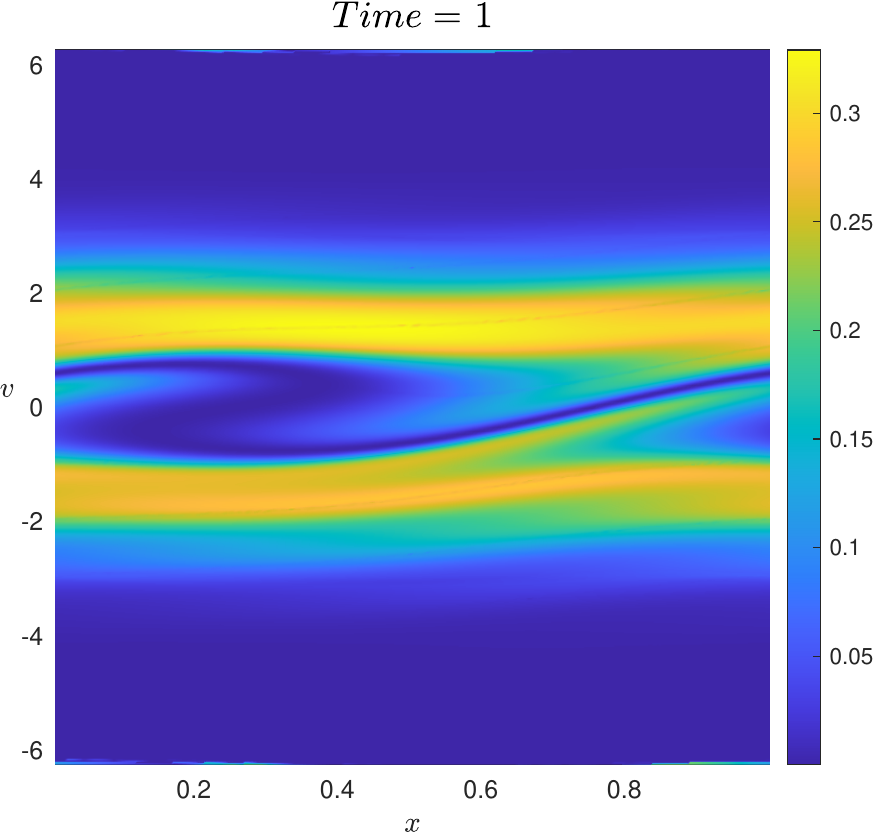}
\end{subfigure}\\[1ex]%
\begin{subfigure}{.3\textwidth}
  \centering
  \includegraphics[width=1\linewidth]{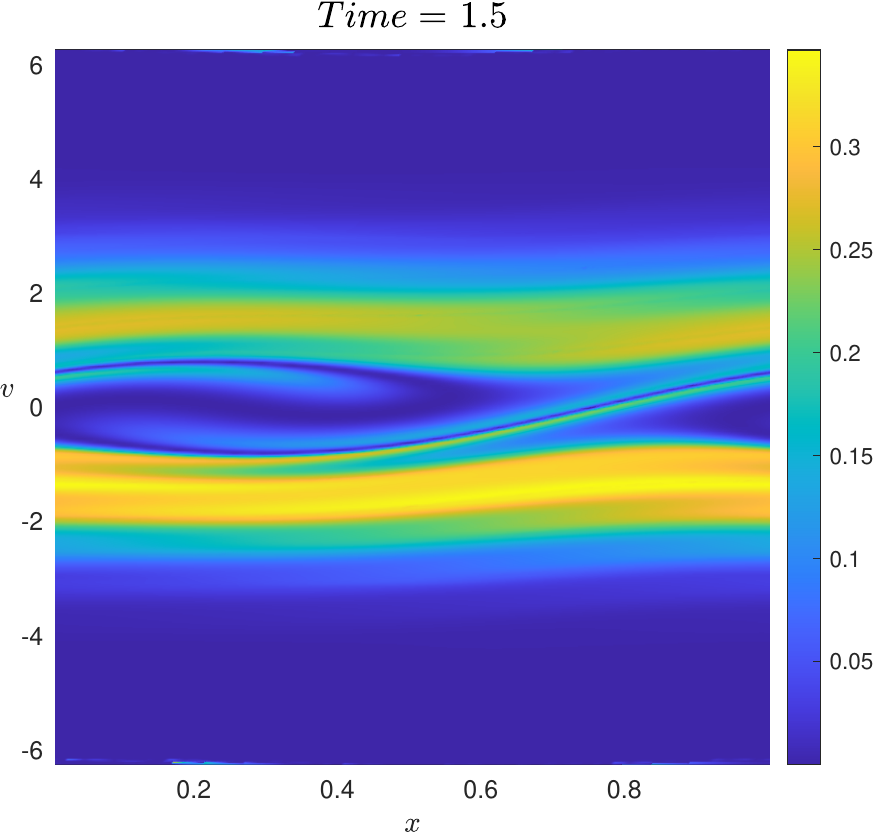}
\end{subfigure}
\begin{subfigure}{.3\textwidth}
  \centering
  \includegraphics[width=1\linewidth]{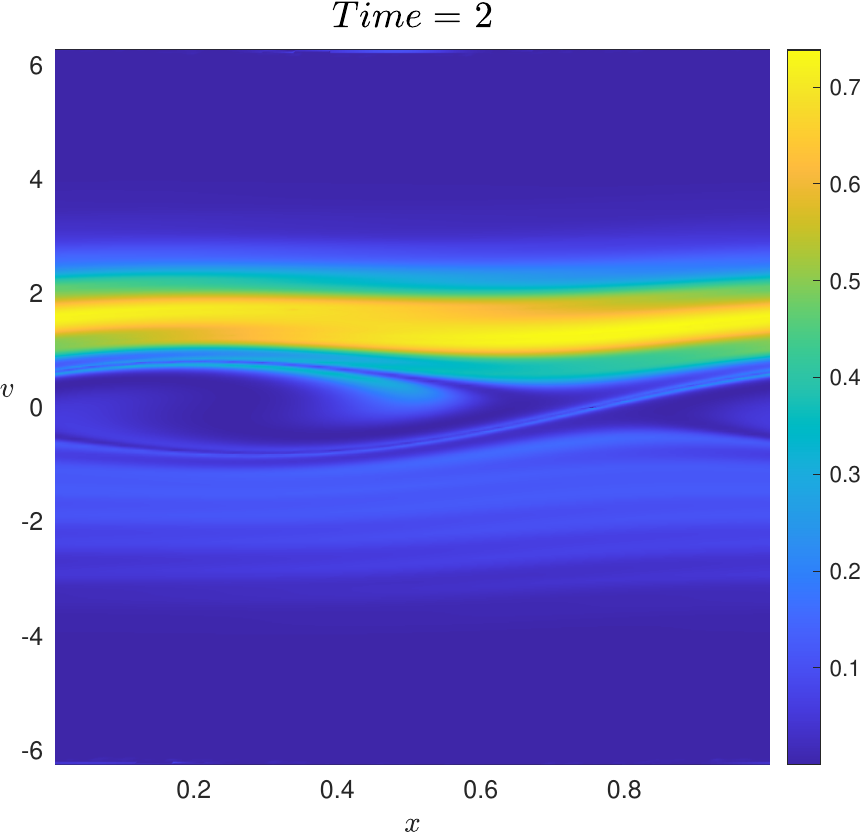}
\end{subfigure}
\begin{subfigure}{.3\textwidth}
  \centering
  \includegraphics[width=1\linewidth]{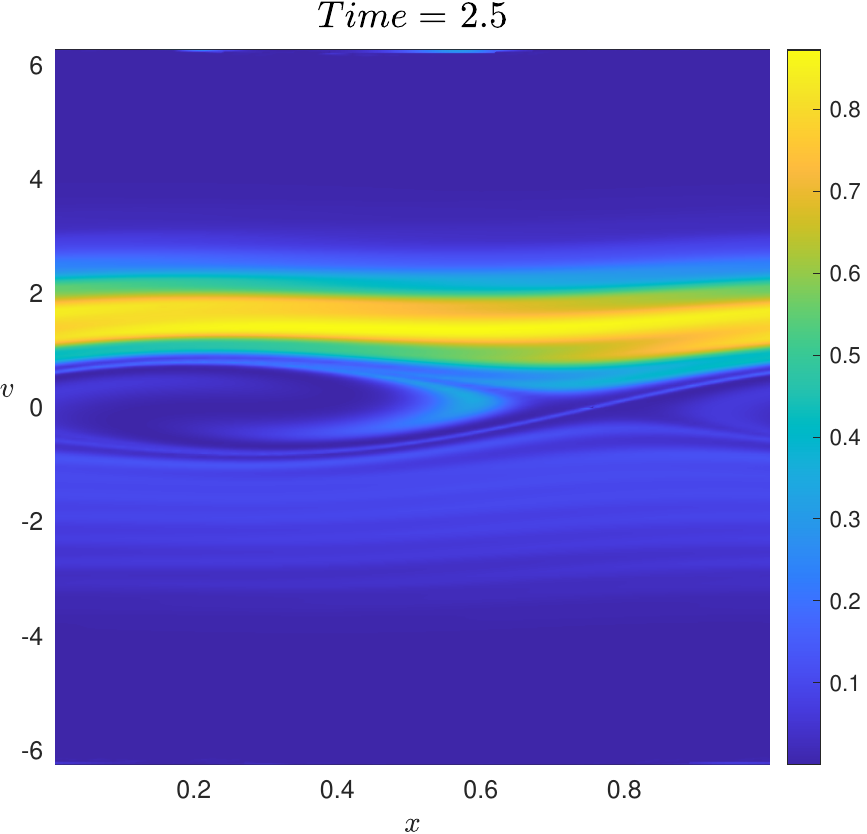}
\end{subfigure}
\caption{Snapshots: approximation of the solution of the stochastic PDE with multiplicative Stratonovich noise~\eqref{eq:SPDE-mult-Strato} with initial value $f_0$ given by~\eqref{eq:f0}, with $\sigma_1$ given by~\eqref{eq:sigma1sin} at times~$\{0,0.5,1,1.5,2,2.5\}$, using the Lie--Trotter splitting scheme~\eqref{eq:LT-mult-Strato} with time-step size $\tau=0.1$.}
\label{fig:snap-strato2}
\end{figure}

Next, we illustrate the evolution law of the $L^2$ norm stated in Proposition~\ref{propo:LT-mult-Strato}. In these experiments, one has $d=1$, $T=1$, $\delta x=\frac{1}{200}$, $\delta v=\frac{4\pi}{400}$ and $\tau=0.1$. The expectations are computed using an averaging procedure over $5.10^5$ samples. The results are presented in Figure~\ref{fig:normMultiStrato}, with different choices of the diffusion coefficients: $K=1$ with $\sigma_1(x,v)=1$, $K=1$ with $\sigma_1$ given by~\eqref{eq:sigma1sin}, and $K=2$ with $\sigma_1$ and $\sigma_2$ given by~\eqref{eq:sigma12}, respectively. Note that the condition~\eqref{eq:conditionsigma2constant} is satisfied in the first and in the third case. We observe a good agreement with the theoretical results given in Proposition~\ref{propo:LT-mult-Strato}.

\begin{figure}[h]
\begin{subfigure}{.5\textwidth}
  \centering
  \includegraphics[width=.9\linewidth]{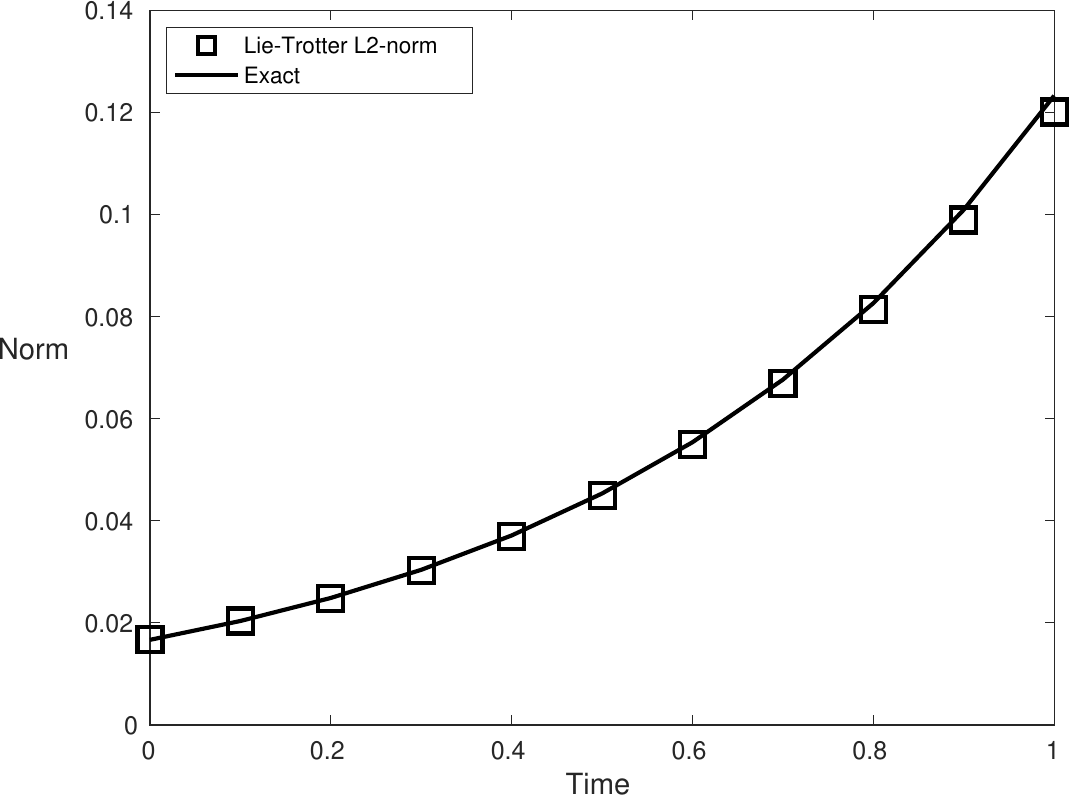}
  \caption{Noise given by $\sigma_1(x,v)=1$, $K=1$.}
\end{subfigure}%
\begin{subfigure}{.5\textwidth}
  \centering
  \includegraphics[width=.9\linewidth]{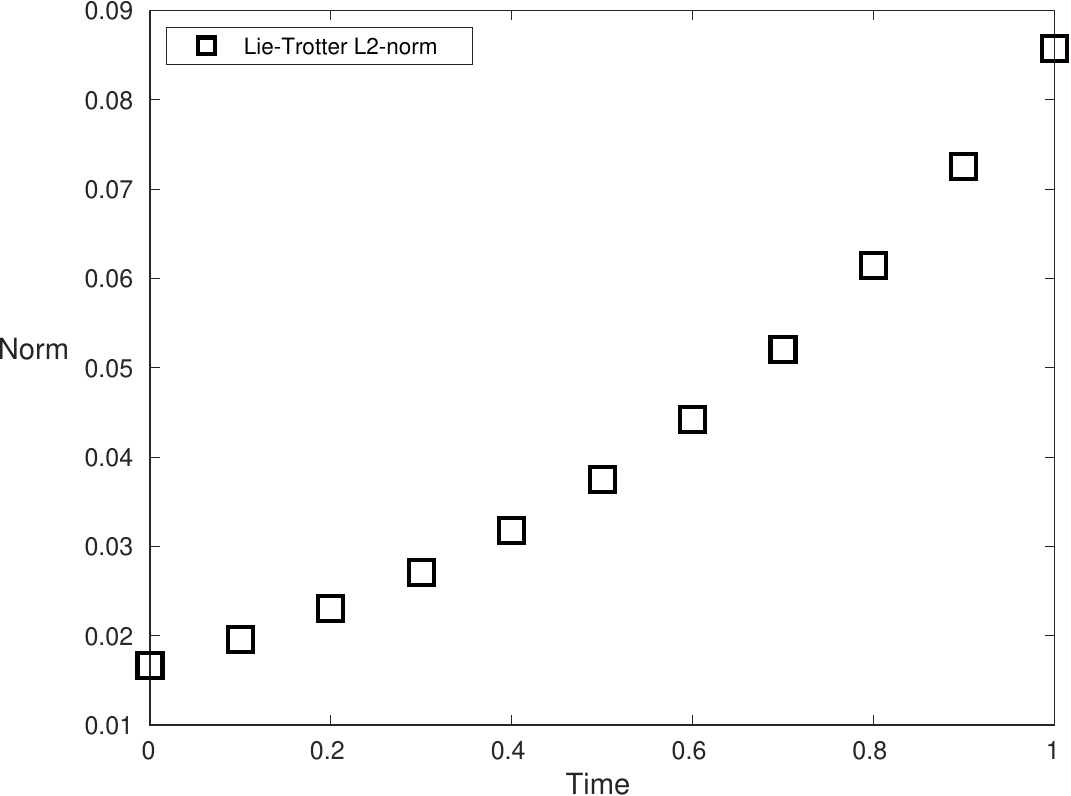}
  \caption{Noise given by~\eqref{eq:sigma1sin}, $K=1$.}
\end{subfigure}\\[1ex]%
\begin{subfigure}{.5\textwidth}
  \centering
  \includegraphics[width=.9\linewidth]{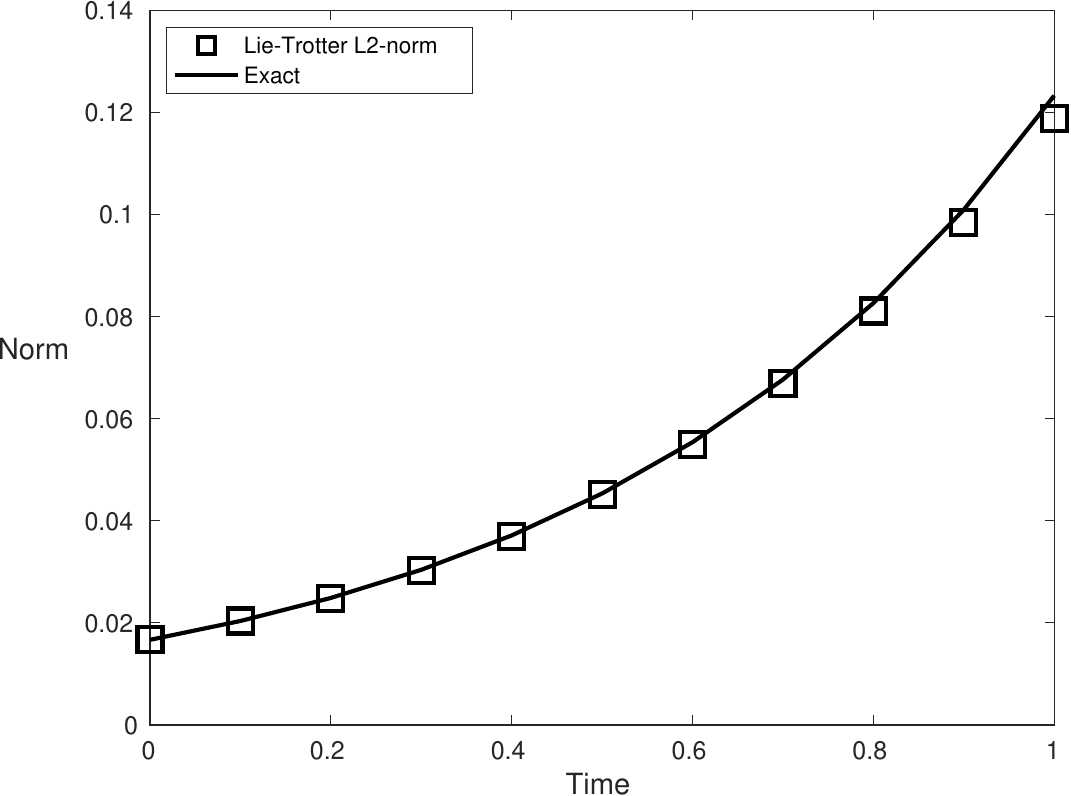}
  \caption{Noise given by~\eqref{eq:sigma12}, $K=2$.}
\end{subfigure}
\caption{Evolution law for the $L_{x,v}^2$ norm: illustration of Proposition~\ref{propo:LT-mult-Strato} when applying the Lie--Trotter scheme~\eqref{eq:LT-mult-Strato} with $\tau=0.1$ to the SPDE with multiplicative Stratonovich noise~\eqref{eq:SPDE-mult-Strato} with time-step size $\tau=0.1$.}
\label{fig:normMultiStrato}
\end{figure}

We conclude these numerical experiments in the multiplicative Stratonovich noise case by investigating the mean-square order of convergence of the Lie--Trotter splitting scheme~\eqref{eq:LT-mult-Strato}. The same procedure as in the multiplicative It\^o noise case (Section~\ref{sec:multiplicativeIto}) is applied. A reference solution is computed using the splitting scheme with time-step size $\tau_{\rm ref}=2^{-14}$, and the errors are computed when the time-step size $\tau$ takes values in $\{2^{-7},\ldots,2^{-13}\}$. The expectation is computed using a Monte Carlo averaging procedure over $500$ independent samples. The discretization parameters are
$\delta x=\frac{1}{100}$, $\delta v=\frac{4\pi}{200}$. The final time is $T=0.5$. The noise is given by~\eqref{eq:sigma1sin} or by~\eqref{eq:sigma12}. The results are presented in a loglog plot in Figure~\ref{fig:ms-MultStrato}. We observe a mean-square convergence order equal to $1$.

\begin{figure}[h]
\begin{subfigure}{.5\textwidth}
  \centering
  \includegraphics[width=.9\linewidth]{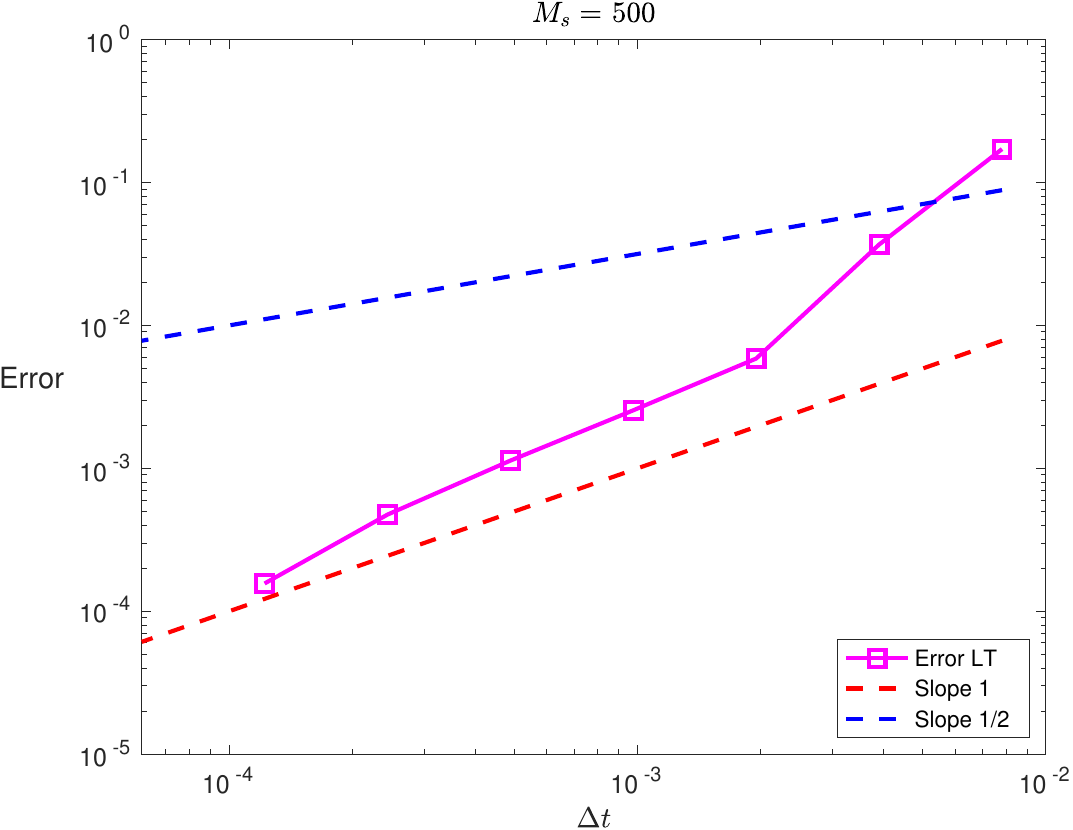}
  \caption{Noise: $\sigma_1(x,v)=\begin{cases} \sin(v) & \text{for}\quad |v|\leq 3 \\ 0 & \text{else}. \end{cases}$}
\end{subfigure}%
\begin{subfigure}{.5\textwidth}
  \centering
  \includegraphics[width=.9\linewidth]{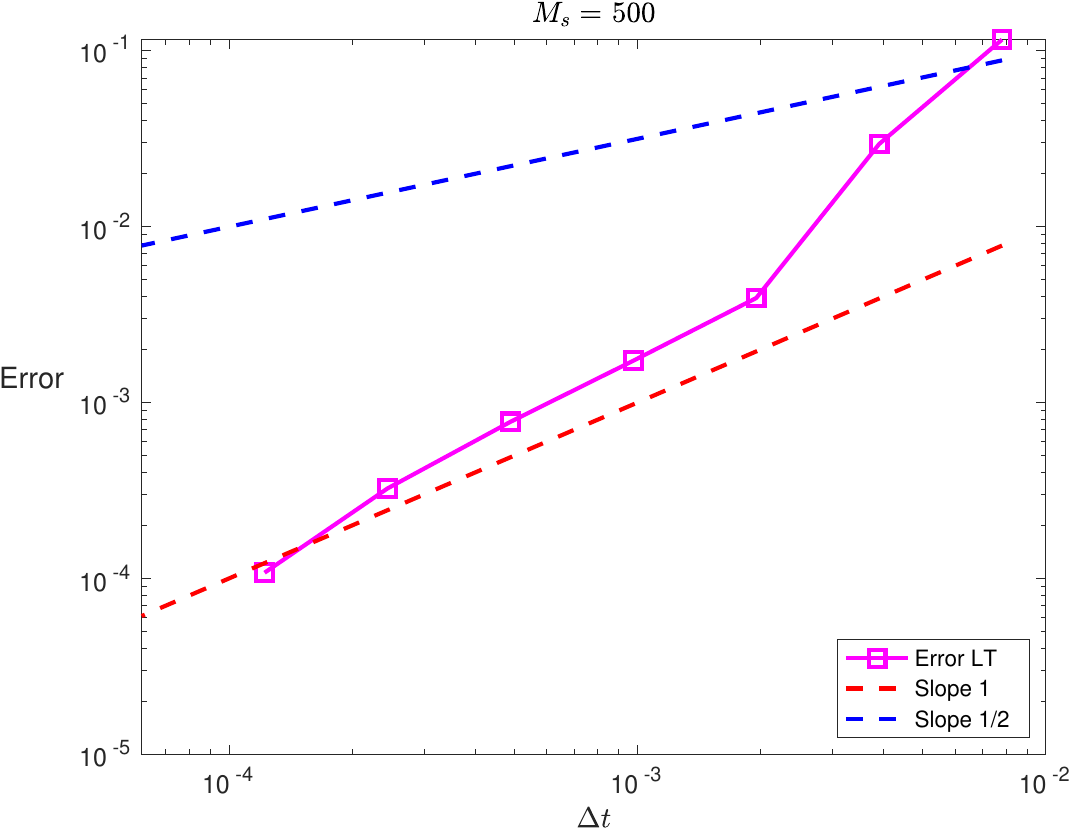}
  \caption{Noise: $\sigma_1(x,v)=\cos{(x)}$ and $\sigma_2(x,v)=\sin{(x)}$.}
\end{subfigure}
\caption{Mean-square errors: Lie--Trotter scheme~\eqref{eq:LT-mult-Strato} applied to the SPDE with multiplicative Stratonovich noise~\eqref{eq:SPDE-mult-Strato} driven by one-dimensional noise ($K=1$, left) and by two-dimensional noise ($K=2$, right).}
\label{fig:ms-MultStrato}
\end{figure}

\section{The stochastic linear Vlasov equation perturbed by transport noise}\label{sec:transport}

In this section, we consider a stochastic perturbation of the Vlasov partial differential equation ~\eqref{eq:Vlasov} where the noise is of transport type: for $t\ge 0, x\in\T^d, v\in \R^d$
\begin{equation}\label{eq:SPDE-transport}
\left\lbrace
\begin{aligned}
&\text df^\tr(t,x,v)+v\cdot\nabla_xf^\tr(t,x,v)\,\text dt+E(x)\cdot\nabla_v f^\tr(t,x,v)\,\text dt+\nabla_vf^\tr(t,x,v)\odot \text d\bm{W}(t,x)=0~,\\
&f^\tr(0,x,v)=f_0(x,v)~,
\end{aligned}
\right.
\end{equation}
where the noise $\bm{W}(t,x)$ is given by~\eqref{eq:Wtx} and the notation
\begin{align*}
\nabla_vf^\tr(t,x,v)\odot \text d\bm{W}(t,x)&=\sum_{k=1}^{K}\nabla_vf^\tr(t,x,v)\cdot \bm{\sigma}_{k}(x)\circ \text d\beta_k(t)\\
&=\sum_{j=1}^{d}\sum_{k=1}^{K}\partial_{v_j}f^\tr(t,x,v)\sigma_{j,k}(x) \circ \text d\beta_k(t),
\end{align*}
is used, with $\circ$ denoting the Stratonovich product and $\odot$ denoting the combination of the Stratonovich product $\circ$ and of the inner product $\cdot$ in the Euclidean space $\R^d$.

To identify the equivalent It\^o formulation of the stochastic partial differential equation~\eqref{eq:SPDE-transport}, it is convenient to set for all $x\in\T^d$ and $1\le i,j\le d$
\[
a_{i,j}(x)=\sum_{k=1}^{K}\sigma_{i,k}(x)\sigma_{j,k}(x).
\]
The It\^o formulation of~\eqref{eq:SPDE-transport} is then the following: for $t\ge 0, x\in\T^d, v\in \R^d$
\begin{equation}\label{eq:SPDE-transport-Ito}
\left\lbrace
\begin{aligned}
&\text df^\tr(t,x,v)+v\cdot\nabla_xf^\tr(t,x,v)\,\text dt+E(x)\cdot\nabla_v f^\tr(t,x,v)\,\text dt+\nabla_vf^\tr(t,x,v)\cdot \text d\bm{W}(t,x)\\
&=\frac12\sum_{k=1}^{K}(\bm{\sigma}_k(x)\cdot \nabla_v)^2 f^\tr(t,x,v)\,\text dt~,\\
&f^\tr(0,x,v)=f_0(x,v)~,
\end{aligned}
\right.
\end{equation}
with the notation
\begin{align*}
\nabla_vf^\tr(t,x,v)\cdot \text d\bm{W}(t,x)&=\sum_{j=1}^{d}\sum_{k=1}^{K}\partial_{v_j}f^\tr(t,x,v)\sigma_{j,k}(x) \text d\beta_k(t)\\
\sum_{k=1}^{K}(\bm{\sigma}_k(x,v)\cdot \nabla_v)^2 f^\tr(t,x,v)&=\sum_{i,j=1}^{d}\sum_{k=1}^{K}\sigma_{i,k}(x)\sigma_{j,k}(x)\partial_{v_i}\partial_{v_j}f^\tr(t,x,v)\\
&=\sum_{i,j=1}^{d}a_{i,j}(x)\partial_{v_i}\partial_{v_j}f^\tr(t,x,v).
\end{align*}

\subsection{Analysis and properties of the problem}
Recall that the deterministic linear Vlasov equation~\eqref{eq:Vlasov} is connected with the ordinary differential equation~\eqref{eq:ODE}, see Section~\ref{sec:deterministic}. This connection is explored in Sections~\ref{sec:additive} and~\ref{sec:multiplicative} for some stochastic perturbations of~\eqref{eq:Vlasov}. Contrary to those situations, for the stochastic linear Vlasov equation driven by transport noise~\eqref{eq:SPDE-transport}, the connection requires to introduce the stochastic differential equation
\begin{equation}\label{eq:SDE}
\left\lbrace
\begin{aligned}
&\text dX_t=V_t\\
&\text dV_t=E(X_t)\,\text dt+\sum_{k=1}^{K}\bm{\sigma}_k(X_t)\,\text d\beta_k(t),
\end{aligned}
\right.
\end{equation}
for all $t\ge 0$, instead of the ordinary differential equation~\eqref{eq:ODE}. In the second line of~\eqref{eq:SDE}, one may use the notation $\text d\bm{W}(t,X_t)$ to refer to $\sum_{k=1}^{K}\bm{\sigma}_k(X_t)\,\text d\beta_k(t)$.

Assume that $\bigl(f^\tr(t)\bigr)_{t\ge 0}$ is a sufficiently regular solution of the SPDE~\eqref{eq:SPDE-transport} (or of its equivalent formulation~\eqref{eq:SPDE-transport-Ito}), then for any solution $\bigl(X_t,V_t\bigr)_{t\ge 0}$ of the SDE system~\eqref{eq:SDE}, applying the It\^o--Wentzell formula (see \autoref{sec:ItoWentzell}) one obtains the identity
\[
\text df^\tr(t,X_t,V_t)=0.
\]
This means that the SDE~\eqref{eq:SDE} provides characteristic curves for the stochastic Vlasov equation~\eqref{eq:SPDE-transport} driven by transport noise. This gives a strategy to solve the SPDE~\eqref{eq:SPDE-transport} by the method of lines. Instead of using the flow $\bigl(\phi_t\bigr)_{t\in\R}$ associated with the ODE~\eqref{eq:ODE}, in the present case this strategy is based on the notion of stochastic flow of diffeomorphisms, see for instance the monograph~\cite{MR1472487}. We use the notation $\bigl(\psi_t\bigr)_{t\ge 0}$ to denote the stochastic flow of diffeomorphisms associated with the SDE~\eqref{eq:SDE}. In particular, for any (deterministic) initial values $(X_0,V_0)\in\T^d\times\R^d$, the unique solution of the SDE~\eqref{eq:SDE} at any time $t\ge 0$ is given by $(X_t,V_t)=\psi_t(X_0,V_0)$. In addition, almost surely the mapping $\psi_t$ preserves the volume in $\T^d\times\R^d$ for all $t\ge 0$.

It results from the identity above that, for all $t\ge 0$ and all (deterministic) $X_0\in\T^d$ and $V_0\in\R^d$, one has
\[
f^\tr(t,\psi_t(X_0,V_0))=f^\tr(t,X_t,V_t)=f^\tr(0,X_0,V_0)=f_0(X_0,V_0).
\]
Finally one obtains the expression of the solution of~\eqref{eq:SPDE-transport} using characteristic curves: for all $t\ge 0$, $x\in\T^d$, $v\in\R^d$ one has
\begin{equation}\label{eq:solutionTransport}
f^\tr(t,x,v)=f_0(\psi_t^{-1}(x,v)).
\end{equation}
Conversely, if $(t,x,v)\in\R^+\times\T^d\times\R^d\mapsto f^\tr(t,x,v)$ is defined by~\eqref{eq:solutionTransport} and if the initial value $f_0$ is of class $\mathcal{C}^2$, then the stochastic process $\bigl(f^\tr(t)\bigr)_{t\ge 0}$ defined by~\eqref{eq:solutionTransport} is a weak solution of the SPDE~\eqref{eq:SPDE-transport-Ito} (the It\^o formulation is considered): for any smooth compactly supported function $\varphi:\T^d\times\R^d\to\R$, one has
\begin{equation}\label{eq:weaksolutionTransport}
\begin{aligned}
\text d\Bigl(\iint \varphi(x,v)f(t,x,v)\,\text dx\,\text dv\Bigr)&=\iint v\cdot \nabla_x\varphi(x,v)f(t,x,v)\,\text dx\,\text dv\,\text dt\\
&+\iint E(x)\cdot \nabla_v\varphi(x,v)f(t,x,v)\,\text dx\,\text dv\,\text dt\\
&+\iint \sum_{k=1}^{K}\sigma_k(x)\cdot \nabla_v\varphi(x,v) f(t,x,v)\,\text dx\,\text dv\,\text d\beta_k(t)\\
&+\frac12\iint\sum_{i,j=1}^{d}a_{i,j}(x)\partial_{v_i}\partial_{v_j}\varphi(x,v)f(t,x,v)\,\text dx\,\text dv\,\text dt.
\end{aligned}
\end{equation}

The proof of the identity~\eqref{eq:weaksolutionTransport} combines two arguments. First, since the diffeomorphism $\psi_t$ preserves the volume of $\T^d\times\R^d$ for all $t\ge 0$, using~\eqref{eq:solutionTransport} one has the identity
\[
\begin{aligned}
\iint \varphi(x,v)f(t,x,v)\,\text dx\,\text dv&=\iint \varphi(\psi_t(x,v))f(0,x,v)\,\text dx\,\text dv\\
&=\iint \varphi(X_t^{x,v},V_t^{x,v})f(0,x,v)\,\text dx\,\text dv,
\end{aligned}
\]
where $(X_t^{x,v},V_t^{x,v})=\psi_t(x,v)$ is the solution at time $t\ge 0$ of~\eqref{eq:SDE} with initial values $X_0^{x,v}=x$ and $V_ 0^{x,v}=v$. It then remains to apply the standard It\^o formula to $t\mapsto \varphi(X_t^{x,v},V_t^{x,v})$ and to use integration by parts arguments combined with the identity above obtained by the first argument to obtain~\eqref{eq:weaksolutionTransport}. The details are omitted.
We refer to~\cite{MR3513594} for instance where this type of arguments are employed in the analysis of a class of nonlinear transport equations with transport noise.

Using the expression~\eqref{eq:solutionTransport} of the solution of~\eqref{eq:SPDE-transport}, it is straightforward to check that the preservation properties satisfied in the deterministic case by the solution $\bigl(f^\de(t)\bigr)_{t\ge 0}$ of the deterministic linear Vlasov equation~\eqref{eq:Vlasov} are also satisfied when a transport noise perturbation is applied. The proof is omitted.

\begin{proposition}\label{propo:transport}
Let $\bigl(f^\tr(t)\bigr)_{t\ge 0}$ be the solution of the SPDE~\eqref{eq:SPDE-transport} with (non-random) initial value $f_0$. One has the following properties.
\begin{itemize}
\item \emph{Preservation of positivity}. Assume that $f_0(x,v)\ge 0$ for all $(x,v)\in\T^d\times\R^d$. Then one has $f^\tr(t,x,v)\ge 0$ almost surely for all $t\ge 0$ and $(x,v)\in\T^d\times\R^d$.
\item \emph{Preservation of integrals}. Let $\Phi:\R\to\R^+$ be a real-valued measurable mapping. Then for all $t\ge 0$ one has almost surely
\[
\iint \Phi(f^\tr(t,x,v))\,\text dx\, \text dv=\iint \Phi(f_0(x,v))\,\text dx\, \text dv.
\]
In particular, if $p\in[1,\infty)$ and if $f_0\in L_{x,v}^p$, then for all $t\ge 0$ one has $f^\tr(t)\in L_{x,v}^p$ almost surely and
\[
\|f^\tr(t)\|_{L_{x,v}^p}=\|f_0\|_{L_{x,v}^p}.
\]
\end{itemize}
\end{proposition}

\subsection{Splitting scheme}
We are now in position to introduce the proposed numerical scheme for the discretization of the SPDE with transport noise~\eqref{eq:SPDE-transport}. Like in the previous sections, a Lie--Trotter splitting strategy is applied. To deal with the stochastic perturbation in~\eqref{eq:SPDE-transport}, one needs to consider the auxiliary stochastic subsystem
\[
\text df(t,x,v)+\nabla_vf(t,x,v)\odot \text d\bm{W}(t,x)=0~,\quad (t,x,v)\in\R^+\times\T^d\times\R^d.
\]
This auxiliary stochastic subsystem can be solved exactly: for all $t\ge s\ge 0$ and all $(x,v)\in\T^d\times\R^d$ one has
\[
f(t,x,v)=f\bigl(s,x,v-({\bm W}(t,x)-{\bm W}(s,x))\bigr).
\]
Alternative notation for solving the auxiliary stochastic subsystem above can be used: for all $t\ge s\ge 0$ one has
\[
f(t)=T({\bm W}(t)-{\bm W}(s))f(s)=T^d(\beta_d(t)-\beta_d(s))\ldots T^1(\beta_1(t)-\beta_1(s))f(s),
\]
where for any mapping $\bm{\sigma}:x\in\T^d\mapsto \sigma(x)\in\R^d$ one has
\[
T(\bm{\sigma})f(x,v)=f(x,v-\bm{\sigma}(x))~,\quad x\in\T^d, v\in\R^d
\]
and where the auxiliary linear operators $T^k(y)$, with $y\in\R$ and $1\le k\le K$, are defined by
\[
T^k(y)f(x,v)=f(x,v-y\sigma_k(x))~,\quad x\in\T^d, v\in\R^d.
\]
Using the Lie--Trotter integrator~\eqref{eq:LTdeter} for the deterministic part and combining the discretizations of the deterministic and stochastic parts yields the following scheme: given the initial value $f_0$ and the time-step size $\tau\in(0,1)$, set $f_0^\tr=f_0$ and for any nonnegative integer $n\ge 0$ set
\begin{equation}\label{eq:LT-transport}
f_{n+1}^\tr=T(\delta \bm{W}_{n})S^2(\tau)S^1(\tau)f_n^\tr=T^{d}(\delta\beta_{n,d})\circ \cdots \circ T^1(\delta\beta_{n,1})\circ S^2(\tau)S^1(\tau)f_n^\tr~,\quad n\ge 0.
\end{equation}
where the Wiener increments $\delta\beta_{n,k}$ and $\delta\bm{W}_n$ are given by~\eqref{eq:WincBeta} and~\eqref{eq:WincbW} respectively, see Section~\ref{sec:notation}.

The Lie--Trotter splitting scheme~\eqref{eq:LT-transport} is consistent with Stratonovich interpretation of the noise in the SPDE~\eqref{eq:SPDE-transport}, and satisfies the same properties as the exact solution stated in Proposition~\ref{propo:transport}.
\begin{proposition}\label{propo:transportnum}
Let $(f_n^\tr)_{n\ge 0}$ be the solution of the Lie--Trotter splitting scheme~\eqref{eq:LT-transport} with initial value $f_0$.
One then has the following properties.
\begin{itemize}
\item \emph{Preservation of positivity}. Assume that $f_0(x,v)\ge 0$ for all $(x,v)\in\T^d\times\R^d$. Then one has $f_n^\tr(x,v)\ge 0$ almost surely for any nonnegative integer $n\ge 0$ and all $(x,v)\in\T^d\times\R^d$.
\item \emph{Preservation of integrals}. Let $\Phi:\R\to\R^+$ be a real-valued measurable mapping. Then for all $n\ge 0$ one has almost surely
\[
\iint \Phi(f_n^\tr(x,v))\,\text dx\, \text dv=\iint \Phi(f_0(x,v))\,\text dx\, \text dv.
\]
In particular, if $p\in[1,\infty)$ and if $f_0\in L_{x,v}^p$, then for all $n\ge 0$ one has $f_n^\tr\in L_{x,v}^p$ almost surely and
\[
\|f_n^\tr\|_{L_{x,v}^p}=\|f_0\|_{L_{x,v}^p}
\]
almost surely.
\end{itemize}
\end{proposition}

\begin{proof}[Proof of Proposition~\ref{propo:transportnum}]
All the results are proved by recursion, using the fact that the operators $T(\delta \bm{W}_n)$, $S^2(\tau)$ and $S^1(\tau)$ satisfy the considered properties.
\end{proof}

Like for the splitting scheme~\eqref{eq:LTdeter} defined in Section~\ref{sec:deterministic} for the discretization of the deterministic linear Vlasov equation~\eqref{eq:Vlasov}, the Lie--Trotter splitting scheme~\eqref{eq:LT-transport} can be interpreted as a discrete version of the expression~\eqref{eq:solutionTransport} for the exact solution of~\eqref{eq:SPDE-transport}, where the stochastic flow $\psi_t$ is approximated using a splitting integrator applied to the SDE system~\eqref{eq:SDE}.

\begin{remark}\label{rem:temporalnoiseTransport}
If the noise in the SPDE~\eqref{eq:SPDE-transport} is a purely temporal Wiener process, i.e. ${\bm W}(t,x)={\bm W}(t)=\sum_{k=1}^{K}\beta_k(t)\bm{\sigma}_k$ for all $(t,x)\in\R^+\times\T^d$ where $\bm{\sigma}_1,\ldots,\bm{\sigma}_K$ are elements of $\R^d$, then the exact solution of~\eqref{eq:SPDE-transport} and the numerical solution given by~\eqref{eq:LT-transport} can be written
\[
f^\tr(t)=T({\bm W}(t))f^\de(t)~,\quad t\ge 0~;\quad f_n^\tr=T({\bm W}(t_n))f_n^\de~,\quad n\ge 0,
\]
where the auxiliary continuous-time process $\bigl(f^\de(t)\bigr)_{t\ge 0}$ is the solution of the deterministic PDE~\eqref{eq:Vlasov} and the auxiliary continuous-time process $\bigl(f_n^\de\bigr)_{n\ge 0}$ is given by the deterministic Lie--Trotter splitting scheme~\eqref{eq:LTdeter}.

In that situation, Propositions~\ref{propo:transport} and~\ref{propo:transportnum} are straightforward consequences of the results described for the deterministic problem in Section~\ref{sec:deterministic}.
\end{remark}

\subsection{Numerical experiments}

We begin the numerical experiments by illustrating the behavior of the linear Vlasov equation perturbed by transport noise~\eqref{eq:SPDE-transport}, in dimension $d=1$. In all the experiments below, the initial value $f_0$ is given by~\eqref{eq:f0} and the vector field $E$ is given by~\eqref{eq:E}. The discretization parameters are given by $\delta x=\frac{1}{400}$, $\delta v=\frac{4\pi}{800}$, and $\tau=0.1$. The snapshots of the numerical solution at times $\{0,0.5,1,1.5,2,2.5\}$ computed using the splitting scheme~\eqref{eq:LT-transport} are provided in Figure~\ref{fig:snap-transp}, with $K=1$ and diffusion coefficient $\sigma_1(x,v)=0.5$. One observes that the solution remains nonnegative, which illustrates the positivity preserving property stated in Proposition~\ref{propo:transportnum} on the considered realization.

\begin{figure}[h]
\begin{subfigure}{.3\textwidth}
  \centering
\includegraphics[width=1\linewidth]{VlasovInit-eps-converted-to.pdf}
\end{subfigure}%
\begin{subfigure}{.3\textwidth}
  \centering
  \includegraphics[width=1\linewidth]{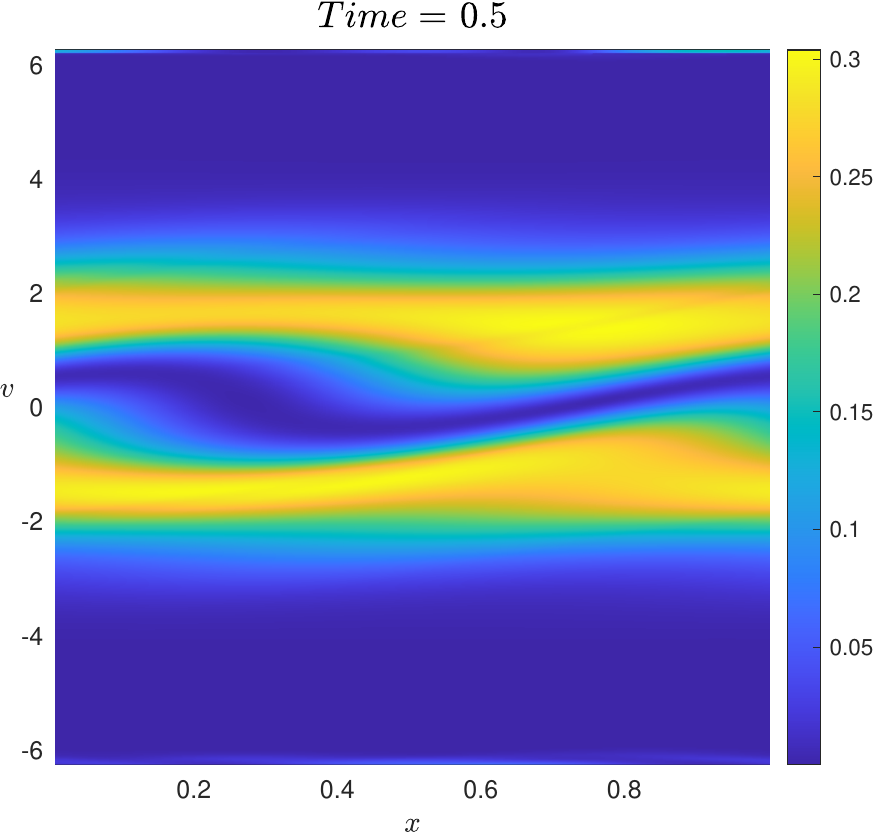}
\end{subfigure}%
\begin{subfigure}{.3\textwidth}
  \centering
  \includegraphics[width=1\linewidth]{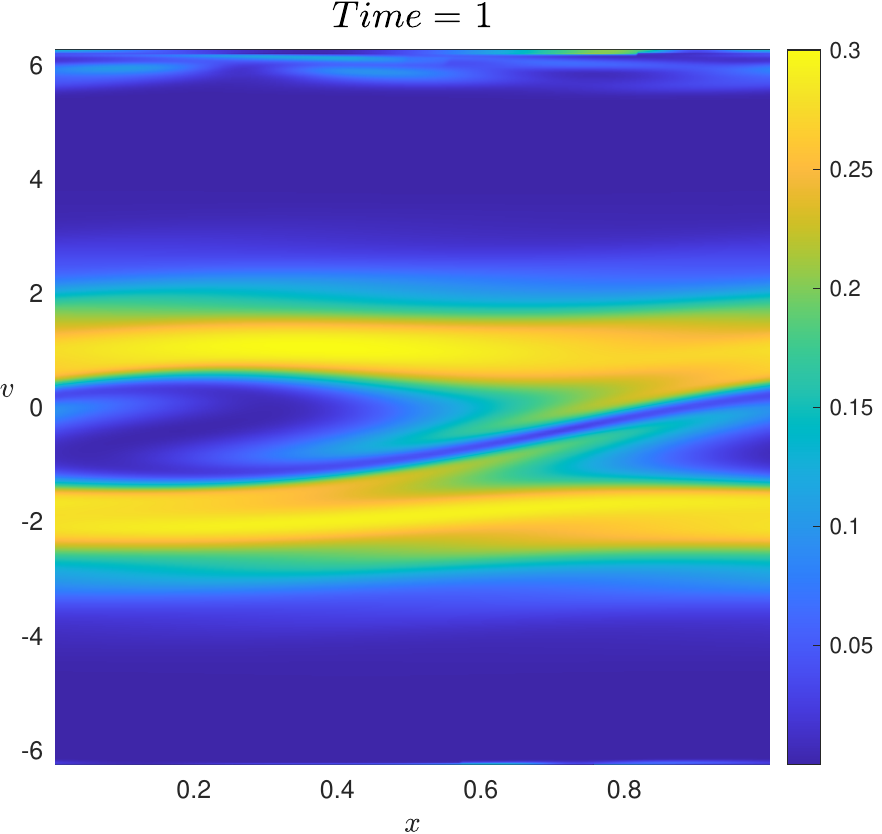}
\end{subfigure}\\[1ex]%
\begin{subfigure}{.3\textwidth}
  \centering
  \includegraphics[width=1\linewidth]{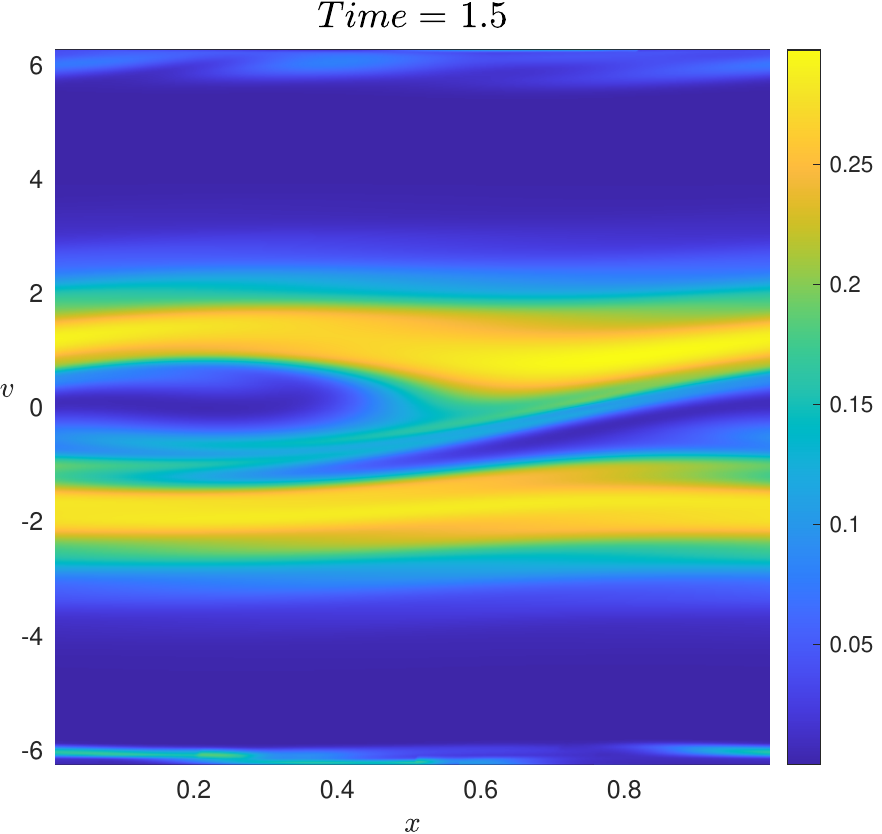}
\end{subfigure}
\begin{subfigure}{.3\textwidth}
  \centering
  \includegraphics[width=1\linewidth]{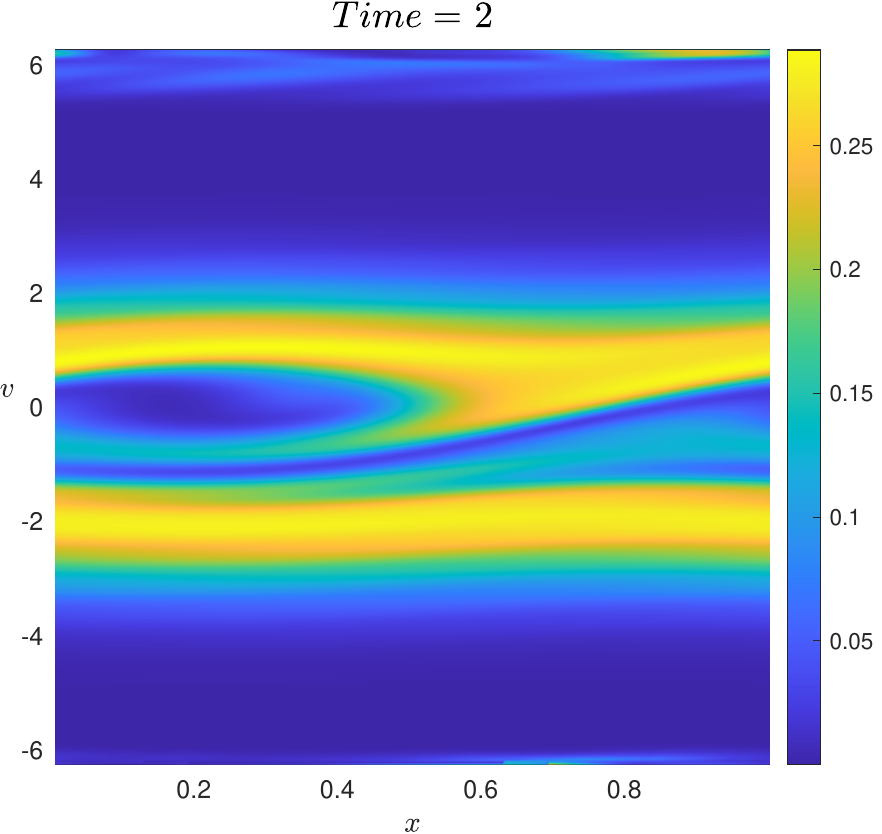}
\end{subfigure}
\begin{subfigure}{.3\textwidth}
  \centering
  \includegraphics[width=1\linewidth]{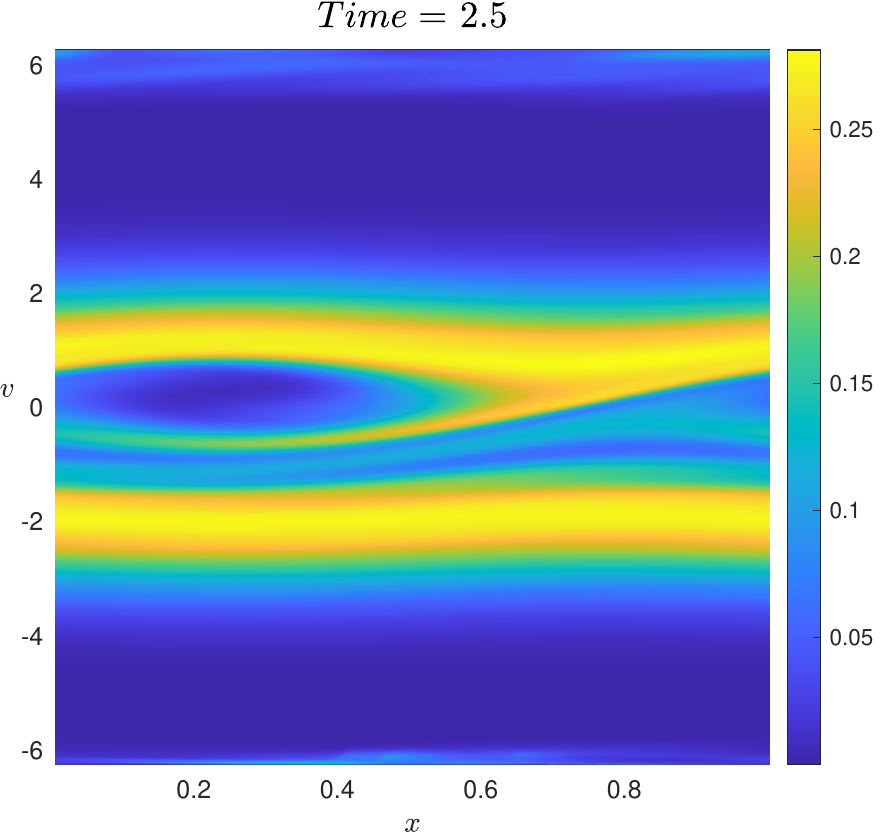}
\end{subfigure}
\caption{Snapshots: approximation of the solution of the stochastic PDE with transport noise~\eqref{eq:SPDE-transport} with initial value $f_0$ given by~\eqref{eq:f0}, with $\sigma_1(x,v)=0.5$ at times~$\{0,0.5,1,1.5,2,2.5\}$, using the Lie--Trotter splitting scheme~\eqref{eq:LT-transport} with time-step size $\tau=0.1$.}
\label{fig:snap-transp}
\end{figure}

Next, in Figure~\ref{fig:normTransport}, we illustrate the almost sure preservation of the $L^1$, $L^3$ and $L^{55}$ norms when the Lie--Trotter splitting scheme~\eqref{eq:LT-transport} is applied to the SPDE~\eqref{eq:SPDE-transport}. This is performed for three independent realizations. The discretization parameters are chosen to be: $\delta x=\frac{1}{3000}$, $\delta v=\frac{4\pi}{3000}$, and $\tau=0.1$. The final time is $T=1$.
In this figure, we observe the preservation of the norms, which illustrates Proposition~\ref{propo:transportnum}.

\begin{figure}[h]
\centering
\includegraphics[width=.5\linewidth]{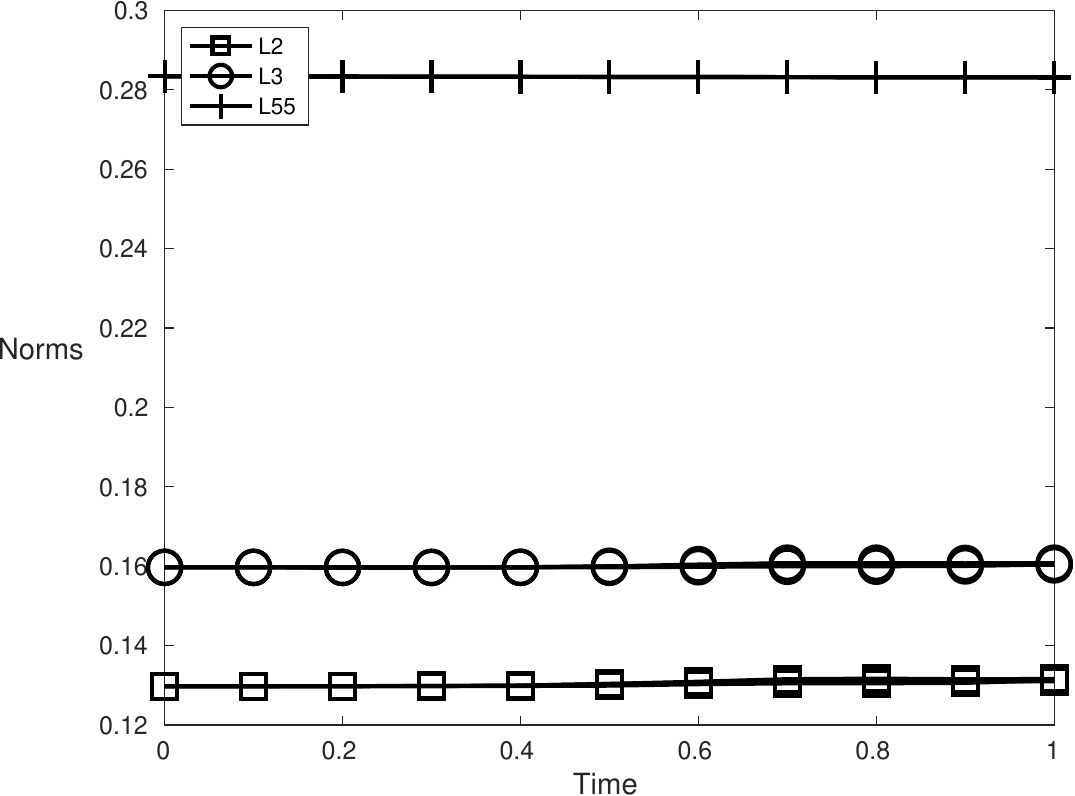}
\caption{Evolution of the $L_{x,v}^1, L_{x,v}^3$ and $L_{x,v}^{55}$ norms: illustration of Proposition~\ref{propo:transportnum} when applying the Lie--Trotter scheme~\eqref{eq:LT-transport} to the SPDE with transport noise~\eqref{eq:SPDE-transport}, with time-step size $\tau=0.1$.}
\label{fig:normTransport}
\end{figure}

We conclude these numerical experiments in the transport noise case by investigating the mean-square order of convergence of the Lie--Trotter splitting scheme~\eqref{eq:LT-transport}. The same procedure as in the previous cases is applied. A reference solution is computed using the splitting scheme with time-step size $\tau_{\rm ref}=2^{-14}$, and the errors are computed when the time-step size $\tau$ takes values in $\{2^{-7},\ldots,2^{-13}\}$. The expectation is computed using a Monte Carlo averaging procedure over $500$ independent samples. The discretization parameters are $\delta x=\frac{1}{100}$, $\delta v=\frac{4\pi}{200}$. The final time is $T=0.5$.
The results are presented in a loglog plot in Figure~\ref{fig:ms-transport}.
We observe a mean-square convergence order of at least $1/2$.

\begin{figure}[h]
\centering
\includegraphics[width=.5\linewidth]{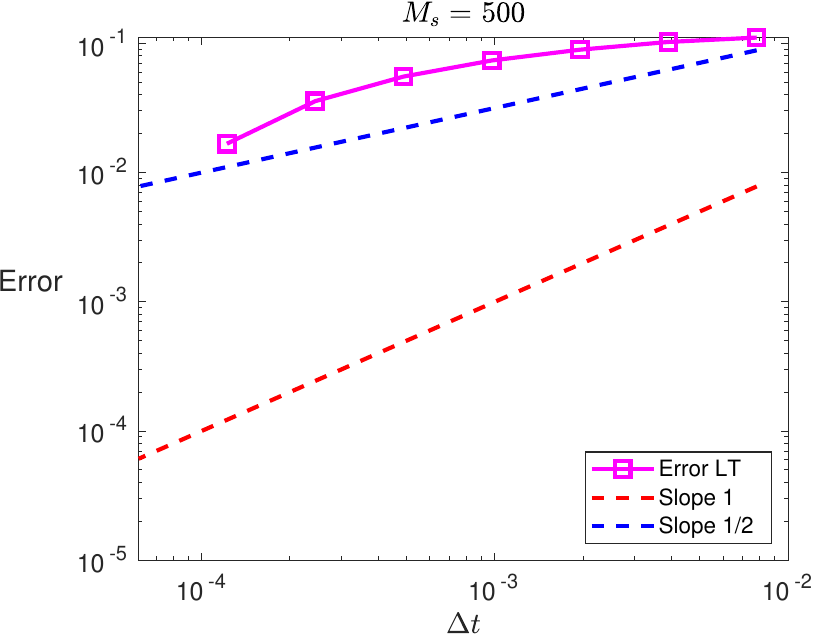}
\caption{Mean-square errors: Lie--Trotter scheme~\eqref{eq:LT-transport}
applied to the SPDE~\eqref{eq:SPDE-transport} driven by one noise ($K=1$, $\sigma_1(x,v)=0.5$).}
\label{fig:ms-transport}
\end{figure}

\section{Conclusion and perspectives}
In this paper, we numerically investigate the behaviour of solutions of linear Vlasov partial differential equations perturbed by Wiener noise, which may be additive noise, multiplicative It\^o and Stratonovich noise, and transport noise. For this purpose, we apply Lie--Trotter splitting integrators for the temporal discretization of these SPDEs.
We show that these time integrators are able to preserve key qualitative properties of the exact solutions, such as preserving norms or ensuring positivity. Several promising theoretical results are provided and illustrated numerically.

We leave some fundamental questions open for future work. Let us mention a few of them.
\begin{enumerate}
\item Based on the numerical experiments of this paper, we conjecture that the order of mean-square convergence of the Lie--Trotter splitting scheme is equal to $1$ for all the considered types of noise. Thus the main perspective for a possible future work is to prove mean-square error bounds.
\item We have only considered the behavior of the temporal discretization error. It will be interesting to study the full discretization error,
and in particular whether some conditions on the time-step size $\tau$ and the mesh sizes $\delta x$ and $\delta v$ need to be imposed for stability or accuracy reasons.
\item Constructing higher-order methods could also be an interesting question. For instance, in the context of SPDEs with small noise, it may be possible to apply a Strang splitting strategy to deal with the deterministic subsystems and obtain better rates of convergence.
\item We have only considered some simple linear Vlasov equations perturbed by additive or linear noise. It may be interesting to study more complex problems, for instance with nonlinear terms.
\end{enumerate}

\begin{appendix}

\section{The It\^o--Wentzell formula}\label{sec:ItoWentzell}

Let $D,K\in\N$ be two integers. Let $\beta_1,\ldots,\beta_K$ be independent standard real-valued Wiener processes, and let $b_j:\R^D\to \R$ and $\varsigma_{j,k}:\R^D\to \R$, $1\le j\le D$, $1\le k\le K$, be Lipschitz continuous mappings.

Assume that the $\R^D$-valued stochastic process $\bigl(\xi(t)\bigr)_{t\ge 0}$ is solution of the stochastic differential equation
\[
\text d\xi_j(t)=b_j(\xi(t))\,\text dt+\sum_{k=1}^{K}\varsigma_{j,k}(\xi(t))\,\text d\beta_k(t)~,\quad 1\le j\le D,~t\ge 0.
\]

Let $G:(t,\xi)\in\R^+\times \R^D\mapsto G(t,\xi)\in\R$ be a stochastic process, such that for all $\xi\in\R^D$ and $t\ge 0$ one has
\[
\text d G(t,\xi)=J(t,\xi)\,\text dt+\sum_{k=1}^{K}H_k(t,\xi)\,\text d\beta_k(t),
\]
where $J:(t,\xi)\in\R^+\times \R^D\mapsto J(t,\xi)\in\R$ and $H_k:(t,\xi)\in\R^+\times \R^D\mapsto H_k(t,\xi)\in\R$, $1\le k\le K$, are stochastic processes. Assume that $G$ is almost surely of class $\mathcal{C}^{0,2}$.

Then the real-valued stochastic processes $t\mapsto G(t,\xi(t))$ is solution of the stochastic differential equation
\begin{align*}
\text dG(t,\xi(t))&=J(t,\xi(t))\,\text dt+\sum_{k=1}^{K}H_k(t,\xi(t))\,\text d\beta_k(t)\\
&+\sum_{j=1}^{d}b_j(\xi(t))\partial_{\xi_j}G(t,\xi(t))\,\text dt+\sum_{j=1}^{d}\sum_{k=1}^{K}\varsigma_{j,k}(t,\xi(t))\partial_{x_j}G(t,\xi(t))\,\text d\beta_k(t)\\
&+\frac12\sum_{i,j=1}^{d}\sum_{k=1}^{K}\varsigma_{i,k}(t,\xi(t))\varsigma_{j,k}(t,\xi(t))\partial_{x_i}\partial_{x_j}G(t,\xi(t))\,\text dt\\
&+\sum_{j=1}^{d}\sum_{k=1}^{K}\varsigma_{j,k}(t,\xi(t))\partial_{\xi_j}H_k(t,\xi(t))\,\text dt.
\end{align*}

We refer to~\cite[Theorem~1.17]{MR3839316} for a statement of the formula above.

\end{appendix}

\section*{Acknowledgments}
We thank the referees and the associated editor for helpful comments and suggestions on the initial version of this work.

The numerical experiments have been performed by adapting the code from \cite{codesmatlab} to take into account the considered stochastic perturbations.

We thank Simone Calogero and Lukas Einkemmer for interesting discussions.
The work of CEB is partially supported by the project SIMALIN (ANR-19-CE40-0016) operated by the French National Research Agency.
The work of DC is partially supported by the Swedish Research Council (VR) (project no. 2018-04443).
The computations were performed on resources provided by the National Academic Infrastructure for Supercomputing in Sweden (NAISS) at UPPMAX, Uppsala University partially funded by the Swedish Research Council through grant agreement no. 2022-06725.

\providecommand{\href}[2]{#2}
\providecommand{\arxiv}[1]{\href{http://arxiv.org/abs/#1}{arXiv:#1}}
\providecommand{\url}[1]{\texttt{#1}}
\providecommand{\urlprefix}{URL }

\end{document}